\documentclass[11pt]{article}
\usepackage{fullpage}
\usepackage{amsmath,amsthm,amsfonts,amssymb,dsfont,bm}
\usepackage{enumerate,color,xcolor}
\usepackage{graphicx}
\usepackage[colorlinks,linkcolor=cyan,citecolor=blue]{hyperref}
\usepackage{url}
\usepackage{enumitem}

\numberwithin{equation}{section}

\theoremstyle{plain}
\newtheorem{theorem}{Theorem}[section]  
\newtheorem{corollary}[theorem]{Corollary}
\newtheorem{lemma}[theorem]{Lemma}
\newtheorem{proposition}[theorem]{Proposition}
\newtheorem{assumption}[theorem]{Assumption}

\theoremstyle{definition}
\newtheorem{definition}[theorem]{Definition}

\theoremstyle{remark}
\newtheorem{remark}[theorem]{Remark}

\newcommand{\E}{\mathbb{E}}
\newcommand{\R}{\mathbb{R}}
\newcommand{\cL}{\mathcal{L}}
\newcommand{\cW}{\mathcal{W}}

\newcommand{\ep}{\varepsilon}

\newcommand{\eps}{\epsilon}

\title{Sampling Non-Log-Concave Densities\\
  via Hessian-Free High-Resolution Dynamics}

\author{Xiaoyu Wang\thanks{FinTech Thrust,
  Hong Kong University of Science and Technology (Guangzhou), Guangzhou, Guangdong, People's Republic of China; 
  \texttt{xiaoyuwang@hkust-gz.edu.cn}}
  \and
  Yingli Wang\thanks{Corresponding author. School of Mathematics, Shanghai University of Finance and Economics, Shanghai, People's Republic of China; \texttt{naturesky1994@gmail.com}}
  \and
  Lingjiong Zhu\thanks{Department of Mathematics, Florida State University, Tallahassee, Florida, United States of America;
  \texttt{zhu@math.fsu.edu}}}

\date{\today}

\begin{document}
\maketitle

\begin{abstract}
We study the problem of sampling from a target distribution
$\pi(q)\propto e^{-U(q)}$ on $\mathbb{R}^d$, 
where $U$ can be non-convex, via the Hessian-free
high-resolution (HFHR) dynamics, which is a second-order Langevin-type process that has $e^{-U(q)-\frac12|p|^2}$ as its unique invariant distribution, and it reduces to
kinetic Langevin dynamics (KLD) as the resolution parameter $\alpha\to0$.
The existing theory for HFHR dynamics in the literature is restricted to strongly-convex $U$, although numerical experiments are promising
for non-convex settings as well. 
We focus on studying the convergence of HFHR dynamics when $U$ can be non-convex, which bridges
a gap between theory and practice.
Under a standard assumption of dissipativity and smoothness on $U$, we adopt
the reflection/synchronous coupling method. This yields a Lyapunov-weighted
Wasserstein distance in which the HFHR semigroup is exponentially
contractive for all sufficiently small $\alpha>0$ whenever KLD is. We further show that, under an additional assumption that asymptotically $\nabla U$ has linear growth at infinity, the contraction rate for HFHR dynamics is strictly better than that of KLD, with an explicit gain. As a case study, we verify the assumptions and the resulting acceleration for three examples: a multi-well potential, Bayesian linear regression with $L^p$ regularizer and Bayesian binary classification.
We conduct numerical experiments based
on these examples, as well as an additional example of Bayesian logistic regression with real data processed by the neural networks, which illustrates the efficiency of the algorithms based on HFHR dynamics and verifies the acceleration and superior performance compared to KLD.
\end{abstract}


\section{Introduction}

We consider the problem of sampling from a target distribution
\[
  \pi(q)\propto e^{-U(q)}, \qquad q\in\mathbb{R}^d,
\]
where $U:\R^d\to\R$ is a potential function. Such sampling problems arise
routinely in Bayesian statistics, inverse problems and modern machine
learning, e.g.\ posterior sampling for high--dimensional models and
Bayesian formulations of large--scale optimization
\cite{gelman1995bayesian,stuart2010inverse,andrieu2003introduction,teh2016consistency,DistMCMC19,GIWZ2024,DIGing2025}.

A classical approach is based on the \textit{overdamped Langevin dynamics} (OLD),
\begin{equation}\label{eq:overdamped-2}
  dq_t = -\nabla U(q_t)\,dt + \sqrt{2}\,dB_t,
\end{equation}
whose invariant distribution (under mild conditions) has density
$\pi(q)\propto e^{-U(q)}$; see e.g.\
\cite{chiang1987diffusion,stroock-langevin-spectrum}.
In practice, one can simulate \eqref{eq:overdamped-2} via the Euler--Maruyama
scheme
\begin{equation}\label{discrete:overdamped}
  q_{k+1}
  = q_k - \eta \nabla U(q_k) + \sqrt{2\eta}\,\xi_{k+1},
\end{equation}
often referred to as the unadjusted Langevin algorithm (ULA)
\cite{Dalalyan,DM2017,DM2016}, where $\xi_{k}$ are independent and identically distributed (i.i.d.) Gaussian random vectors $\mathcal{N}(0,I_{d})$.
Over the last decade, a sharp non--asymptotic theory has been developed
for \eqref{discrete:overdamped} in various distances (total variation,
Wasserstein, Kullback--Leibler, $\chi^2$, R\'enyi), and in settings with
stochastic gradients
\cite{Dalalyan,DM2017,DM2016,DK2017,Raginsky,Barkhagen2021,Chau2019,Zhang2019,CB2018,EHZ2022}.

To accelerate convergence, one can introduce a momentum variable and
consider the \textit{kinetic Langevin dynamics} (KLD) (also known as underdamped or second-order Langevin dynamics)
\cite{mattingly2002ergodicity,Villani2009,cheng2018underdamped,cheng-nonconvex,CLW2020,JianfengLu,dalalyan2018kinetic,GGZ2,Ma2019,GGZ}:
\begin{equation}\label{eqn:underdamped}
\begin{cases}
  dp_t = -\gamma p_t\,dt - \nabla U(q_t)\,dt + \sqrt{2\gamma}\,dB_t,\\[0.3em]
  dq_t = p_t\,dt,
\end{cases}
\end{equation}
where $(B_t)_{t\ge0}$ is a $d$--dimensional Brownian motion and
$\gamma>0$ is the friction parameter.
Under mild assumptions, \eqref{eqn:underdamped} admits a unique
invariant measure with density $\propto e^{-U(q)-\frac12|p|^2}$, whose
$q$--marginal coincides with $\pi$.
It is by now well--understood that, both at the continuous and discrete
levels, kinetic Langevin dynamics and its discretized algorithms can converge faster than the overdamped
counterpart, with improved dependence on the dimension $d$ and accuracy
$\eps$ \cite{Eberle,JianfengLu,cheng2018underdamped,GGZ}.

Kinetic Langevin dynamics is closely related to \textit{Nesterov's accelerated gradient} (NAG) method in optimization \cite{Nesterov1983,Nesterov2013,Ma2019,GGZ}. Motivated by the high--resolution
ordinary differential equation (ODE) viewpoint on NAG, \cite{li2022hessian} proposed the \emph{Hessian-free
high-resolution} (HFHR) dynamics, a $2d$--dimensional Langevin-type
dynamics with state $(q_t,p_t)\in\R^{2d}$:
\begin{align}
  d q_t &= \left(p_t - \alpha\nabla U(q_t)\right)\,dt
           + \sqrt{2\alpha}\,d B_t^q,\label{HFHR1}\\
  d p_t &= \left(-\gamma p_t - \nabla U(q_t)\right)\,dt
           + \sqrt{2\gamma}\,d B_t^p,\label{HFHR2}
\end{align}
where $B^q,B^p$ are independent $d$--dimensional Brownian motions and
$\alpha>0$ is a ``resolution'' parameter.
Formally, as $\alpha\to0$ the system \eqref{HFHR1}--\eqref{HFHR2} reduces
to \eqref{eqn:underdamped}, while for fixed $\alpha>0$ it preserves the
Gibbs measure with density $\propto e^{-U(q)-\frac12|p|^2}$
\cite{li2022hessian}.
The drift in \eqref{HFHR1}--\eqref{HFHR2} depends only on $\nabla U$ and
is therefore ``Hessian-free'', in contrast to other high--resolution
ODEs for NAG which involve $\nabla^2 U$; see e.g.\ \cite{Shi2022}.
Recent works have further exploited the connection to NAG method in optimization to design gradient-adjusted dynamics for accelerated sampling which includes HFHR dynamics as a special case \cite{zuo2025gradient}.

Numerical experiments in \cite{li2022hessian} show that HFHR dynamics can exhibit
substantial acceleration over kinetic Langevin dynamics on non-convex sampling
tasks. However, the available theory is essentially restricted to
strongly convex (log-concave) potentials \cite{li2022hessian}. The non-convex case is much
more delicate: when $U$ is non-convex, the Jacobian of the drift has
expanding directions and na\"{i}ve Lyapunov estimates may fail to control
the dynamics globally. At the same time, for kinetic Langevin dynamics
\eqref{eqn:underdamped} a sharp coupling-based theory is available in
non-convex landscapes thanks to the work of
\cite{eberle2016reflection,Eberle}, who constructed a reflection/synchronous
coupling and a weighted Wasserstein distance in which the Markov
semigroup is exponentially contractive.

This motivates the following questions:
\begin{itemize}
  \item[(Q1)] \emph{Can HFHR dynamics be shown to converge exponentially
    fast to equilibrium for non-log-concave targets, under the same type
    of conditions on $U$ that are used for kinetic Langevin dynamics?}
  \item[(Q2)] \emph{Does HFHR dynamics genuinely accelerate mixing, in the sense
    that its contraction rate in a suitable Wasserstein distance is
    strictly better than that of kinetic Langevin dynamics, at least for small
    $\alpha>0$?}
\end{itemize}

Our goal in this paper is to answer both questions within a unified
coupling framework. We adapt the reflection/synchronous coupling of
\cite{eberle2016reflection,Eberle} to HFHR dynamics and combine it with a
Lyapunov-weighted distance, in the spirit of \cite{Eberle}, to obtain non-asymptotic
global contractivity. The analysis reveals precisely how the additional
Hessian-free drift in \eqref{HFHR1}--\eqref{HFHR2} affects the Lyapunov
structure and the Wasserstein contraction rate.

\medskip



Our contributions can be summarized as follows.

\begin{itemize}[leftmargin=*]
  \item[(1)] \emph{Lyapunov structure and global contractivity for HFHR dynamics.}
    We first show that under smoothness and dissipativity assumptions on possibly non-convex $U$ (Assumption~\ref{assump:potential}), for all sufficiently small $\alpha>0$, the
    kinetic Langevin Lyapunov function $\mathcal V_0$ remains a
    Lyapunov function for the HFHR infinitesimal generator $\mathcal L_\alpha$ and
    hence already implies \textit{non-asymptotic exponential convergence} of HFHR dynamics as for kinetic Langevin dynamics under the
    same set of assumptions on $U$; see
    Proposition~\ref{prop:V0-drift-HFHR} and
    Corollary~\ref{cor:convergence-V0}.
    More generally, given any Lyapunov function $\mathcal V$ satisfying
    the drift condition \eqref{eq:generic-drift-again}, we adapt the
    \textit{reflection/synchronous coupling} of
    \cite{eberle2016reflection,Eberle} to HFHR dynamics and construct a
    Lyapunov-weighted semimetric $\rho_{\mathcal V}$ that combines a
    concave function of a phase-space distance with the weight
    $1+\mathcal V(z)+\mathcal V(z')$. We show that the associated weighted Wasserstein
    distance $\cW_{\rho_{\mathcal V}}$ contracts exponentially under
    the HFHR semigroup with an \textit{explicit} contraction rate $c(\lambda)>0$; see
    Theorem~\ref{thm:master-contraction}.

  \item[(2)] \emph{Refined Lyapunov function and quantitative acceleration.}
    We construct a novel refined Lyapunov function for HFHR dynamics of the form
    $\mathcal V_\alpha = \mathcal V_0 + \alpha \mathcal M$, where
    $\mathcal V_0$ is the kinetic Langevin Lyapunov function and $\mathcal M$ is a Hessian-free corrector.
    Under an additional assumption that asymptotically $\nabla U$ has linear growth at infinity
    (Assumption~\ref{assump:asymptotic-linear-drift}), we show that $\mathcal V_\alpha$ yields an improved drift rate
    $\lambda_\alpha \ge \lambda + \Theta(\alpha)$ (Proposition~\ref{prop:Valpha-drift}), where $\lambda_\alpha$ is the drift constant in the generator/Lyapunov inequality for HFHR dynamics with parameter $\alpha$, and $\lambda$ denotes the baseline ($\alpha=0$) constant; this is the $\lambda$ in Assumption~\ref{assump:potential}(iii) (the dissipativity condition).
    This translates into a \textit{strict improvement} in the contraction rate.
    Specifically, denoting by $c_0$ and $c_\alpha$ the contraction rates of kinetic Langevin dynamics and HFHR dynamics respectively, we prove that (Corollary~\ref{cor:global-acceleration}) for all sufficiently small $\alpha>0$ there exists an \textit{explicitly computable}
   $\kappa_{\mathrm{global}}>0$ such that
    \[
      c_\alpha \;\ge\; c_0 + \kappa_{\mathrm{global}}\,\alpha.
    \]
    Crucially, we show that this acceleration is robust: it holds regardless of whether the convergence bottleneck is determined by the Lyapunov branch (recurrence from infinity) or the metric branch (barrier crossing). This implies that HFHR dynamics achieves a \textit{strictly better} contraction rate than kinetic Langevin dynamics in a weighted Wasserstein distance $\cW_{\rho_{\mathcal V_\alpha}}$, and hence also in the standard $2$-Wasserstein distance $\mathcal W_2$ (Corollary~\ref{cor:W2-acceleration}).

  \item[(3)] \emph{Case study}.
    As concrete illustrations, we study three examples where potential $U$ is non-convex in general: a multi-well potential (Section~\ref{sec:case:multi}), Bayesian linear regressions with $L^{p}$ regularizer (Section~\ref{sec:case:L:p}) and Bayesian binary classification (Section~\ref{sec:case:classification}). For all these examples, we verify that both Assumptions~\ref{assump:potential} and \ref{assump:asymptotic-linear-drift} are satisfied. Therefore, all the previous theoretical results from Sections~\ref{sec:global-contractivity} and \ref{sec:acceleration} are applicable, which shows that HFHR dynamics achieves a \textit{strictly better} contraction rate than kinetic Langevin dynamics for all these examples.

  \item[(4)] We illustrate our theory by numerical experiments based
    on these examples that satisfy all the assumptions for our theoretical results. In particular, we conduct experiments for a multi-well potential (Section~\ref{sec:num:multi}), Bayesian linear regressions with $L^p$ regularizer with synthetic data (Section~\ref{sec:num:linear})
    and Bayesian binary classification with real data (Section~\ref{sec:num:classification}) using the algorithms based on the discretizations of HFHR dynamics and kinetic Langevin dynamics. 
    Our experiments show acceleration
    and superior performance of algorithms based on HFHR dynamics compared to kinetic Langevin dynamics, 
    validating our theoretical findings.
    In addition, we conduct experiments of Bayesian logistic regression with real data processed by the neural networks which may not satisfy the assumptions in our theory, but still shows excellent numerical performance (Section~\ref{sec:num:neural}). 
\end{itemize}

We emphasize that the additional structural assumption used to obtain
the improved contraction rate in (2)-(3) is only needed for the
\emph{acceleration} results. The basic exponential convergence of HFHR dynamics
in a weighted Wasserstein distance already follows, for a small resolution parameter $\alpha$,
under the same assumptions on the potential function $U$ as in the kinetic Langevin case.

\section{Preliminaries}
\label{sec:setup}

In this section, we summarize the precise stochastic dynamics that we
study, introduce its infinitesimal generator, and state the standing
assumption on the potential function $U$ under which all our results are derived.
Throughout the paper, we work in phase space $\R^{2d}$ with coordinates
$z=(q,p)$, where $q\in\R^d$ denotes the \textit{position} and $p\in\R^d$ the
\textit{momentum}.

\subsection{HFHR dynamics and infinitesimal generator}

We recall from \eqref{HFHR1}-\eqref{HFHR2} that the \textit{Hessian-free high-resolution} (HFHR) dynamics is defined by the
stochastic differential equation (SDE):
\begin{equation}\label{eq:HFHR-SDE}
\begin{aligned}
  dq_t &= \left(p_t - \alpha\nabla U(q_t)\right)\,dt
          + \sqrt{2\alpha}\,dB_t^q,\\
  dp_t &= \left(-\gamma p_t - \nabla U(q_t)\right)\,dt
          + \sqrt{2\gamma}\,dB_t^p,
\end{aligned}
\end{equation}
where $B^q$ and $B^p$ are independent standard Brownian motions in
$\R^d$, $\gamma>0$ is the \textit{friction} parameter and $\alpha>0$ is the
\textit{resolution} parameter. Formally, as $\alpha\to0$ the system reduces to
the kinetic Langevin dynamics \eqref{eqn:underdamped}, while for fixed $\alpha>0$ it preserves
the Gibbs measure with density $\propto e^{-U(q)-\frac12|p|^2}$.
The infinitesimal generator $\cL_\alpha$ of \eqref{eq:HFHR-SDE} acts on
$C^2$ test functions $\varphi:\R^{2d}\to\R$ as
\begin{equation}\label{eq:Lalpha-generator-def}
  \cL_\alpha\varphi(q,p)
  = (p-\alpha\nabla U(q))\cdot\nabla_q\varphi
    +(-\gamma p-\nabla U(q))\cdot\nabla_p\varphi
    +\alpha\Delta_q\varphi+\gamma\Delta_p\varphi.
\end{equation}
For perturbative arguments, it is convenient to decompose the drift
operator in \eqref{eq:Lalpha-generator-def} as
\begin{subequations}
\begin{align}
  \mathcal A_0 &:= p\cdot\nabla_q
    +(-\gamma p-\nabla U(q))\cdot\nabla_p, \label{eq:A0-def}\\
  \mathcal A'  &:= -\nabla U(q)\cdot\nabla_q, \label{eq:Aprime-def}
\end{align}
\end{subequations}
where $\mathcal A_0$ is the kinetic Langevin drift ($\alpha=0$) and
$\mathcal A'$ is the additional Hessian-free drift from the infinitesimal generator of the HFHR
dynamics. With the notation in \eqref{eq:A0-def}-\eqref{eq:Aprime-def}, we can re-write \eqref{eq:Lalpha-generator-def} as
\begin{equation}\label{eq:Lalpha-decomp}
  \cL_\alpha = \mathcal A_0 + \alpha \mathcal A' + \alpha\Delta_q + \gamma\Delta_p.
\end{equation}

\subsection{Assumptions on the potential}

We now state the main assumptions on the potential function $U$. 

\begin{assumption}\label{assump:potential}
There exist constants $L,A\in(0,\infty)$ and $\lambda\in(0,1/4]$ such
that $U$ satisfies:
\begin{enumerate}[label=(\roman*)]
  \item \textbf{Lower bound and regularity:} $U\in C^1(\R^d)$ and
    $U(q)\ge0$ for all $q\in\R^d$.
  \item \textbf{Lipschitz gradient:} $\nabla U$ is $L$-Lipschitz:
    \begin{equation}\label{eq:U-Lip}
      |\nabla U(q)-\nabla U(q')|
      \le L|q-q'|, \qquad q,q'\in\R^d.
    \end{equation}
  \item \textbf{Dissipativity:} $U$ satisfies the drift condition
    \begin{equation}\label{eq:U-drift}
      \frac{1}{2}q\cdot\nabla U(q)
      \;\ge\;
      \lambda\left(U(q) + \frac{\gamma^2}{4}|q|^2\right) - A,
      \qquad q\in\R^d.
    \end{equation}
\end{enumerate}
\end{assumption}

Assumption~\ref{assump:potential} is the same assumption that is used for kinetic
Langevin dynamics in \cite{Eberle} and, in particular, already implies
exponential convergence of the kinetic Langevin dynamics.
The lower bound
$U\ge0$ is imposed for convenience and could be relaxed to $U$ being
bounded from below.
Condition
\eqref{eq:U-Lip} is the  $L$-smoothness condition of $U$, which is standard in the Langevin literature \cite{DK2017,Raginsky,Eberle,dalalyan2018kinetic,GGZ,li2022hessian}.
Condition \eqref{eq:U-drift} is a dissipativity condition which controls the
growth of $U$ outside a compact set, which, together with its variants, are often assumed in the Langevin literature when the potential is non-convex \cite{Raginsky,Eberle,GGZ}.

\subsection{Kinetic Langevin Lyapunov function}

In this section, we review the Lyapunov function introduced for kinetic
Langevin dynamics in \cite{Eberle}. Define $\mathcal V_0:\R^{2d}\to\R$ by
\begin{equation}\label{eq:V0-general-quadratic}
  \mathcal V_0(q,p)
  := U(q)
     + \frac{\gamma^2}{4}\left(
         |q+\gamma^{-1}p|^2
         + |\gamma^{-1}p|^2
         - \lambda|q|^2
       \right),
\end{equation}
where $\lambda$ is as in Assumption~\ref{assump:potential}. 
Let $\mu_{\min}$ and $\mu_{\max}$ denote the smallest and largest eigenvalues of the symmetric matrix 
\begin{equation}\label{M:matrix}
M := \frac{1}{4}\begin{pmatrix} \gamma^2(1-\lambda) & \gamma \\ \gamma & 2 \end{pmatrix},
\end{equation}
such that
\begin{equation}\label{mu:min:max}
\begin{aligned}
&\mu_{\min}:=\frac{1}{8}\left(\gamma^2(1-\lambda) + 2 - \sqrt{(\gamma^2(1-\lambda) - 2)^2 + 4\gamma^2}\right),\\
&\mu_{\max}:=\frac{1}{8}\left(\gamma^2(1-\lambda) + 2 + \sqrt{(\gamma^2(1-\lambda) - 2)^2 + 4\gamma^2}\right).
\end{aligned}
\end{equation}
Since $\lambda \le 1/4$, we have $\det(M) = \frac{\gamma^2}{16}(1-2\lambda) > 0$, ensuring $\mu_{\min} > 0$.
Then, for all $(q,p)\in\R^{2d}$,
\begin{equation}\label{eq:V0-equivalent}
  c_1\left(1 + U(q) + |q|^2 + |p|^2\right)
  \le 1+\mathcal V_0(q,p)
  \le c_2\left(1 + U(q) + |q|^2 + |p|^2\right),
\end{equation}
holds with explicit constants 
\begin{equation}\label{eq:c1-c2-def}
c_1 := \min(1, \mu_{\min}), 
\qquad
c_2 := \max(1, \mu_{\max}). 
\end{equation}
Moreover, under Assumption~\ref{assump:potential}, $\mathcal V_0$ is a
Lyapunov function for the kinetic Langevin infinitesimal generator $\cL_0$:
\begin{equation}\label{eq:L0-V0-Lyapunov}
  \cL_0 \mathcal V_0(q,p)
  \;\le\; \gamma\left(d + A - \lambda\,\mathcal V_0(q,p)\right),
\end{equation}
where $\lambda$ and $A$ are the constants specified in Assumption~\ref{assump:potential}(iii); see \cite[Proposition~2.4]{Eberle}.
In particular, $\mathcal V_0$ already yields exponential convergence of
kinetic Langevin dynamics to equilibrium.

In the sequel, we will first show in Section~\ref{sec:global-contractivity}
that, for $\alpha$ small enough, the unimproved Lyapunov function
$\mathcal V_0$ still satisfies a drift condition for the HFHR infinitesimal generator
$\cL_\alpha$ and hence implies exponential convergence of the HFHR
dynamics. In Section~\ref{sec:acceleration}, we then construct an
improved Lyapunov function $\mathcal V_\alpha=\mathcal V_0+\alpha\mathcal M$
and, under an additional structural assumption on $U$, obtain an
\emph{improved} drift rate and contraction rate for HFHR dynamics.

\subsection{Baseline Lyapunov drift for HFHR dynamics}

We now record a simple perturbation result which shows that, for
$\alpha$ small enough, the kinetic Langevin Lyapunov function
$\mathcal V_0$ still satisfies a Lyapunov drift condition for the HFHR
infinitesimal generator $\cL_\alpha$.

\begin{proposition}[Baseline Lyapunov drift for HFHR dynamics]\label{prop:V0-drift-HFHR}
Suppose Assumption~\ref{assump:potential} holds and let $\mathcal V_0$ be defined as in \eqref{eq:V0-general-quadratic}.  
Then, for every $\alpha\ge0$, the HFHR infinitesimal generator $\cL_\alpha$ \eqref{eq:Lalpha-generator-def} satisfies the drift inequality
\begin{equation}\label{eq:Lalpha-V0-Lyapunov}
  \cL_\alpha \mathcal V_0(q,p)
  \;\le\;
  \gamma\left(d + A_\alpha - \hat\lambda_\alpha\,\mathcal V_0(q,p)\right),
  \qquad (q,p)\in\R^{2d},
\end{equation}
where 
\begin{equation}\label{eq:Aalpha-hatlambda-def}
  A_\alpha \;:=\; A + \frac{J_1}{\gamma}\,\alpha,
  \qquad
  \hat\lambda_\alpha \;:=\; \lambda - \frac{J_1}{\gamma}\,\alpha,
\end{equation}
$A$ and $\lambda$ are the constants from Assumption~\ref{assump:potential}(iii), and the constant $J_1$ can be chosen explicitly as
\begin{equation}\label{eq:K-explicit}
  J_1
  \;:=\;
  K_A + K_\Delta,
  \quad
  K_A \;:=\; \frac{1}{c_1}\left[
    \frac{\gamma^4}{4}(1-\lambda)^2 + \frac{\gamma^2}{4}
  \right],
  \quad
  K_\Delta \;:=\; Ld + \frac{\gamma^2}{2}d(1-\lambda),
\end{equation}
where $c_1 := \min(1, \mu_{\min})$, with $\mu_{\min}$ given explicitly in \eqref{mu:min:max}.
In particular,
\begin{equation}\label{eq:hatlambda-explicit}
  \hat\lambda_\alpha
  \;=\;
  \lambda
  - \frac{\alpha}{\gamma}\left\{  
      \frac{1}{c_1}\left[
        \frac{\gamma^4}{4}(1-\lambda)^2 + \frac{\gamma^2}{4}
      \right]
      + Ld + \frac{\gamma^2}{2}d(1-\lambda)
    \right\}.
\end{equation}
Consequently, if we choose
\begin{equation}\label{eq:alphastar-explicit}
  \alpha_0
  \;:=\;
  \frac{\gamma \lambda}{2}\left\{ 
      \frac{1}{c_1}\left[
        \frac{\gamma^4}{4}(1-\lambda)^2 + \frac{\gamma^2}{4}
      \right]
      + Ld + \frac{\gamma^2}{2}d(1-\lambda)
    \right\}^{-1},
\end{equation}
then $\hat\lambda_\alpha\ge \lambda/2>0$ for all $\alpha\in[0,\alpha_0]$.
\end{proposition}

\begin{proof}
    We provide the proof in Appendix~\ref{app:V0-drift-HFHR}.
\end{proof}


\section{Global Contractivity: A General Framework}
\label{sec:global-contractivity}

In this section, we establish a general framework for the geometric ergodicity of the HFHR dynamics. We first define the reflection--synchronous coupling and the associated transport semimetric. We then prove a ``Master Theorem'' which states that if \emph{any} Lyapunov function satisfies a drift condition with rate $\lambda>0$, the dynamics contracts with a specific rate $c(\lambda)>0$ that is explicitly computable. Finally, we apply this framework to the kinetic Langevin Lyapunov function $\mathcal V_0$ to obtain global contractivity for HFHR dynamics when the resolution parameter $\alpha$ is small.


\subsection{Coupling construction}\label{subsec:coupling}

We construct a coupling of two HFHR processes $(z_t)_{t\ge0} = (q_t,p_t)_{t\ge0}$ and
$(z'_t)_{t\ge0} = (q'_t,p'_t)_{t\ge0}$ driven by the same parameters $\alpha,\gamma>0$.
Let
\begin{equation}\label{eq:difference-of-coupling}
  \Delta q_t := q_t - q'_t, \qquad
  \Delta p_t := p_t - p'_t.
\end{equation}
Following \cite{eberle2016reflection,Eberle}, we define the effective velocity difference
\begin{equation}\label{eq:R-def}
  \mathbf R_t := \Delta q_t + \gamma^{-1}\Delta p_t.
\end{equation}
Let $e_t := \mathbf R_t/|\mathbf R_t|$ if $\mathbf R_t\neq0$ and fix some unit vector otherwise.
Denote by $\mathcal P_t := e_t e_t^\top$ the orthogonal projection onto the span of $e_t$.

The coupling is defined as follows: both copies solve the HFHR SDE~\eqref{eq:HFHR-SDE}, driven by Brownian motions
$(B^q,B^p)$ and $(B^{q'},B^{p'})$ satisfying
\begin{equation}\label{eq:coupling-noise}
  dB_t^{q'} = dB_t^q, \qquad
  dB_t^{p'} = (I_d - 2\chi(t)\mathcal P_t)\,dB_t^p,
\end{equation}
where $\chi(t)\in\{0,1\}$ is a control process which interpolates between \textit{reflection coupling} in the effective velocity
direction ($\chi(t)=1$) and \textit{synchronous coupling} ($\chi(t)=0$). The precise choice of $\chi(t)$, depending on the
current distance, will be specified in the proof of Proposition~\ref{prop:regional-contractivity}.

\paragraph{Cutoff family of couplings.}
For later use, we introduce an approximate sticky family of couplings indexed by a cutoff parameter $\xi>0$.
Define
$\chi_\xi(t) := \mathbf{1}_{\{|\mathbf R_t|\ge \xi\}}$,
and use $\chi(t)=\chi_\xi(t)$ in~\eqref{eq:coupling-noise}. The limiting sticky coupling is obtained by sending
$\xi\downarrow0$.


\subsection{Distance function and admissible Lyapunov functions}

Next, we define the underlying distance and the Lyapunov-weighted semimetric. Set
\begin{equation}\label{L:eff}
  L_{\mathrm{eff}}(\alpha) := (1+\alpha\gamma)L,
\end{equation}
and fix a slack parameter $\eta_0>0$. Define the metric weight
\begin{equation}\label{eq:theta-def}
  \theta := (1+\eta_0)\,L_{\mathrm{eff}}(\alpha)\,\gamma^{-2}.
\end{equation}
Then, for $z=(q,p)$ and $z'=(q',p')$, we define
\begin{equation}\label{eq:r-def-again}
  r(z,z') := \theta\,|q-q'| + \left| (q-q') + \gamma^{-1}(p-p') \right|,
\end{equation}
where throughout the paper $|\cdot|$ denotes the Euclidean norm on $\R^d$.

Next, we introduce the class of admissible Lyapunov functions.

\begin{definition}[Admissible Lyapunov function]
\label{def:admissible-lyapunov}
A function $\mathcal V: \R^{2d} \to [0, \infty)$ is said to be
\emph{$(\lambda, D)$-admissible} for the infinitesimal generator $\cL_\alpha$
if it is $C^2$ (or $C^1$ with locally Lipschitz derivatives) and satisfies the drift inequality
\begin{equation}\label{eq:generic-drift-again}
  \cL_\alpha \mathcal V(q,p)
  \;\le\; \gamma\left(d + D - \lambda \mathcal V(q,p)\right),
  \qquad \text{for a.e. } (q,p)\in\R^{2d},
\end{equation}
for some constants $\lambda>0$ and $D\in\R$.
\end{definition}

\begin{remark}
While the generator $\cL_\alpha$ defined in \eqref{eq:Lalpha-generator-def}
involves the Laplacian $\Delta_q$, strict $C^2$ regularity of $U$ is not required.
Under Assumption~\ref{assump:potential}, $\nabla U$ is Lipschitz continuous; by
Rademacher's theorem, the second derivatives of $U$ (and hence of $\mathcal V_{0}$
in \eqref{eq:V0-general-quadratic} and $\mathcal V_{\alpha}$ in \eqref{eq:Valpha-def})
exist almost everywhere and are essentially bounded. The drift inequality
\eqref{eq:generic-drift-again} should therefore be understood in the almost-everywhere sense.
\end{remark}

\paragraph{A gradient constant associated with $\mathcal V$.}
For Lyapunov functions of the form $\mathcal V=\mathcal V_0+\mathfrak Q$ where $\mathfrak Q(z):=\frac12 z^\top \mathsf A z$ 
with $\mathsf A:=
\begin{pmatrix}
\mathsf A_{qq} & \mathsf A_{qp}\\
\mathsf A_{pq} & \mathsf A_{pp}
\end{pmatrix}$ and $\mathsf A_{qp}=\mathsf A_{pq}^\top$,
define
\begin{align}\label{grad:const}
\bar C_{\mathcal V}
:= \max\left\{1,(2\theta)^{-1}\right\} + \frac{C_{\mathfrak Q}}{\gamma k_1},
\qquad
C_{\mathfrak Q}:=\|\mathsf A_{pp}\|_{\mathrm{op}}+\|\mathsf A_{pq}\|_{\mathrm{op}},
\end{align}
where $\theta$ is the metric weight in \eqref{eq:theta-def} and $k_1$ is from Lemma~\ref{lem:r-equivalent}.

\paragraph{Concave distance profile.}
Note that through $\bar C_{\mathcal V}$ the profile $f_\lambda$ depends on the chosen Lyapunov function $\mathcal V$; we suppress this dependence in the notation. Fix parameters $\eta_0>0$, $c>0$ and $\varepsilon>0$ (to be chosen later).
Following \cite{eberle2016reflection,Eberle}, we construct
a concave distance profile $f_\lambda:[0,\infty)\to[0,\infty)$ adapted to the
metric weight~$\theta$ in \eqref{eq:theta-def} as follows.
Let $R_1(\lambda)=R_1(\lambda;L_{\mathrm{eff}}(\alpha))>0$ be a cutoff radius to be specified
in the proof of Theorem~\ref{thm:master-contraction}. Define, for $s\ge0$,
\begin{equation}\label{eq:phi-def}
  \varphi_{\lambda}(s)
  := \exp\!\left(
      -\frac{1+\eta_0}{8}\,L_{\mathrm{eff}}(\alpha)\,s^2
      -\frac{\gamma^{2}}{2}\,\varepsilon\,\bar C_{\mathcal V}\,s^2
    \right),
\end{equation}
and
\begin{equation}\label{eq:Phi-def}
  \Phi_{\lambda}(s) := \int_{0}^{s}\varphi_{\lambda}(x)\,dx.
\end{equation}
Next define the auxiliary correction factor
\begin{equation}\label{eq:gint-def}
  g_{\lambda}(r)
  := 1-\frac{9}{4}\,c\,\gamma
      \int_{0}^{r}\Phi_{\lambda}(s)\,\left(\varphi_{\lambda}(s)\right)^{-1}\,ds,
\end{equation}
and finally set
\begin{equation}\label{eq:f-def}
  f_{\lambda}(r)
  := \int_{0}^{r\wedge R_{1}} \varphi_{\lambda}(s)\,g_{\lambda}(s)\,ds,
  \qquad r\ge0.
\end{equation}
In particular, $f_\lambda$ is increasing and concave on $[0,R_1(\lambda)]$, and it is
constant on $[R_1(\lambda),\infty)$. Moreover, for $r\in(0,R_1(\lambda))$ we have
$f_\lambda'(r)=\varphi_\lambda(r)g_\lambda(r)$ and
$f_\lambda''(r)=\varphi_\lambda'(r)g_\lambda(r)+\varphi_\lambda(r)g_\lambda'(r)$
almost everywhere. The choice \eqref{eq:phi-def} is designed so that, in the
short-distance regime where reflection coupling is active, the term involving
$f_\lambda''$ cancels against the ``bad'' drift contribution proportional to
$r f_\lambda'$ (including the cross-variation term controlled by $\bar C_{\mathcal V}$)
in the regional estimate for the semimetric drift.

\paragraph{Lyapunov-weighted semimetric.}
Given a $(\lambda,D)$-admissible Lyapunov function $\mathcal V$, we define the
Lyapunov-weighted semimetric as
\begin{equation}\label{eq:rho-V-def-again}
  \rho_{\mathcal V}(z,z')
  := f_\lambda\!\left(r(z,z')\right)
     \left(1+\varepsilon \mathcal V(z)+\varepsilon \mathcal V(z')\right),
  \qquad z,z'\in\R^{2d},
\end{equation}
where $r$ and $f_\lambda$ are defined in \eqref{eq:r-def-again} and
\eqref{eq:f-def}. The associated Wasserstein distance between probability
measures $\mu,\nu$ on $\R^{2d}$ is
\[
  \cW_{\rho_{\mathcal V}}(\mu,\nu)
  := \inf_{\Gamma\in\Pi(\mu,\nu)}
     \int_{\mathbb{R}^{2d}\times\mathbb{R}^{2d}} \rho_{\mathcal V}(z,z')\,\Gamma(dz,dz'),
\]
where $\Pi(\mu,\nu)$ is the set of couplings of $\mu$ and $\nu$.


\subsection{Semimartingale decomposition and regional analysis}
\label{subsec:regional}

We work with the coupled HFHR processes defined by the HFHR SDE
\eqref{eq:HFHR-SDE} and the coupled noises \eqref{eq:coupling-noise}. Define
\begin{equation}\label{Z:t:etc}
  Z_t := q_t-q_t',
  \qquad
  W_t := p_t-p_t',
  \qquad
  \mathbf R_t := Z_t + \gamma^{-1}W_t,
  \qquad
  r_t := r(z_t,z_t').
\end{equation}
Given a $(\lambda,D)$-admissible Lyapunov function $\mathcal V$ and the
corresponding profile $f_\lambda$ constructed in
\eqref{eq:phi-def}--\eqref{eq:f-def}, set
\begin{equation}\label{G:t:rho:t}
  G_t := 1 + \varepsilon \mathcal V(z_t) + \varepsilon \mathcal V(z_t'),
  \qquad
  \rho_t := f_\lambda(r_t)\,G_t .
\end{equation}

Before proceeding to the drift analysis of $\rho_t$, we record that the
underlying distance $r$ is equivalent to the Euclidean metric on phase
space.

\begin{lemma}[Equivalence of $r$ and the Euclidean distance]
\label{lem:r-equivalent}
Let $r$ be defined by \eqref{eq:r-def-again}, i.e.
\[
  r(z,z')
  = \theta |q-q'|
    + \left|(q-q')+\gamma^{-1}(p-p')\right|,
  \qquad z=(q,p),\; z'=(q',p')\in\R^{2d},
\]
with $\theta=(1+\eta_0)L_{\mathrm{eff}}(\alpha)\gamma^{-2}>0$. Then, for all
$z,z'\in\R^{2d}$,
\begin{equation}\label{eq:r-euclidean-equiv-app}
  k_1\,|z-z'|
  \;\le\; r(z,z')
  \;\le\; k_2\,|z-z'|,
\end{equation}
where the constants $k_1,k_2$ are explicitly given by
\begin{equation}\label{eq:explicit_k}
  k_1 := \frac{\theta}{1+\gamma(1+\theta)},
  \qquad
  k_2 := \sqrt{(\theta+1)^2 + \gamma^{-2}}.
\end{equation}
In particular, $r$ is equivalent to the Euclidean distance on $\R^{2d}$.
\end{lemma}

\begin{proof}
We provide the proof in Appendix~\ref{app:r-equivalent}.
\end{proof}

\medskip
The next step is to analyze the semimartingale drift of the Lyapunov-weighted distance process
$\rho_t=f_\lambda(r_t)G_t$. We will repeatedly use the fact that $r$ is comparable to the Euclidean metric,
and we will also need a mild smallness condition on $\alpha$ to preserve a strictly dissipative
coefficient in the $|Z_t|$-term of the distance drift.

\paragraph{A smallness condition on $\alpha$.}
Fix a parameter $\kappa_{\mathrm{adjust}}\in(0,1)$.
Due to the drift term $-\alpha\nabla U(q)$ in the HFHR SDE \eqref{eq:HFHR-SDE},
the dynamics of the difference $Z_t=q_t-q'_t$ involves the term $-\alpha(\nabla U(q_t)-\nabla U(q'_t))$.
This produces an extra contribution of size $\alpha L|Z_t|$ in the one-sided estimate for $d|Z_t|$.
When we translate this into a drift bound for the distance process $r_t$, this term reduces the baseline
dissipation coefficient $\eta_0/(1+\eta_0)$ that is present in the kinetic Langevin case.
We therefore introduce the \emph{net} dissipation parameter
\begin{equation}\label{eq:delta-alpha-def}
  \delta_\alpha:=\frac{\eta_0}{1+\eta_0}-\frac{\alpha L}{\gamma}.
\end{equation}
The condition $\delta_\alpha>0$ means that the additional HFHR drift does not overwhelm the baseline contraction,
so that the drift bounds for $r_t$ retain a strictly dissipative linear term in $|Z_t|$, uniformly in time, which is
needed to establish the regional contraction estimates.
Throughout the regional analysis, we assume
\begin{equation}\label{eq:alpha-small-drift}
  \alpha \le (1-\kappa_{\mathrm{adjust}})\,\frac{\eta_0}{1+\eta_0}\,\frac{\gamma}{L},
\end{equation}
so that
$\delta_\alpha \ge \kappa_{\mathrm{adjust}}\,\frac{\eta_0}{1+\eta_0} > 0$.
Accordingly, in the drift bounds for $r_t$ we use the dissipation coefficient $\delta_\alpha$
(or, when a uniform bound is convenient, $\kappa_{\mathrm{adjust}}\,\frac{\eta_0}{1+\eta_0}$).

\begin{remark}
In the kinetic Langevin case $\alpha=0$, we have $\delta_\alpha=\eta_0/(1+\eta_0)$.
Hence one may take $\kappa_{\mathrm{adjust}}\uparrow 1$ (and effectively $\kappa_{\mathrm{adjust}}=1$)
in the bounds, recovering the corresponding kinetic Langevin contraction rate without the extra prefactor.
\end{remark}

We next derive a drift decomposition for $e^{ct}\rho_t$ for any fixed $c\in\mathbb{R}$.
For Lyapunov functions $\mathcal V$ of the form $\mathcal V=\mathcal V_0+\mathfrak Q$
(with $\mathfrak Q$ quadratic), we identify the key drift coefficient that will be
estimated region by region.

\begin{lemma}[Drift decomposition]\label{lem:drift-decomp}
Recall the definition of $\delta_\alpha$ from \eqref{eq:delta-alpha-def} and assume that
\[
  \delta_\alpha=\frac{\eta_0}{1+\eta_0}-\frac{\alpha L}{\gamma}>0,
  \qquad\text{i.e.}\qquad
  \alpha < \frac{\eta_0}{1+\eta_0}\frac{\gamma}{L}.
\]
Fix $\varepsilon>0$ and $c\in\R$. Recall from \eqref{Z:t:etc} and \eqref{G:t:rho:t} that
\begin{align}
 & Z_t:=q_t-q'_t,\quad W_t:=p_t-p'_t,\quad \mathbf R_t:=Z_t+\gamma^{-1}W_t,\quad
  r_t:=\theta|Z_t|+|\mathbf R_t|, \label{defn:r:t}\\
 &G_t:=1+\varepsilon\mathcal V(z_t)+\varepsilon\mathcal V(z'_t),\quad
 \rho_t:=f_\lambda(r_t)G_t, \label{defn:G:t}
\end{align}
where $\theta:=(1+\eta_0)L_{\rm eff}(\alpha)\gamma^{-2}$ for some $\eta_0>0$ as in \eqref{eq:theta-def}.
Assume that $\mathcal V$ is a $(\lambda,D)$-admissible Lyapunov function 
and that
\begin{equation}\label{eq:V-structure-assump}
\mathcal V(z)=\mathcal V_0(z)+\mathfrak Q(z),
\end{equation}
where $\mathcal V_0$ is the kinetic Lyapunov function \eqref{eq:V0-general-quadratic} and
$\mathfrak Q$ is a quadratic form on $\R^{2d}$ (in $z=(q,p)$), i.e.
\[
\mathfrak Q(z):=\frac12\,z^\top \mathsf A z,\qquad \mathsf A\in\R^{2d\times 2d}\ \text{is symmetric}.
\]
Then
\[
  e^{ct}\rho_t \le \rho_0+\gamma\int_0^t e^{cs}K_s\,ds + M_t,
\]
where $M_t$ is a continuous local martingale and $K_t$ satisfies
\begin{align}
K_t &\le\;
  4\gamma^{-2}\left(\chi(t)\right)^2 f_\lambda''(r_t)G_t
  +\left(\theta|\mathbf R_t|-\delta_\alpha\,\theta|Z_t|\right)\, f'_{\lambda,-}(r_t)\,G_t \nonumber\\
&\quad+ 4\varepsilon \bar{C}_{\mathcal{V}} \,\left(\chi(t)\right)^2\, r_t f'_{\lambda,-}(r_t) +\gamma^{-1}\varepsilon f_\lambda(r_t)\left[\cL_\alpha\mathcal V(z_t)+\cL_\alpha\mathcal V(z'_t)\right]
+\gamma^{-1}c f_\lambda(r_t)G_t,\label{K:t:ineq}
\end{align}
with $\bar{C}_{\mathcal{V}}$ defined in \eqref{grad:const} 
and $k_1$ is the norm-equivalence constant from Lemma~\ref{lem:r-equivalent}.
\end{lemma}

\begin{proof}
We provide the proof in Appendix~\ref{app:drift-decomp}.
\end{proof}

\begin{remark}\label{rem:Q-structure}
The structural assumption \eqref{eq:V-structure-assump} is tailored to the analysis of HFHR dynamics. 
In Section~\ref{sec:acceleration}, we will construct a refined Lyapunov function $\mathcal{V}_\alpha = \mathcal{V}_0 + \alpha\mathcal{M}$, corresponding to the choice $\mathfrak{Q}(z) := \alpha\mathcal{M}(z)$. 
Since $\mathcal{M}$ will be constructed as a quadratic polynomial (see Lemma~\ref{lem:first-order-improvement}), its gradient is linear, and thus the Lipschitz condition on $\nabla_p \mathfrak{Q}$ is automatically satisfied globally. 
For the baseline convergence result (Corollary~\ref{cor:convergence-V0}), we simply take $\mathfrak{Q} \equiv 0$ (so that $C_{\mathfrak{Q}}=0$).
\end{remark}

\medskip

The next step is to estimate $K_t$ in different regions of the state space.
We distinguish small and intermediate distances, where the concavity of
$f_\lambda$ and reflection coupling dominate, from large distances, where
the Lyapunov drift of $\mathcal V$ takes over.

\begin{proposition}[Regional contractivity]\label{prop:regional-contractivity}
Assume Assumption~\ref{assump:potential} holds.
Fix $\eta_0>0$ and $\kappa_{\mathrm{adjust}}\in(0,1)$.
Assume the smallness condition \eqref{eq:alpha-small-drift} holds:
\begin{equation}\label{eq:alpha-small-drift-prop}
  \alpha \le (1-\kappa_{\mathrm{adjust}})\,\frac{\eta_0}{1+\eta_0}\,\frac{\gamma}{L},
\end{equation}
so that $\delta_\alpha\ge \kappa_{\mathrm{adjust}}\frac{\eta_0}{1+\eta_0}>0$, where $\delta_\alpha$ is defined in \eqref{eq:delta-alpha-def}.
Let $\mathcal V$ be a $(\lambda,D)$-admissible Lyapunov function for $\cL_\alpha$
in the sense of Definition~\ref{def:admissible-lyapunov}. Assume in addition that
$\mathcal V$ is coercive: there exist constants $a_{\mathcal V}>0$ and
$b_{\mathcal V}\ge0$ such that
\begin{equation}\label{eq:V-coercive}
  \mathcal V(z)\;\ge\; a_{\mathcal V}|z|^2-b_{\mathcal V},
  \qquad z\in\R^{2d}.
\end{equation}
Set $\theta$ as in \eqref{eq:theta-def} and $r(z,z')$ as in \eqref{eq:r-def-again}.
Let $f_\lambda$ be defined by \eqref{eq:phi-def}--\eqref{eq:f-def}
with some cutoff radius $R_1(\lambda)=R_1(\lambda;L_{\mathrm{eff}}(\alpha))>0$.

For each $\xi>0$, consider the cutoff coupling obtained by choosing
$\chi(t)=\chi_\xi(t)=\mathbf 1_{\{|\mathbf R_t|\ge\xi\}}$ in \eqref{eq:coupling-noise}.
Let $\left(z_t^\xi,z_t^{\prime,\xi}\right)$ be the resulting coupled processes and define
\[
  G_t^\xi := 1+\varepsilon\mathcal V\left(z_t^\xi\right)+\varepsilon\mathcal V\left(z_t^{\prime,\xi}\right),
  \qquad
  \rho_t^\xi := f_\lambda\left(r\left(z_t^\xi,z_t^{\prime,\xi}\right)\right)\,G_t^\xi .
\]
Then there exist constants $c_0,\varepsilon_0>0$ and $C_{\mathrm{reg}}<\infty$
(depending only on $\lambda,D,\eta_0$, $\gamma$, $L_{\mathrm{eff}}(\alpha)$,
$\bar C_{\mathcal V}$, and the construction of $f_\lambda$, but independent of $\xi$) such that for any
$0<c\le c_0$ and $0<\varepsilon\le\varepsilon_0$, the drift coefficient
$K_t^\xi$ from Lemma~\ref{lem:drift-decomp} (applied to $\rho_t^\xi$) satisfies
\begin{equation}\label{eq:Kt-xi-bound}
  K_t^\xi \;\le\; C_{\mathrm{reg}}\,\xi\,G_t^\xi,
  \qquad t\ge0.
\end{equation}
Consequently, for every $t\ge0$,
\begin{equation}\label{eq:rho-exp-bound}
  \E\!\left[e^{ct}\rho_t^\xi\right]
  \;\le\; \E[\rho_0] + \gamma C_{\mathrm{reg}}\xi \int_0^t e^{cs}\E\left[G_s^\xi\right]\,ds,
\end{equation}
and hence
$\limsup_{\xi\downarrow0}\E\!\left[e^{ct}\rho_t^\xi\right] \le \E[\rho_0]$.
In particular, any limiting (``sticky'') coupling obtained along $\xi\downarrow0$
is contractive in expectation with rate $c$.
\end{proposition}

\begin{proof}
We provide the proof in Appendix~\ref{app:regional-contractivity}.
\end{proof}


\subsection{Master theorem on global contraction}

We denote by $(P_t^\alpha)_{t\ge0}$ the Markov semigroup associated with the HFHR dynamics \eqref{eq:HFHR-SDE}.
We have the following Master Theorem that shows the contraction of HFHR dynamics with an explicitly computable contraction rate.

\begin{theorem}[Master theorem on global contraction]\label{thm:master-contraction}
Assume Assumption~\ref{assump:potential} holds.
Fix $\kappa_{\mathrm{adjust}}\in(0,1)$ and assume the smallness condition
\eqref{eq:alpha-small-drift} holds, 
so that with $\delta_\alpha$ defined in \eqref{eq:delta-alpha-def} we have
$\delta_\alpha\ge \kappa_{\mathrm{adjust}}\frac{\eta_0}{1+\eta_0}$.

Let $\mathcal V$ be a $(\lambda,D)$-admissible Lyapunov function 
for the HFHR generator $\cL_\alpha$
in the sense of Definition~\ref{def:admissible-lyapunov}, with $\lambda\in(0,1/4]$ and $d+D>0$.
Assume moreover that $\mathcal V$ is coercive in the sense of Proposition~\ref{prop:regional-contractivity}, i.e. \eqref{eq:V-coercive} holds.

Let $\eta_0>0$ be a parameter (chosen explicitly below in \eqref{eq:Lambda0-eta0-def}) 
and recall the definitions of $\theta$ as in \eqref{eq:theta-def} and $r(z,z')$ as in \eqref{eq:r-def-again}.
Let $R_1(\lambda)>0$ satisfy
\begin{equation}\label{eq:R1-condition-v2}
  R_1^2(\lambda) \ge \frac{96(d+A)}{5\lambda(1-2\lambda)\gamma^2((1+\theta)^2+\gamma^{-2})}.
\end{equation}
Let $f_\lambda$ be defined by \eqref{eq:phi-def}--\eqref{eq:f-def} with this $R_1(\lambda)$, and let $\rho_{\mathcal{V}}(z,z')$ be the
Lyapunov-weighted semimetric defined in \eqref{eq:rho-V-def-again}:
\begin{equation}\label{eq:rhoV-master}
\rho_{\mathcal V}(z,z')
:= f_\lambda(r(z,z'))\left(1+\varepsilon\mathcal V(z)+\varepsilon\mathcal V(z')\right).
\end{equation}
Define
\begin{equation}\label{eq:Lambda0-eta0-def}
  \Lambda_0(\lambda):=\frac{L\,R_1^2(\lambda)}{8},
  \qquad
  \eta_0:=\Lambda_0^{-1}(\lambda).
\end{equation}
Recall that $L_{\mathrm{eff}}(\alpha)=(1+\alpha\gamma)L$ and set
\begin{equation}\label{eq:lambda-alpha-def}
\Lambda_\alpha(\lambda):=\frac{L_{\mathrm{eff}}(\alpha)\,R_1^2(\lambda)}{8}.
\end{equation}

Choose
\begin{equation}\label{eq:rate-explicit}
  c(\lambda)
  := \frac{\gamma}{384}
    \min\left\{\widetilde\Lambda_{1,\alpha}(\lambda),\,
               \widetilde\Lambda_{2,\alpha}(\lambda),\,
               \widetilde\Lambda_{3,\alpha}(\lambda)\right\},
\end{equation}
where
\begin{align}
  &\widetilde\Lambda_{1,\alpha}(\lambda) := \frac{\lambda L_{\mathrm{eff}}(\alpha)}{\gamma^{2}},\quad
  \widetilde\Lambda_{2,\alpha}(\lambda) := \Lambda_\alpha^{1/2}(\lambda)\, e^{-\Lambda_\alpha(\lambda)}\,\frac{L_{\mathrm{eff}}(\alpha)}{\gamma^{2}},\nonumber
  \\
  &\widetilde\Lambda_{3,\alpha}(\lambda) := \kappa_{\mathrm{adjust}}\,\Lambda_\alpha^{1/2}(\lambda)\, e^{-\Lambda_\alpha(\lambda)}.
  \label{eq:Lambda-branches}
\end{align}
Then the HFHR semigroup $P_t^\alpha$ is exponentially contractive in the weighted Wasserstein distance $\cW_{\rho_{\mathcal V}}$. Specifically, let $c(\lambda)$ be defined by \eqref{eq:rate-explicit}. Fix any $c\in(0,c(\lambda)]$ and set
\begin{equation}\label{eq:epsilon-def}
  \varepsilon:=\frac{4c}{\gamma(d+D)}.
\end{equation}
Then
\begin{equation}\label{eq:W-rho-contraction}
  \cW_{\rho_{\mathcal V}}(\mu P_t^\alpha,\nu P_t^\alpha)
  \le e^{-ct}\,\cW_{\rho_{\mathcal V}}(\mu,\nu),\qquad t\ge0,
\end{equation}
for all probability measures $\mu,\nu$ on $\R^{2d}$ with finite $\mathcal V$-moments. In particular, choosing $c=c(\lambda)$ yields an explicit admissible contraction rate.
\end{theorem}

\begin{proof}
We provide the proof in Appendix~\ref{app:master-contraction}.
\end{proof}

In particular, when $\alpha=0$ we have $L_{\mathrm{eff}}(0)=L$, and the rate \eqref{eq:rate-explicit} reduces to the expression obtained in~\cite[Theorem~2.3]{Eberle}.
Moreover, the explicit form of $c(\lambda)$ in~\eqref{eq:rate-explicit} already anticipates the two mechanisms that will later be separated into the \emph{Lyapunov branch} (denoted $\widetilde\Lambda_{1,\alpha}$) and the \emph{metric branches} (denoted $\widetilde\Lambda_{2,\alpha},\widetilde\Lambda_{3,\alpha}$).
Roughly speaking, the Lyapunov branch corresponds to the contribution of the Lyapunov drift (via the $(\lambda,D)$-admissibility of $\mathcal V$) which controls excursions to large energies and yields contraction once the process is sufficiently far out, while the metric branches correspond to the local contraction mechanism encoded in the semimetric $\rho_{\mathcal V}$ (through the concavity/flatness design of $f_\lambda$ and reflection vs.\ synchronous coupling) and dominate in the ``nearby'' regime.
Thus, $c(\lambda)$ can be interpreted as the effective global contraction rate obtained by balancing these two effects: it is the rate for which both the Lyapunov drift and the local metric estimates close simultaneously.

\subsection{Global convergence of HFHR dynamics}

We now apply Theorem~\ref{thm:master-contraction} with the kinetic Langevin Lyapunov function $\mathcal V_0$ defined in Section~\ref{sec:setup}. 
Recall from Proposition~\ref{prop:V0-drift-HFHR} that $\mathcal V_0$ satisfies the required drift condition for the HFHR infinitesimal generator $\cL_\alpha$ provided $\alpha \le \alpha_0$. This allows us to specialize the general master theorem (Theorem~\ref{thm:master-contraction}) to the baseline Lyapunov function, yielding the following exponential convergence result.

\begin{corollary}[Global convergence of HFHR dynamics]\label{cor:convergence-V0}
Assume Assumption~\ref{assump:potential}.
Let $\alpha\in[0,\alpha_0]$, where $\alpha_0$ is as in Proposition~\ref{prop:V0-drift-HFHR}.
Fix $\kappa_{\mathrm{adjust}}\in(0,1)$ and assume the smallness condition
\eqref{eq:alpha-small-drift} holds with $\eta_0:=\left(\Lambda_0(\lambda_\alpha)\right)^{-1}$, where $\Lambda_0(\cdot)$ is defined in \eqref{eq:Lambda0-eta0-def}.
Set $\lambda_\alpha:=\hat\lambda_\alpha>0$,
with $\hat\lambda_\alpha$ given in \eqref{eq:hatlambda-explicit}.
Let $R_1(\lambda)>0$ satisfy \eqref{eq:R1-condition-v2} 
with $\lambda=\lambda_\alpha$ and $A=A_\alpha$, note that $\eta_0=\left(\Lambda_0(\lambda)\right)^{-1}$.
Let $\theta$ and $r$ be defined by \eqref{eq:theta-def}--\eqref{eq:r-def-again},
let $f_\lambda$ be defined by \eqref{eq:phi-def}--\eqref{eq:f-def} with this $R_1(\lambda)$,
and choose any 
$0<c_\alpha\le c_*(\lambda_\alpha,A_\alpha)$ and
$0<\varepsilon_\alpha\le \varepsilon_*(\lambda_\alpha,A_\alpha)$
as provided by Theorem~\ref{thm:master-contraction} (applied with
$\mathcal V=\mathcal V_0$ and $(\lambda,A)=(\lambda_\alpha,A_\alpha)$).
Let $\rho_{\mathcal V_0,\alpha}$ denote the corresponding Lyapunov-weighted semimetric
\eqref{eq:rhoV-master}.

Then the HFHR dynamics \eqref{eq:HFHR-SDE} admits a unique invariant probability
measure $\pi_\alpha$ in the class
$\{\mu:\int_{\R^{2d}}\mathcal V_0\,d\mu<\infty\}$, and for all probability measures
$\mu,\nu$ with $\int_{\R^{2d}} \mathcal V_0\,d\mu+\int_{\R^{2d}} \mathcal V_0\,d\nu<\infty$,
\[
  \cW_{\rho_{\mathcal V_0,\alpha}}(\mu P_t^\alpha,\nu P_t^\alpha)
  \le e^{-c_\alpha t}\,\cW_{\rho_{\mathcal V_0,\alpha}}(\mu,\nu),
  \qquad t\ge0.
\]
In particular,
\[
  \cW_{\rho_{\mathcal V_0,\alpha}}(\mu P_t^\alpha,\pi_\alpha)
  \le e^{-c_\alpha t}\,\cW_{\rho_{\mathcal V_0,\alpha}}(\mu,\pi_\alpha),
  \qquad t\ge0.
\]
Moreover, an explicit admissible choice of $c_\alpha$ is given by \eqref{eq:rate-explicit}
in Theorem~\ref{thm:master-contraction} with $(\lambda,A)=(\lambda_\alpha,A_\alpha)$.
\end{corollary}

\begin{proof}
    We provide the proof in Appendix~\ref{app:convergence-V0}.
\end{proof}

\medskip

However, we observe that the drift rate $\hat\lambda_\alpha$ in \eqref{eq:hatlambda-explicit} might be smaller than the baseline rate $\lambda$ due to the perturbative treatment of the Hessian-free drift. Consequently, the resulting contraction rate $c(\hat\lambda_\alpha)$ does not yet exhibit acceleration over kinetic Langevin dynamics.
In Section~\ref{sec:acceleration}, we will apply the same abstract
framework with the improved Lyapunov function
$\mathcal V_\alpha=\mathcal V_0+\alpha\mathcal M$ in place of
$\mathcal V_0$ in order to obtain improved contraction rates.

\subsection{From the weighted Wasserstein distance to the 2-Wasserstein distance}
\label{subsec:W2}

In this subsection, we explain how to pass from a contraction in the
weighted Wasserstein distance $\mathcal W_{\rho_{\mathcal V}}$ to a
contraction in the standard Wasserstein distance $\mathcal W_2$.
Recall that the master contraction theorem (Theorem~\ref{thm:master-contraction}) is stated in terms of the weighted cost
$\rho_{\mathcal V}$ built from a Lyapunov function $\mathcal V$.
In our applications we will use two choices:
(i) the \emph{baseline} Lyapunov function $\mathcal V=\mathcal V_0$, for which we can give explicit quadratic bounds,
and (ii) a more general $(\lambda,D)$-Lyapunov function $\mathcal V$ appearing in Theorem~\ref{thm:master-contraction}.
The passage from the weighted Wasserstein distance $\mathcal W_{\rho_{\mathcal V}}$ to the standard 2-Wasserstein distance $\mathcal W_2$ only needs that $\mathcal V$ controls second moments
(via a quadratic lower bound), while $(\lambda,D)$-admissibility will be used later only to obtain uniform moment bounds for $\E[\mathcal V(z_t)]$.
To make constants explicit, we first record the quadratic bounds for the canonical choice $\mathcal V_0$.
First, we show that $\mathcal{V}_{0}(q,p)$ can be lower and upper bounded
by quadratic functions with explicit coefficients.
More precisely, we have the following lemma.

\begin{lemma}\label{lem:quadratic:bounds}
For any $(q,p)\in\R^{2d}$, 
\begin{equation}\label{eq:V0-second-moment}
  c_1'\left(1+|q|^2+|p|^2\right)
  \le 1+\mathcal V_0(q,p)
  \le c_2'\left(1+|q|^2+|p|^2\right),
\end{equation}
where
\begin{align}
c_1' := \min(1, \mu_{\min}),
\qquad
c_2' := \max(1, \mu_{\max}) + U(0)+\frac{L}{2}+\frac{1}{2}|\nabla U(0)|,
\end{align}
where $\mu_{\min}, \mu_{\max}$ are the eigenvalues of $M$ (defined in \eqref{M:matrix}) with explicit formulas given in \eqref{mu:min:max}. 
\end{lemma}

\begin{proof}
    We provide the proof in Appendix~\ref{proof:lem:quadratic:bounds}.
\end{proof}

In particular, \eqref{eq:V0-second-moment} implies that $\mathcal V_0$ controls the Euclidean second moment on $\R^{2d}$, which is exactly what is needed to compare the weighted Wasserstein distance $\mathcal W_{\rho_{\mathcal V}}$ with the standard $2$-Wasserstein distance $\mathcal W_2$ (we will apply this with $\mathcal V=\mathcal V_0$ below, and more generally with any $\mathcal V$ satisfying a quadratic lower bound).
Moreover, the underlying phase-space distance $r$ defined in
\eqref{eq:r-def-again} is equivalent to the Euclidean distance on $\R^{2d}$:
\begin{equation}\label{eq:r-euclidean-equiv}
  k_1\,|(q,p)-(q',p')|
  \;\le\; r\left((q,p),(q',p')\right)
  \;\le\; k_2\,|(q,p)-(q',p')|,
  \qquad (q,p),(q',p')\in\R^{2d},
\end{equation}
where $k_1,k_2>0$ are explicit constants depending only on $\theta$ and $\gamma$ (see \eqref{eq:explicit_k}).
The first inequality in \eqref{eq:r-euclidean-equiv} is a straightforward consequence of the definition
\eqref{eq:r-def-again}, while the second inequality in \eqref{eq:r-euclidean-equiv} follows from the fact that $r$
is equivalent to the Euclidean distance; see Lemma~\ref{lem:r-equivalent}.

The next lemma quantifies how $\rho_{\mathcal V}$ controls the quadratic
transport cost.

\begin{lemma}\label{lem:rho-controls-W2}
Assume Assumption~\ref{assump:potential} holds. Fix $\eta_0>0$ and let
$f_\lambda$ be defined by \eqref{eq:phi-def}--\eqref{eq:f-def}
with cutoff radius $R_1(\lambda)>0$.  Assume moreover that
\footnote{This positivity condition is ensured, for instance, by the explicit parameter
choice in \eqref{eq:rate-explicit}; see the verification in the proof of
Theorem~\ref{thm:master-contraction}, Region~II, where we show $g_*\ge 1/2$.}
\begin{equation}\label{eq:g-positive}
  g_*(\lambda) := \inf_{0\le s\le R_1(\lambda)} g_\lambda(s) \;>\;0.
\end{equation}
Let $\mathcal V:\R^{2d}\to[1,\infty)$ satisfy the quadratic lower bound:
\begin{equation}\label{eq:V-quadratic-lower-new}
  |z|^2 \le C_V\left(1+\mathcal V(z)\right),\qquad z\in\R^{2d},
\end{equation}
for some constant $C_V\in(0,\infty)$.  For $\varepsilon\in(0,1]$, let $\rho_{\mathcal{V}}(z,z')$ be the
Lyapunov-weighted semimetric defined in \eqref{eq:rho-V-def-again}:
\[
  \rho_{\mathcal V}(z,z')
  := f_\lambda(r(z,z'))\left(1+\varepsilon \mathcal V(z)
                                 +\varepsilon \mathcal V(z')\right).
\]
Then there exists $C_\rho<\infty$ such that for all probability measures
$\mu,\nu$ on $\R^{2d}$ with $\int_{\mathbb{R}^{2d}} \mathcal V\,d\mu+\int_{\mathbb{R}^{2d}} \mathcal V\,d\nu<\infty$,
\[
  \mathcal W_2^2(\mu,\nu)
  \le C_\rho\, \mathcal W_{\rho_{\mathcal V}}(\mu,\nu).
\]
More explicitly, one may take
\begin{equation}\label{eq:Crho-explicit}
  C_\rho
  := \frac1\varepsilon
    \max\left\{
      \frac{k_1^{-2}R_1(\lambda)}{g_*\,c_r},\,
      \frac{4C_V}{c_0}
    \right\},
\end{equation}
where
\[
  c_r := \inf_{0\le s\le R_1(\lambda)}\varphi_\lambda(s),\qquad
  c_0 := f_\lambda(R_1(\lambda))=\int_{0}^{R_1(\lambda)}\varphi_\lambda(s)g_\lambda(s)\,ds,
\]
and $k_1=\frac{\theta}{1+\gamma(1+\theta)}$ is from Lemma~\ref{lem:r-equivalent}.
In particular, when $\mathcal V=\mathcal V_0$, Lemma~\ref{lem:quadratic:bounds}
implies \eqref{eq:V-quadratic-lower-new} with $C_V=1/c_1'$.
\end{lemma}

\begin{proof}
    We provide the proof in Appendix~\ref{app:rho-controls-W2}.
\end{proof}

Combining Lemma~\ref{lem:rho-controls-W2} with the contraction
\eqref{eq:W-rho-contraction}, we obtain the following corollary, which
upgrades exponential contraction in $\cW_{\rho_{\mathcal V}}$ to an
estimate in the standard $2$-Wasserstein distance $\cW_2$.

\begin{corollary}[Exponential contraction in $\mathcal W_2$]\label{cor:W2-contraction}
Under the assumptions of Theorem~\ref{thm:master-contraction}, let $c>0$ and
$\varepsilon=\frac{4c}{\gamma(d+D)}$ be as in \eqref{eq:epsilon-def}, 
and let $\rho_{\mathcal V}$ be defined in \eqref{eq:rhoV-master}.  Let $C_{\rho}<\infty$
be the constant from Lemma~\ref{lem:rho-controls-W2} (computed with this $\varepsilon$).
Then, for all probability measures $\mu,\nu$ such that
$\int_{\R^{2d}}\mathcal V\,d\mu+\int_{\R^{2d}}\mathcal V\,d\nu<\infty$ and all $t\ge0$,
\[
  \mathcal W_2(\mu P_t^\alpha,\nu P_t^\alpha)
  \;\le\;
  C_{\rho}^{1/2}\,e^{-ct/2}
  \left(\cW_{\rho_{\mathcal V}}(\mu,\nu)\right)^{1/2}.
\]
In particular, if $\pi_\alpha$ is an invariant probability measure with
$\int_{\R^{2d}}\mathcal V\,d\pi_\alpha<\infty$, then taking $\nu=\pi_\alpha$
yields exponential convergence of $\mu P_t^\alpha$ to $\pi_\alpha$ in $\mathcal W_2$.
\end{corollary}

\begin{proof}
    We provide the proof in Appendix~\ref{app:W2-contraction}.
\end{proof}

By combining Corollary~\ref{cor:W2-contraction} with the baseline Lyapunov function $\mathcal V_0$ and the existence and uniqueness of the invariant measure from Corollary~\ref{cor:convergence-V0}, we obtain the following baseline $\mathcal W_2$ convergence estimate.

\begin{corollary}[Baseline exponential convergence in $\mathcal W_2$]
\label{cor:W2-convergence-baseline}
Under the assumptions of Corollary~\ref{cor:convergence-V0}, apply
Theorem~\ref{thm:master-contraction} with $\mathcal V=\mathcal V_0$ and
$(\lambda,D)=(\hat\lambda_\alpha,A_\alpha)$. Let $c_\alpha>0$ and
$\varepsilon_\alpha=\frac{4c_\alpha}{\gamma(d+A_\alpha)}$ be the resulting parameters
(one admissible explicit choice of $c_\alpha$ is given by \eqref{eq:rate-explicit}
with $\lambda=\hat\lambda_\alpha$ and $D=A_\alpha$), and let
$\rho_{\mathcal V_0,\alpha}$ be the corresponding weighted semimetric.
Let $C_{\rho,\alpha}$ be the constant from Lemma~\ref{lem:rho-controls-W2} associated
with $\mathcal V_0$ and computed with $\varepsilon=\varepsilon_\alpha$.
Then, for all $t\ge0$ and any probability measure $\mu$ with
$\int_{\R^{2d}}\mathcal V_0\,d\mu<\infty$,
\[
  \mathcal W_2(\mu P_t^\alpha, \pi_\alpha)
  \;\le\;
  C_{\rho,\alpha}^{1/2}\,e^{-\frac12 c_\alpha t}
  \left(\cW_{\rho_{\mathcal V_0,\alpha}}(\mu,\pi_\alpha)\right)^{1/2}.
\]
\end{corollary}

\begin{proof}
We provide the proof in Appendix~\ref{proof:cor:W2-convergence-baseline}.
\end{proof}


\section{Acceleration Analysis}
\label{sec:acceleration}

Corollary~\ref{cor:convergence-V0} yields exponential convergence of the HFHR
dynamics for $\alpha\in[0,\alpha_0]$ and
$\alpha \le (1-\kappa_{\mathrm{adjust}})\frac{\eta_0}{1+\eta_0}\frac{\gamma}{L}$,
when the baseline Lyapunov function $\mathcal V_0$ is used. However, the corresponding Lyapunov drift rate
$\hat\lambda_\alpha=\lambda-\mathcal O(\alpha)$ (Proposition~\ref{prop:V0-drift-HFHR})
is slightly smaller than $\lambda$, and therefore the contraction parameter
that can be selected in Theorem~\ref{thm:master-contraction} need not improve
over the unperturbed case $\alpha=0$.
In this section we show that, under an additional structural condition on $U$,
one can construct an \emph{improved} Lyapunov function $\mathcal V_\alpha$
whose drift rate increases at first order in $\alpha$, and reapply the abstract
contraction framework of Section~\ref{sec:global-contractivity} to obtain
an accelerated convergence bound.

\subsection{Structure condition and refined Lyapunov function}

We introduce the additional structural assumption that allows for a first-order
improvement of the Lyapunov drift.

\begin{assumption}[Asymptotically linear gradient]\label{assump:asymptotic-linear-drift}
There exist a constant $C_{\mathrm{linear}}>0$, a symmetric positive definite matrix
$Q_\infty\in\R^{d\times d}$, and a nonincreasing function
$\varrho:[0,\infty)\to[0,\infty)$ with $\varrho(r)\to 0$ as $r\to\infty$ such that
\begin{equation}\label{eq:gradU-asymptotic-assump}
  \left|\nabla U(q)-Q_\infty q\right|
  \le \varrho(|q|)\,|q|,
  \qquad |q|\ge C_{\mathrm{linear}}.
\end{equation}
\end{assumption}

Assumption~\ref{assump:asymptotic-linear-drift} means that $\nabla U$ is
asymptotically linear in a uniform relative sense: the ratio
$|\nabla U(q)-Q_\infty q|/|q|$ vanishes as $|q|\to\infty$.
Integrating \eqref{eq:gradU-asymptotic-assump} along rays yields the quadratic
tail behavior
$U(q)=\frac12 q^\top Q_\infty q + o(|q|^2)$ as $|q|\to\infty$.
A typical class covered by this assumption is
$U(q)=\frac12 q^\top Q_\infty q + W(q)$ with $\nabla W(q)=o(|q|)$ as $|q|\to\infty$.
We will verify in Section~\ref{sec:case-study} that the examples including multi-well potentials (Section~\ref{sec:case:multi}), Bayesian linear regression with an $L^p$
regularizer (Section~\ref{sec:case:L:p}), and Bayesian binary classification (Section~\ref{sec:case:classification}) all satisfy Assumption~\ref{assump:asymptotic-linear-drift}.

To motivate the refined Lyapunov construction, we first record the exact
contribution of the additional HFHR drift $\mathcal A'$ in \eqref{eq:Aprime-def} acting on the baseline
Lyapunov function $\mathcal V_0$.

\begin{lemma}[Exact decomposition of the interaction drift]
\label{lem:exact-drift-decomp}
Let $\mathcal V_0$ and $\mathcal A'$ be defined by \eqref{eq:V0-general-quadratic} and \eqref{eq:Aprime-def} respectively. Then, for any
potential $U\in C^1(\R^d)$,
\begin{equation}\label{eq:exact-drift-identity}
\mathcal A'\mathcal V_0(q,p)
= -|\nabla U(q)|^2
- \frac{\gamma^2}{2}(1-\lambda)\,\nabla U(q)\cdot q
- \frac{\gamma}{2}\,\nabla U(q)\cdot p.
\end{equation}
\end{lemma}

\begin{proof}
    We provide the proof in Appendix~\ref{app:exact-drift-decomp}.
\end{proof}

Heuristically, under Assumption~\ref{assump:asymptotic-linear-drift}, the interaction drift $\mathcal A'\mathcal V_0$ behaves like a negative multiple of $\mathcal V_0$ in the spatial tail, which suggests introducing a corrector $\mathcal M$ to realize a uniform drift gain. It motivates the structural condition and the constant $c_{\mathrm{imp}}$ introduced in the following lemma.

\begin{lemma}[First-order improvement]\label{lem:first-order-improvement}
Assume that $U$ satisfies Assumption~\ref{assump:potential} and
Assumption~\ref{assump:asymptotic-linear-drift}, and let $\mathcal V_0$
be the kinetic Langevin Lyapunov function defined in
\eqref{eq:V0-general-quadratic}. Then there exists a function
$\mathcal M:\R^{2d}\to\R$ with the following properties:

\begin{enumerate}[label=(\roman*)]
  \item \textbf{Growth and regularity.}
    The function $\mathcal M$ is $C^2$ and has at most quadratic growth:
    \begin{equation}\label{eq:M-growth}
      |\mathcal M(q,p)|
      \;\le\; C_\mathcal M\left(1+|q|^2+|p|^2\right),
      \qquad (q,p)\in\R^{2d},
    \end{equation}
    where 
    \begin{equation}\label{eq:C_M-def}
      C_\mathcal M:= \frac{\|\mathsf K\|_{\mathrm{op}}}{2}<\infty,
    \end{equation}
    where $\mathsf K$ is given in \eqref{eq:K-matrix}
    and its first derivatives have at most linear growth:
    \begin{equation}\label{eq:gradM-growth}
      |\nabla_q \mathcal M(q,p)| + |\nabla_p \mathcal M(q,p)|
      \;\le\; C_\mathcal M\left(1+|q|+|p|\right),
      \qquad (q,p)\in\R^{2d}.
    \end{equation}
    Moreover, 
    \begin{equation}\label{eq:lapM-growth}
      |\Delta_q \mathcal M(q,p)| + |\Delta_p \mathcal M(q,p)|
      \;\le\; C_\Delta\left(1+\mathcal V_0(q,p)\right),
      \qquad (q,p)\in\R^{2d},
    \end{equation}
    where
    \begin{equation}\label{eq:C_Delta-def}
      C_\Delta:=2d\,\|\mathsf K\|_{\mathrm{op}}<\infty.
    \end{equation}

  \item \textbf{First-order improvement.}
    There exist explicit constants $\underline c_{\mathrm{imp}}>0$ and $C_{\mathrm{imp}}\ge0$
    such that
    \begin{equation}\label{eq:improvement-condition}
      \mathcal A_0 \mathcal M(q,p) + \mathcal A'\mathcal V_0(q,p)
      \;\le\;
      C_{\mathrm{imp}} - \underline c_{\mathrm{imp}}\,\mathcal V_0(q,p),
      \qquad (q,p)\in\R^{2d},
    \end{equation}
    where $\mathcal A_0$ and $\mathcal A'$ are defined in
    \eqref{eq:A0-def}--\eqref{eq:Aprime-def}. Moreover, the constant
    $\underline c_{\mathrm{imp}}$ can be chosen as
    \begin{equation}\label{eq:cimp-lemma}
      \underline c_{\mathrm{imp}}
      :=
      \frac{3}{8}\cdot
      \frac{
        a_{\min}+1-\sqrt{(a_{\min}-1)^2+\gamma^2}
      }{
        a_{\max}+1+\sqrt{(a_{\max}-1)^2+\gamma^2}
        \;+\;
        8\delta_U(R_0)
      },
    \end{equation}
    and
    \begin{equation}
      C_{\mathrm{imp}}
      := \sup_{p,q\in\mathbb{R}^{d}:|q|\leq R_{0}}
          \left\{
            \mathcal A_0\mathcal M(q,p) + \mathcal A'\mathcal V_0(q,p)
            + \underline c_{\mathrm{imp}}\mathcal V_0(q,p)
          \right\}
      <\infty,
    \end{equation}
    where
    \begin{equation}\label{eq:amin-amax-lemma}
      a_{\min}:=\lambda_{\min}(Q_\infty)+\frac{\gamma^2}{2}(1-\lambda),
      \qquad
      a_{\max}:=\lambda_{\max}(Q_\infty)+\frac{\gamma^2}{2}(1-\lambda),
    \end{equation}
    and $\delta_U(R)$ is defined by
    \[
      \delta_U(R)
      :=\sup_{|q|\ge R}\frac{\left|U(q)-\frac12\langle Q_\infty q,q\rangle\right|}{1+|q|^2},
      \qquad R\ge 1.
    \]
    This choice suffices since on $\{|q|\ge R_0\}$ we have the uniform negative drift
    $\mathcal A_0\mathcal M+\mathcal A'\mathcal V_0\le -\underline c_{\rm imp}\mathcal V_0$.
    Finally, the cutoff radius $R_0$ is explicitly defined as follows.
    Let $Q_\infty$ and $\varrho(\cdot)$ be as in Assumption~\ref{assump:asymptotic-linear-drift}.
    Define the tail modulus
    \begin{equation}\label{defn:tail:modulus}
      \rho_\nabla(R)
      :=\sup_{|q|\ge R}\frac{\left|\nabla U(q)-Q_\infty q\right|}{|q|},
      \qquad R\ge 1.
    \end{equation}
    Note that by Assumption~\ref{assump:asymptotic-linear-drift}, $\rho_\nabla(R)\le \varrho(R)$ for all $R\ge C_{\mathrm{linear}}$.
    Then set
    \begin{equation}\label{eq:R0-def-lemma}
      \rho_\star
      :=\frac{-A+\sqrt{A^2+\frac54\underline{a}}}{2},
      \quad
      A
      := 2\left(\left\|\mathsf K_{pq}\right\|_{\mathrm{op}}+\left\|\mathsf K_{pp}\right\|_{\mathrm{op}}\right)
         +4\lambda_{\max}(Q_\infty)+\gamma^2|1-\lambda|+\gamma,
    \end{equation}
    and
    \begin{equation}\label{eq:R0-choice-lemma}
      R_0:=\inf\{R\ge \max\{1,C_{\mathrm{linear}}\}:\rho_\nabla(R)\le\rho_\star\},
    \end{equation}
    where $\underline{a}$ is given in \eqref{eq:alpha-bounds} in the proof.
\end{enumerate}

In particular, $\mathcal M$ can be chosen to be a quadratic polynomial
in $(q,p)\in\mathbb{R}^{2d}$ such that $\mathcal M(z)=\frac12 z^\top \mathsf K z$,
where
\begin{equation}\label{eq:K-matrix}
  \mathsf K=\begin{pmatrix}\mathsf K_{qq}&\mathsf K_{qp}\\ \mathsf K_{pq}&\mathsf K_{pp}\end{pmatrix},
\end{equation}
is symmetric and $\mathsf K$ is the solution to
\[
  B^\top \mathsf K + \mathsf K B = C_{B_1},
  \qquad
  B_1(z)=\frac12 z^\top C_{B_1} z,
\]
with
\[
  B_1(q,p) := -Q(q,p) - \left(
      \frac12|p|^2 + \frac{\gamma}{2}\langle q,p\rangle
      + \frac12\left\langle\left(Q_\infty+\frac{\gamma^2}{2}(1-\lambda)I_d\right)q,q\right\rangle
    \right).
\]
Equivalently, $C_{B_1}=\nabla^2 B_1$ is the explicit symmetric matrix associated with the quadratic form $B_1$.
\end{lemma}

\begin{proof}
    We provide the proof in Appendix~\ref{app:first-order-improvement}.
\end{proof}

\begin{remark}
Lemma~\ref{lem:first-order-improvement} is only needed to obtain a
\emph{first-order improvement} of the Lyapunov drift rate (and, via the master
contraction theorem, an improved Wasserstein contraction rate) for the HFHR
dynamics. Basic geometric ergodicity and contraction already follow from
Assumption~\ref{assump:potential} alone by using the uncorrected Lyapunov
function $\mathcal V_0$.
In Section~\ref{sec:case-study}, we verify Assumption~\ref{assump:asymptotic-linear-drift}
and illustrate the construction of the quadratic corrector $\mathcal M$ in
Lemma~\ref{lem:first-order-improvement} for several representative examples,
including multi-well potentials (Section~\ref{sec:case:multi}), Bayesian linear regression with an $L^p$
regularizer (Section~\ref{sec:case:L:p}), and Bayesian binary classification (Section~\ref{sec:case:classification}).
\end{remark}

Under the assumptions of Lemma~\ref{lem:first-order-improvement}, define the
refined Lyapunov function
\begin{equation}\label{eq:Valpha-def}
  \mathcal V_\alpha(q,p) := \mathcal V_0(q,p) + \alpha \mathcal M(q,p).
\end{equation}

We now show that $\mathcal V_\alpha$ satisfies a Lyapunov drift condition for
the HFHR generator $\cL_\alpha$ with a \emph{strictly improved} rate at
first order in $\alpha$.
Recalling \eqref{eq:Lalpha-decomp}, we write
\begin{equation}\label{eq:Lalpha-decomp-appB}
  \cL_\alpha = \cL_0 + \alpha \mathcal A' + \alpha \Delta_q,
\end{equation}
where $\cL_0=\mathcal A_0+\gamma\Delta_p$ is the kinetic Langevin generator and
$\mathcal A_0,\mathcal A'$ are defined in \eqref{eq:A0-def}--\eqref{eq:Aprime-def}.

Before analyzing the drift of $\mathcal V_\alpha$, we verify that the
perturbation term $\alpha \mathcal M$ does not change the global growth of the
Lyapunov function: for $\alpha$ sufficiently small, $\mathcal V_\alpha$ remains
equivalent to $\mathcal V_0$ up to explicit constants.

\begin{lemma}[Equivalence of $\mathcal V_\alpha$ and $\mathcal V_0$]
\label{lem:Valpha-V0-equivalence}
Assume Assumption~\ref{assump:potential} and Lemma~\ref{lem:first-order-improvement}.
Let $C_{\mathcal M}$ be as in \eqref{eq:C_M-def} and define
\begin{equation}\label{eq:alpha-star-explicit}
  \alpha_* := \frac{c_1}{2C_{\mathcal M}}.
\end{equation}
Then for all $\alpha\in[0,\alpha_*]$ and all $(q,p)\in\R^{2d}$,
\begin{equation}\label{eq:Valpha-V0-equivalence}
  \frac12\left(1+\mathcal V_0(q,p)\right)
  \;\le\; 1+\mathcal V_\alpha(q,p)
  \;\le\; \frac32\left(1+\mathcal V_0(q,p)\right).
\end{equation}
\end{lemma}

\begin{proof}
    We provide the proof in Appendix~\ref{app:Valpha-V0-equivalence}.
\end{proof}

With the growth bounds established, we now turn to the analysis on the drift. The following lemma provides an expansion of the generator action $\cL_\alpha \mathcal V_\alpha$ in powers of $\alpha$, which allows us to isolate the first-order contribution responsible for the acceleration.

\begin{lemma}[Drift expansion for $\mathcal V_\alpha$]
\label{lem:Valpha-drift-expansion}
Under Assumption~\ref{assump:potential} and Lemma~\ref{lem:first-order-improvement},
let $\mathcal V_\alpha=\mathcal V_0+\alpha\mathcal M$, where $\mathcal M$ is the quadratic polynomial
constructed in Lemma~\ref{lem:first-order-improvement}.
Let $\underline c_{\mathrm{imp}}>0$ and $C_{\mathrm{imp}}\ge0$ be the explicit constants
in \eqref{eq:improvement-condition}--\eqref{eq:cimp-lemma}. Let $c_1$ be the constant in
\eqref{eq:V0-equivalent}. Define
\[
  K_\Delta
  := dL+\frac{\gamma^2}{2}d|1-\lambda|,
  \qquad
  \text{so that}\quad
  |\Delta_q\mathcal V_0(q,p)|\le K_\Delta
  \quad\text{for a.e. }(q,p)\in\R^{2d}.
\]
Moreover, since $\mathcal M(z)=\frac12 z^\top \mathsf K z$ with $\mathsf K$ as in \eqref{eq:K-matrix},
\[
  \Delta_p\mathcal M(q,p)=\mathrm{tr}(\mathsf K_{pp})
  \quad\text{and}\quad
  \Delta_q\mathcal M(q,p)=\mathrm{tr}(\mathsf K_{qq})
  \qquad\text{for all }(q,p)\in\R^{2d}.
\]
Then for all $\alpha\in[0,1]$ and all $(q,p)\in\R^{2d}$,
\begin{align}\label{eq:Lalpha-Valpha-bound-fixed}
  \cL_\alpha\mathcal V_\alpha(q,p)
  \le\;& \gamma(d+A) - \lambda\mathcal V_0(q,p)
      + \alpha\left( C_1 - \underline c_{\mathrm{imp}}\mathcal V_0(q,p)\right)
      + C_2\alpha^2(1+\mathcal V_0(q,p)),
\end{align}
where we can choose
\begin{equation}\label{eq:C1C2-explicit-fixed}
  C_1:=C_{\mathrm{imp}}+\gamma\,\mathrm{tr}(\mathsf K_{pp})+K_\Delta,
  \qquad
  C_2:=|\mathrm{tr}(\mathsf K_{qq})|
      + 3\,C_{\mathcal M}\frac{(L+|\nabla U(0)|)}{c_1},
\end{equation}
where $C_{\mathcal M}$ is the constant in \eqref{eq:C_M-def}.
\end{lemma}

\begin{proof}
    We provide the proof in Appendix~\ref{app:Valpha-drift-expansion}.
\end{proof}

Building on this expansion and the equivalence of $\mathcal{V}_{\alpha}$ and $\mathcal{V}_{0}$ established in Lemma~\ref{lem:Valpha-V0-equivalence}, we can now state one of the main results of this section: the improved Lyapunov function $\mathcal V_\alpha$ yields a \textit{strictly improved} drift rate for small $\alpha$.

\begin{proposition}[Enhanced drift rate for HFHR dynamics]
\label{prop:Valpha-drift}
Suppose Assumption~\ref{assump:potential} and Lemma~\ref{lem:first-order-improvement} hold,
and let $\mathcal V_\alpha=\mathcal V_0+\alpha \mathcal M$ be defined as in \eqref{eq:Valpha-def}.
Let $\alpha_0>0$ be the explicit constant in \eqref{eq:alphastar-explicit}
(from Proposition~\ref{prop:V0-drift-HFHR}).
Let $C_{\mathcal M}$ be the constant in \eqref{eq:M-growth} and $c_1$ be the constant in \eqref{eq:V0-equivalent}.
Define
\begin{equation}\label{eq:widetildeCM-def}
  \widetilde C_{\mathcal M}:=\frac{C_{\mathcal M}}{c_1},
  \qquad\text{so that}\qquad
  |\mathcal M|\le \widetilde C_{\mathcal M}\,(1+\mathcal V_0).
\end{equation}
Let $C_1, C_2$ be the explicit constants in Lemma~\ref{lem:Valpha-drift-expansion}.
Define
\begin{align}
  \delta &:= \underline c_{\mathrm{imp}} - \lambda\,\widetilde C_{\mathcal M},
  \label{eq:delta-explicit-again-fixed} \\
  C_\lambda &:= C_2 + \widetilde C_{\mathcal M}\,\underline c_{\mathrm{imp}},
  \label{eq:Clambda-explicit-fixed}
\end{align}
where $\lambda$ is the dissipativity constant in Assumption~\ref{assump:potential}(iii).
Define the effective drift constant $A_\alpha'$ (note: $A_\alpha'$ corresponds to the ``$A$'' in Definition~\ref{def:admissible-lyapunov}) by
\begin{equation}\label{eq:Aalpha-explicit-fixed}
  A_\alpha' := A + \frac{\alpha}{\gamma}\left[\, C_1 + \alpha C_2 + \widetilde C_{\mathcal M}(\lambda + \underline c_{\mathrm{imp}}) \right].
\end{equation}

Then there exists an explicit $\alpha_1\in(0,\alpha_0]$ such that for all $\alpha\in(0,\alpha_1]$,
the function $\mathcal V_\alpha$ is $(\lambda_\alpha, A_\alpha')$-admissible for $\cL_\alpha$ with drift rate
\begin{equation}\label{eq:lambda-alpha-expansion}
  \lambda_\alpha \;\ge\; \lambda + \delta\,\alpha - C_\lambda\,\alpha^2.
\end{equation}
Indeed, $\alpha_{1}$ can be explicitly chosen as
\begin{equation}\label{eq:alpha1-explicit}
  \alpha_1
  := \min\left\{\alpha_0,\; 1,\; \alpha_*,\;
  \alpha_{\mathrm{pos}}\right\},
\end{equation}
where $\alpha_{*}$ is given in \eqref{eq:alpha-star-explicit} and $\alpha_{\mathrm{pos}}>0$ is any constant such that
$\lambda+\delta\alpha-C_\lambda\alpha^2\ge \lambda/2$ 
holds for all $\alpha\in(0,\alpha_{\mathrm{pos}}]$, for example, one may take:
\begin{equation}\label{eq:alpha-pos-explicit}
  \alpha_{\mathrm{pos}}
  :=
  \begin{cases}
  \min\left\{1,\ \frac{\delta+\sqrt{\delta^2+2C_\lambda\lambda}}{2C_\lambda}\right\},
  & C_\lambda>0,\\
  \min\left\{1,\ \frac{\lambda}{2\max\{1,|\delta|\}}\right\},
  & C_\lambda=0.
  \end{cases}
\end{equation}
\end{proposition}

\begin{proof}
We provide the proof in Appendix~\ref{app:Valpha-drift}.
\end{proof}

\subsection{Acceleration of the contraction rate}

We now combine Proposition~\ref{prop:Valpha-drift} with the Master
Theorem~\ref{thm:master-contraction} to obtain an accelerated contraction rate for the HFHR dynamics.
First, we recall that Theorem~\ref{thm:master-contraction} gives a contraction rate
\[
  c(\lambda)
  = \frac{\gamma}{384}
    \min\left\{\widetilde\Lambda_{1,\alpha}(\lambda),\,
               \widetilde\Lambda_{2,\alpha}(\lambda),\,
               \widetilde\Lambda_{3,\alpha}(\lambda)\right\}.
\]
The first term $\widetilde\Lambda_{1,\alpha}$ in \eqref{eq:Lambda-branches} corresponds
to the Lyapunov branch, while $\widetilde\Lambda_{2,\alpha}$ and $\widetilde\Lambda_{3,\alpha}$ in \eqref{eq:Lambda-branches} are the metric
branches.
Let
$c_0 := c(\lambda)$
denote the contraction rate for the kinetic Langevin dynamics at
$\alpha=0$, where $\lambda$ is defined in \eqref{eq:L0-V0-Lyapunov}.
From \cite[Eqn.~(2.18)]{Eberle}, we assume that at $\alpha=0$ the Lyapunov branch is active, i.e.
\[
  c_0=\frac{\gamma}{384}\,\widetilde\Lambda_{1,0}(\lambda),
  \qquad
  \widetilde\Lambda_{1,0}(\lambda)\le \widetilde\Lambda_{2,0}(\lambda),\quad
  \widetilde\Lambda_{1,0}(\lambda)\le \widetilde\Lambda_{3,0}(\lambda).
\]
Let $c_\alpha:=c(\lambda_\alpha)$ denote the contraction rate of the
HFHR dynamics when we use the improved Lyapunov function
$\mathcal V_\alpha$, where $\lambda_\alpha$ is defined in
\eqref{eq:lambda-alpha-expansion}.

Since the global contraction rate $c(\lambda)$ is the minimum of the Lyapunov branch
$\widetilde\Lambda_{1,\alpha}$ and the metric branches $\widetilde\Lambda_{2,\alpha},\widetilde\Lambda_{3,\alpha}$,
we analyze the effect of the improved drift $\lambda_\alpha$ on these branches separately.
In particular, the Lyapunov branch improves directly with $\lambda_\alpha$.
For the metric branch, an improvement holds under additional quantitative conditions,
which we verify below for sufficiently small $\alpha$.
\subsubsection{Acceleration on the metric branch}

We now investigate the behavior of the metric branch when the Lyapunov rate is improved to $\lambda_\alpha$.
Recall from \eqref{L:eff} that for HFHR dynamics, the effective Lipschitz constant is
$L_{\mathrm{eff}}(\alpha)=(1+\alpha\gamma)L$.
To match the scaling in \cite{Eberle}, we fix the geometric constant at the unperturbed regime $\alpha=0$.
In particular, since $L_{\mathrm{eff}}(0)=L$, we set $\theta_0:=L\gamma^{-2}$ and define
\begin{equation}\label{eq:J2-Eberle-explicit}
  J_2 \;:=\; \frac{12}{5} \left(1 + 2\theta_0 + 2\theta_0^2\right)
  \frac{d+A}{\gamma^2 (1-2\lambda)}.
\end{equation}

We focus on the dominant term of the metric parameter.
The following theorem shows that if $\delta>\gamma\lambda$,
then the contraction rate on the metric branch is strictly enhanced.

\begin{theorem}[Metric branch acceleration]\label{thm:metric-acceleration}
Assume the conditions of Proposition~\ref{prop:Valpha-drift} hold.
Let $\lambda_\alpha$ be the drift rate given by Proposition~\ref{prop:Valpha-drift}$,$ i.e.
for all $\alpha\in(0,\alpha_1]$,
\[
\lambda_\alpha \;\ge\; \underline\lambda_\alpha
:= \lambda+\delta\,\alpha-C_\lambda\,\alpha^2,
\]
where $\delta$ and $C_\lambda$ are the explicit constants from
\eqref{eq:delta-explicit-again-fixed}--\eqref{eq:Clambda-explicit-fixed}.
Assume in addition that
\[
D:=\delta-\gamma\lambda>0,
\]
where $\lambda$ is the dissipativity constant in Assumption~\ref{assump:potential}(iii).
Define the (dimension-free) metric parameter function for $\alpha\ge0$ by
\begin{equation}\label{eq:Lambda-metric-def}
  \Lambda_\alpha(\lambda) := J_2\,\frac{(1+\alpha\gamma)L}{\lambda}.
\end{equation}
Let $\Lambda_0:=\Lambda_0(\lambda)=J_2L/\lambda$ and assume $\Lambda_0>\frac12$.
Define
\begin{equation}\label{eq:hlambda-def}
h(\Lambda):=\sqrt{\Lambda}\,e^{-\Lambda},
\end{equation}

\begin{equation}\label{eq:Mh-def}
M_h := \sup_{\Lambda \in [\Lambda_0/2, \Lambda_0]}
\left| \frac{\Lambda^2 - \Lambda - 1/4}{\Lambda^{3/2}} e^{-\Lambda} \right|.
\end{equation}

Then there exists an explicit constant
\begin{equation}\label{eq:alpha-metric-acc-final}
    \alpha_{\mathrm{metric,acc}} := \min \left\{
        \alpha_1,\ 1,\
        \frac{D}{4C_\lambda},\
        \sqrt{\frac{\lambda}{2C_\lambda}},\
        \frac{4\lambda}{D},\
        \frac{8 \lambda^2 \sqrt{\Lambda_0} e^{-\Lambda_0} \left(1 - \frac{1}{2\Lambda_0}\right)}{J_2 L\,D\, M_h}
    \right\},
\end{equation}
with the convention that the terms involving $C_\lambda$ are omitted when $C_\lambda=0$,
such that for all $\alpha\in(0,\alpha_{\mathrm{metric,acc}}]$ the following hold:

\begin{enumerate}[label=(\roman*)]
\item \textbf{Metric parameter decreases.}
\[
\Lambda_\alpha(\lambda_\alpha)\le \Lambda_0-c_\Lambda\,\alpha,
\qquad
c_\Lambda:=\frac18\,J_2L\,\frac{D}{\lambda^2}>0.
\]

\item \textbf{Metric branches increase.}
There exist explicit constants $c_2,c_3>0$ such that
\[
\widetilde\Lambda_{2,\alpha}(\lambda_\alpha)\ge \widetilde\Lambda_{2,0}(\lambda)\,(1+c_2\alpha),
\qquad
\widetilde\Lambda_{3,\alpha}(\lambda_\alpha)\ge \widetilde\Lambda_{3,0}(\lambda)\,(1+c_3\alpha),
\]
where
\[
\widetilde\Lambda_{2,\alpha}(\lambda)
:= h(\Lambda_\alpha(\lambda))\,\frac{L_{\mathrm{eff}}(\alpha)}{\gamma^2},
\qquad
\widetilde\Lambda_{3,\alpha}(\lambda)
:= \kappa_{\mathrm{adjust}}\,h(\Lambda_\alpha(\lambda)),
\]
and $L_{\mathrm{eff}}(\alpha)=(1+\alpha\gamma)L$.
More precisely, one may take any $c_3<c_3^*$ with
\[
c_3^*:=\frac12\left(1-\frac{1}{2\Lambda_0}\right)c_\Lambda,
\]
and then set
$c_2 := \gamma + c_3$.
\end{enumerate}
\end{theorem}

\begin{proof}
    We provide the proof in Appendix~\ref{app:metric-acceleration}.
\end{proof}

\begin{remark}
The abstract condition $\delta>\gamma\lambda$ appearing in
Theorem~\ref{thm:metric-acceleration} is purely quantitative: it compares
the strength of the first-order drift improvement to the baseline
Lyapunov rate, after accounting for the $\mathcal O(\alpha)$ increase of the
effective Lipschitz constant $L_{\mathrm{eff}}(\alpha)=(1+\gamma\alpha)L$
in the metric parameter $\Lambda_\alpha(\lambda)=J_2 L_{\mathrm{eff}}(\alpha)/\lambda$.
In Section~\ref{sec:case-study} we verify this condition explicitly for the
multi-well potential for suitable choices of $\gamma$ and $\lambda$
(Section~\ref{sec:case:multi}), Bayesian linear regression with $L^p$
regularizer (Section~\ref{sec:case:L:p}), and Bayesian binary
classification (Section~\ref{sec:case:classification}). This shows that,
for these examples, HFHR dynamics improves not only the Lyapunov branch
but also the metric branch governing barrier crossing.
\end{remark}

\subsubsection{Acceleration on the Lyapunov branch}

To prove acceleration, we first verify that for small $\alpha$,
the contraction rate $c_\alpha$ remains governed by the Lyapunov branch,
i.e.\ that the minimum in the definition of $c(\lambda_\alpha)$ is still
attained at $\widetilde\Lambda_{1,\alpha}(\lambda_\alpha)$.
In this case, the gain in $\lambda_\alpha$ directly translates into a gain in the convergence
speed. Indeed, according to the definition in \eqref{eq:rate-explicit} and the expression
for $\widetilde\Lambda_{1,\alpha}$ in \eqref{eq:Lambda-branches}, we have
\[
  c_\alpha
  = \frac{\gamma}{384}\,\widetilde\Lambda_{1,\alpha}(\lambda_\alpha)
  = \frac{\gamma}{384}\,\frac{\lambda_\alpha L_{\mathrm{eff}}(\alpha)}{\gamma^{2}}.
\]

To establish acceleration on the Lyapunov branch, we first need to ensure that the convergence bottleneck remains determined by the drift from infinity rather than switching to the metric coupling regime. The following lemma guarantees this stability for sufficiently small perturbation parameters.

\begin{lemma}[Continuity of the active branch]
\label{lem:branch-continuity}
Let $\alpha_1>0$ be as in Proposition~\ref{prop:Valpha-drift}.
Assume that at $\alpha=0$ the Lyapunov branch is strictly active:
\[
  \widetilde\Lambda_{1,0}(\lambda) < \widetilde\Lambda_{2,0}(\lambda),
  \qquad
  \widetilde\Lambda_{1,0}(\lambda) < \widetilde\Lambda_{3,0}(\lambda).
\]
where $\lambda$ is the baseline drift rate at $\alpha=0$.
Assume moreover that $\alpha\mapsto \widetilde\Lambda_{i,\alpha}(\lambda_\alpha)$ is continuous on $[0,\alpha_1]$
for $i=1,2,3$.
Define
\[
  \Delta(\alpha)
  := \min\left\{\widetilde\Lambda_{2,\alpha}(\lambda_\alpha),\,\widetilde\Lambda_{3,\alpha}(\lambda_\alpha)\right\}
     - \widetilde\Lambda_{1,\alpha}(\lambda_\alpha).
\]
and set
\begin{equation}\label{eq:alphabranch-def}
  \alpha_{\mathrm{branch}}
  := \sup\{\alpha\in(0,\alpha_1]:\ \Delta(\beta)>0\ \text{for all } \beta\in[0,\alpha]\}.
\end{equation}
Then $\alpha_{\mathrm{branch}}>0$ and for all $\alpha\in(0,\alpha_{\mathrm{branch}}]$ the Lyapunov branch remains active, i.e.
\[
  c_\alpha = c(\lambda_\alpha) = \frac{\gamma}{384}\,\widetilde\Lambda_{1,\alpha}(\lambda_\alpha).
\]
\end{lemma}

\begin{proof}
We provide the proof in
Appendix~\ref{app:branch-continuity}.
\end{proof}

\begin{remark}
The assumption that the Lyapunov branch $\widetilde\Lambda_{1,\alpha}$
is active at $\alpha=0$ is not automatic in general and should be viewed as a structural condition on the dynamics. It describes regimes in which the global convergence rate is genuinely controlled by the drift from infinity (encoded in the Lyapunov parameter $\lambda$), while the local contraction mechanisms (captured by $\widetilde\Lambda_{2,\alpha},\widetilde\Lambda_{3,\alpha}$)
are comparatively fast. In particular, for strongly metastable targets with pronounced energy barriers (such as classical double-well potentials), explicit bounds in \cite{Eberle} indicate that the contraction rate is often dominated by the metric branch associated with barrier crossing rather than by the Lyapunov branch. Our lemma shows that \emph{whenever} the Lyapunov branch is strictly active at $\alpha=0$, it remains active for all sufficiently small~$\alpha>0$.
\end{remark}

With the continuity of the active branch established, we are now in a position to translate the improved Lyapunov drift $\lambda_\alpha$ directly into a quantitative acceleration of the convergence rate.

\begin{theorem}[Lyapunov branch acceleration]\label{thm:lyapunov-acceleration}
Let $c_0=c(\lambda)$ be the contraction rate of the kinetic Langevin
dynamics at $\alpha=0$ given by Theorem~\ref{thm:master-contraction}.
Assume moreover that the hypotheses of Lemma~\ref{lem:branch-continuity}
hold (in particular, the Lyapunov branch is \emph{strictly} active at $\alpha=0$),
and let $\alpha_{\mathrm{branch}}\in(0,\alpha_1]$ be the threshold defined in \eqref{eq:alphabranch-def}.

Let $c_\alpha := c(\lambda_\alpha)$ denote the contraction rate of the HFHR
dynamics obtained by applying Theorem~\ref{thm:master-contraction} with
$\mathcal V=\mathcal V_\alpha$.
Then there exists
\[
  \alpha_{\mathrm{branch,acc}}
  := \min\left\{\alpha_{\mathrm{branch}},\,1,\,\frac{\kappa}{C'}\right\}
  \in (0,\alpha_{\mathrm{branch}}],
\]
where
\begin{equation}\label{eq:cprime-def}
  C'=\frac{L}{384\gamma}(1+\gamma)\left(C_2 + \widetilde C_{\mathcal M}\,\underline c_{\mathrm{imp}}\right),
\end{equation}
where $C_2$ is defined in \eqref{eq:C1C2-explicit-fixed}, $\widetilde C_{\mathcal M}$ is defined in \eqref{eq:widetildeCM-def}, and $\underline c_{\mathrm{imp}}$ is defined in \eqref{eq:cimp-lemma},
such that for all $\alpha\in(0,\alpha_{\mathrm{branch,acc}}]$,
\begin{equation}\label{eq:acceleration-inequality}
  c_\alpha \;\ge\; c_0 + \kappa\,\alpha,
\end{equation}
where
\begin{equation}\label{defn:kappa}
  \kappa := \frac{L(\delta+\gamma\lambda)}{768\,\gamma},
\end{equation}
and $\delta>0$ and $C_\lambda\ge0$ are the constants from
Proposition~\ref{prop:Valpha-drift}.
\end{theorem}

\begin{proof}
We provide the proof in
Appendix~\ref{app:acceleration}.
\end{proof}

\subsubsection{Global Acceleration}

Based on the acceleration of both the metric branch (Theorem~\ref{thm:metric-acceleration}) and the Lyapunov branch (Theorem~\ref{thm:lyapunov-acceleration}), we conclude with the following global acceleration result.

\begin{corollary}[Global acceleration of HFHR dynamics]
\label{cor:global-acceleration}
Assume the conditions of Proposition~\ref{prop:Valpha-drift} and suppose that
\[
  \delta>\gamma\lambda
  \qquad\text{and}\qquad
  \Lambda_0:=\Lambda(\lambda)=\frac{J_2 L}{\lambda}>\frac12,
\]
where $\lambda$ is the dissipativity constant in Assumption~\ref{assump:potential}(iii), so that Theorem~\ref{thm:metric-acceleration} applies.
Let $\alpha_1$ be as in Proposition~\ref{prop:Valpha-drift},
$\alpha_{\mathrm{branch,acc}}\in(0,\alpha_1]$ be the threshold from Theorem~\ref{thm:lyapunov-acceleration}, 
and $\alpha_{\mathrm{metric,acc}}\in(0,\alpha_1]$ be the threshold from Theorem~\ref{thm:metric-acceleration}.
Define
\[
  \alpha_{\mathrm{global}}
  :=\min\{\alpha_{\mathrm{branch,acc}},\alpha_{\mathrm{metric,acc}}\}. 
\]
Then for all $\alpha\in(0,\alpha_{\mathrm{global}}]$, the HFHR dynamics achieves a strictly
better contraction rate than the kinetic Langevin dynamics, namely
\[
  c_\alpha \;\ge\; c_0 + \kappa_{\mathrm{global}}\,\alpha,
\]
where
\[
  \kappa_{\mathrm{global}}
  := \min\left\{ \kappa,\; c_0 c_2,\; c_0 c_3 \right\} \;>\;0.
\]
Here $c_0=c(\lambda)$ is the contraction rate at $\alpha=0$,
$\kappa$ is the Lyapunov-branch acceleration constant from
Theorem~\ref{thm:lyapunov-acceleration},
and $c_2,c_3>0$ are the metric-branch improvement constants from
Theorem~\ref{thm:metric-acceleration}. 
\end{corollary}

\begin{proof}
    We provide the proof in Appendix~\ref{app:global-acceleration}.
\end{proof}

Corollary~\ref{cor:global-acceleration} shows that, for all sufficiently small $\alpha>0$, the HFHR dynamics
has a \textit{strictly better} contraction rate than that of kinetic Langevin
dynamics in the weighted Wasserstein distance $\cW_{\rho_{\mathcal
V_\alpha}}$.
Finally, we demonstrate that this acceleration is not just an artifact of the weighted Wasserstein distance but directly translates to the standard 2-Wasserstein distance $\mathcal W_2$.

\begin{corollary}[Acceleration in the $2$-Wasserstein distance]
\label{cor:W2-acceleration}
Assume the setting of Corollary~\ref{cor:global-acceleration}, and in addition
choose the drift parameter in Theorem~\ref{thm:master-contraction} as
\[
  \lambda_\alpha := \underline\lambda_\alpha
  = \lambda+\delta\alpha-C_\lambda\alpha^2
  \qquad (\alpha>0),
\]
where $\lambda$ is the dissipativity constant in Assumption~\ref{assump:potential}(iii), $\delta,C_\lambda$ are from Proposition~\ref{prop:Valpha-drift}.
Let $c_0$ and $c_\alpha$ be the contraction rates in $\cW_{\rho_{\mathcal V_\alpha}}$
given by Corollary~\ref{cor:global-acceleration}, and set
\[
  \alpha_{\mathrm{W2}}
  := \min\left\{\alpha_{\mathrm{global}},\ \alpha_{\mathrm{pos}}\right\},
\]
where $\alpha_{\mathrm{global}}$ is the threshold in Corollary~\ref{cor:global-acceleration}
and $\alpha_{\mathrm{pos}}$ is the explicit constant in \eqref{eq:alpha-pos-explicit}
(so that $\underline\lambda_\alpha\ge \lambda/2$ for all $\alpha\in(0,\alpha_{\mathrm{pos}}]$).

Define the interval
\[
  I_\lambda := \left[\lambda_-,\lambda_+\right]
  \quad\text{with}\quad
  \lambda_-:=\frac{\lambda}{2},
  \qquad
  \lambda_+ := \lambda+\delta\,\alpha_{\mathrm{W2}}.
\]
For each $\lambda\in I_\lambda$, let $R_1(\lambda)$ be any admissible cutoff radius
satisfying \eqref{eq:R1-condition-v2} 
(with the corresponding choices of
$\eta_0(\lambda)=\left(\Lambda(\lambda)\right)^{-1}$ and $f_\lambda$ as in Theorem~\ref{thm:master-contraction}),
and let $\varphi_\lambda,\Phi_\lambda,g_\lambda,f_\lambda$ be defined by
\eqref{eq:phi-def}--\eqref{eq:f-def} with this $R_1(\lambda)$.
Define the explicit extremal constants
\[
  R_1(\lambda)^+ := \sup_{\lambda\in I_\lambda} R_1(\lambda),
  \qquad
  c_r^- := \inf_{\lambda\in I_\lambda}\ \inf_{0\le s\le R_1(\lambda)} \varphi_\lambda(s),
  \qquad
  g_*^- := \inf_{\lambda\in I_\lambda}\ \inf_{0\le s\le R_1(\lambda)} g_\lambda(s),
\]
and
\[
  c_0^- := \inf_{\lambda\in I_\lambda} f_\lambda(R_1(\lambda))
         = \inf_{\lambda\in I_\lambda}\int_{0}^{R_1(\lambda)}\varphi_\lambda(s)g_\lambda(s)\,ds.
\]
Assume $g_*^->0$ (this is ensured, for instance, by the explicit parameter choice in
\eqref{eq:rate-explicit}, cf.\ the verification of \eqref{eq:g-positive} in the proof of
Theorem~\ref{thm:master-contraction}).

Moreover, by Lemma~\ref{lem:Valpha-V0-equivalence} and the quadratic growth of $\mathcal V_0$,
fix explicit constants $C_V^{\mathrm{unif}}>0$ and $C\ge0$ such that
\[
  |z|^2 \le C_V^{\mathrm{unif}}\left(1+\mathcal V_\alpha(z)\right) + C,
  \qquad z\in\R^{2d},\ \alpha\in(0,\alpha_{\mathrm{W2}}].
\]
Let $k_1(\lambda)$ be the norm-equivalence constant in Lemma~\ref{lem:r-equivalent}
corresponding to the choice of $\theta(\lambda)$ in \eqref{eq:theta-def} (with $\eta_0(\lambda)$),
and define
\[
  k_1^- := \inf_{\lambda\in I_\lambda} k_1(\lambda) \;>\;0.
\]

Finally, let $\varepsilon_\alpha=\frac{4c_\alpha}{\gamma(d+A_\alpha')}$ be the parameter in
Theorem~\ref{thm:master-contraction} (applied with $\mathcal V=\mathcal V_\alpha$ and
$(\lambda,D)=(\lambda_\alpha,A_\alpha')$), and define the explicit lower bound
\[
  \varepsilon^- := \inf_{\alpha\in(0,\alpha_{\mathrm{W2}}]} \varepsilon_\alpha \;>\;0.
\]
(For instance, one may take $\varepsilon^-=\frac{4c_0}{\gamma(d+A^+)}$ with
$A^+:=\sup_{\alpha\in(0,\alpha_{\mathrm{W2}}]}A_\alpha'<\infty$, using
$c_\alpha\ge c_0$ from Corollary~\ref{cor:global-acceleration}.)
Set
\begin{equation}\label{eq:Crho-unif-explicit}
  C_{\rho}^{\mathrm{unif}}
  := \frac1{\varepsilon^-}
     \max\left\{
       \frac{(k_1^-)^{-2} R_1(\lambda)^+}{g_*^-\,c_r^-},\,
       \frac{4 C_V^{\mathrm{unif}}}{c_0^-}
     \right\}.
\end{equation}

Then for all $\alpha\in(0,\alpha_{\mathrm{W2}}]$, all $t\ge0$, and all
probability measures $\mu,\nu$ with finite $\mathcal V_\alpha$-moments,
\begin{equation}\label{eq:W2-acceleration}
  \mathcal W_2(\mu P_t^\alpha,\nu P_t^\alpha)
  \;\le\;
  \left(C_{\rho}^{\mathrm{unif}}\right)^{1/2}
  e^{-c_\alpha^{(2)} t}
  \left(\cW_{\rho_{\mathcal V_\alpha}}(\mu,\nu)\right)^{1/2},
\end{equation}
where
$c_\alpha^{(2)} := \frac12 c_\alpha$.
Moreover, the acceleration holds in $\mathcal W_2$ with explicit gain:
\[
  c_\alpha^{(2)} \;\ge\; c_0^{(2)} + \kappa^{(2)}\,\alpha,
  \qquad \alpha\in(0,\alpha_{\mathrm{W2}}],
\]
with
$c_0^{(2)} := \frac12 c_0$ and $\kappa^{(2)} := \frac12 \kappa_{\mathrm{global}}$.
\end{corollary}

\begin{proof}
    We provide the proof in Appendix~\ref{app:W2-acceleration}.
\end{proof}


\section{Case Study}
\label{sec:case-study}

In this section, we illustrate our general results through three concrete
non-convex examples: a multi-well potential (Section~\ref{sec:case:multi}),
Bayesian linear regression with $L^p$ regularizer (Section~\ref{sec:case:L:p})
and Bayesian binary classification (Section~\ref{sec:case:classification}).
We verify that these examples all satisfy Assumption~\ref{assump:potential}, recall the
baseline contraction estimate for kinetic Langevin dynamics, and then construct explicit quadratic correctors
$\mathcal M$ tailored to these specific examples
and yields an improvement constant that can be larger than the generic lower bound provided by the abstract theory.
Finally, we show that HFHR dynamics achieves a strictly better contraction rate than that of kinetic Langevin dynamics
for all sufficiently small~$\alpha>0$.

\subsection{Multi-well potential}\label{sec:case:multi}

In this section, we study the example of a high-dimensional non-convex potential. 
Specifically, we consider a $d$-dimensional multi-well potential $U:\mathbb{R}^d\to\mathbb{R}$ constructed as a sum of independent one-dimensional double-wells~\cite[Example 1.1]{Eberle}. Let $z=(q,p)\in\R^{2d}$. We define
\begin{equation}\label{eq:MW-high-dim}
  U(q) := \sum_{i=1}^d v(q_i),
\end{equation}
where $v:\R\to\R$ is the component-wise potential given by
\begin{equation}\label{eq:MW-1D-component}
  v(s) :=
  \begin{cases}
    \frac{1}{2}(|s|-1)^2, & |s|\ge \frac12,\\[0.5em]
    \frac14 - \frac12 s^2, & |s|\le \frac12.
  \end{cases}
\end{equation}
The potential $U$ has $2^d$ local minima located at $(\pm 1, \dots, \pm 1)$ and presents a classic benchmark for sampling multi-modal distributions in high dimensions.
Let us verify that $U$ satisfies the all structural assumptions, i.e. Assumptions~\ref{assump:potential} and \ref{assump:asymptotic-linear-drift}, required for our theory.

\begin{proposition}[Verification of Assumptions]\label{prop:MW-assumptions}
Fix $\gamma>0$. The potential $U$ defined in \eqref{eq:MW-high-dim}
satisfies Assumption~\ref{assump:potential} and
Assumption~\ref{assump:asymptotic-linear-drift}. Specifically:
\begin{enumerate}
  \item[\textnormal{(a)}]\textit{Regularity and Lipschitz gradient:}
  $U$ is $C^{1}(\mathbb R^{d})$ and $\nabla U$ is $L$-Lipschitz with $L=1$.

  \item[\textnormal{(b)}] \textit{Dissipativity:}
  The dissipativity condition \eqref{eq:U-drift} holds for \emph{any}
  $\lambda\in\left(0,\ \frac{1}{4+\gamma^{2}}\right]$,
  with a constant $A$ scaling linearly in $d$. More precisely, one may take
  \[
    A \;:=\; d\,A_{1}(\gamma),
    \qquad
    A_{1}(\gamma)
    \;:=\;
    \frac{\gamma^4+6\gamma^2+16}{4(\gamma^4+10\gamma^2+24)}.
  \]
  In particular, the value $\lambda=1/(4+\gamma^{2})$ is only a convenient
  \emph{upper bound}; smaller choices of $\lambda$ remain valid with the same $A$.

  \item[\textnormal{(c)}] \textit{Asymptotic linear drift:}
  For any $|q|\ge \sqrt d$,
  \[
    |\nabla U(q)-q|\le \varrho(|q|)\,|q|,
    \qquad \varrho(r):=\sqrt d/r,
  \]
  so that $\varrho(r)\to 0$ as $r\to\infty$ for each fixed $d$.
\end{enumerate}
\end{proposition}

\begin{proof}
    We provide the proof in Appendix~\ref{app:MW-assumptions}.
\end{proof}

Since we have verified in Proposition~\ref{prop:MW-assumptions} that the potential $U$ satisfies Assumptions~\ref{assump:potential} and \ref{assump:asymptotic-linear-drift}, the theoretical results 
in Section~\ref{sec:global-contractivity} and Section~\ref{sec:acceleration} are all applicable.
However, due to the very special structure of the multi-well potential $U$ in \eqref{eq:MW-high-dim}, one can obtain
sharper acceleration results
by exploiting the special tail structure of the multi-well potential $U$ in \eqref{eq:MW-high-dim}
to construct an explicit quadratic corrector $\mathcal M$ and obtain acceleration.
While Lemma~\ref{lem:first-order-improvement} already guarantees the \emph{existence} of a corrector under the abstract
asymptotically linear drift condition, the present separable multi-well model allows for a more precise construction:
by matching the quadratic part of $U$ at infinity, we choose $\mathcal M$ so that the dominant interaction drift
$\mathcal A'\mathcal V_0$ is canceled (up to uniformly bounded remainders).
This yields (i) a closed-form improvement constant $c_{\mathrm{imp}}$ and (ii) explicit, dimension-controlled bounds on
the auxiliary constants (such as $C_{\mathcal M}$ and the second-order error coefficient), which are not available from the general existence argument.

The explicit corrector constructed below is consistent with the general theory:
Lemma~\ref{lem:first-order-improvement} characterizes $\mathcal M$ through a Lyapunov equation for the quadratic form at infinity,
and our choice of $\mathcal M$ is precisely one such quadratic solution specialized to the present isotropic/separable setting.
In particular, it can be viewed as a concrete representative of the class of admissible correctors from Lemma~\ref{lem:first-order-improvement},
selected to maximize tractability and to make the constants fully explicit.

\begin{proposition}[Explicit first-order improvement]
\label{prop:MW-explicit-M}
Fix $\gamma>0$ and the dimension $d\in\mathbb N$.
Let $\lambda\in(0,1/4]$ be the parameter in Assumption~\ref{assump:potential}.
Consider the multi-well potential $U(q)=\sum_{i=1}^d v(q_i)$ in
\eqref{eq:MW-high-dim}--\eqref{eq:MW-1D-component}, for which $L=1$ and $\nabla U(0)=0$.
Define the quadratic corrector $\mathcal M:\R^{2d}\to\R$ by
\begin{equation}\label{eq:M-explicit}
  \mathcal{M}(q,p)
  :=\frac{2+\gamma^2}{4\gamma}|q|^2+\frac{1}{2\gamma}|p|^2.
\end{equation}
Let
\[
  B:=1+\frac{\gamma^2}{2}(1-\lambda),
  \qquad
  c_{\mathrm{imp}}
  :=\frac{2\sqrt{B}}{2\sqrt{B}+\gamma}\in(0,1).
\]
Then the following holds.

\begin{enumerate}[label=(\roman*)]
\item \textbf{First-order improvement inequality.}
There exists a constant $C_{\mathrm{imp}}^{(d)}<\infty$ (scaling at most linearly in $d$)
such that for all $(q,p)\in\R^{2d}$,
\begin{equation}\label{eq:MW-first-order-ineq}
  \mathcal A_0\mathcal M(q,p)+\mathcal A'\mathcal V_0(q,p)
  \le -c_{\mathrm{imp}}\,\mathcal V_0(q,p)+C_{\mathrm{imp}}^{(d)}.
\end{equation}
In particular, one may take $C_{\mathrm{imp}}^{(d)}=C_{\mathrm{imp}}^{\mathrm{MW}}(\gamma,\lambda)\,d$
for an explicit constant $C_{\mathrm{imp}}^{\mathrm{MW}}(\gamma,\lambda)$.

\item \textbf{Quadratic lower bound and the uniform growth constant $\widetilde C_{\mathcal M}^{\mathrm{MW}}$.}
Define
\begin{equation}\label{eq:c1MW-def}
  c_1^{\mathrm{MW}}
  :=\frac18\left(\gamma^2(1-\lambda)+2-\sqrt{(\gamma^2(1-\lambda)-2)^2+4\gamma^2}\right)>0,
\end{equation}
so that
\begin{equation}\label{eq:V0-lower-MW}
  \mathcal V_0(q,p)\ge c_1^{\mathrm{MW}}\,(|q|^2+|p|^2),
  \qquad (q,p)\in\R^{2d}.
\end{equation}
Moreover, with
\begin{equation}\label{eq:CMtilde-MW}
  \widetilde C_{\mathcal M}^{\mathrm{MW}}
  :=\frac{2+\gamma^2}{4\gamma\,c_1^{\mathrm{MW}}},
\end{equation}
we have the pointwise bound
\begin{equation}\label{eq:M-growth-MW-tilde}
  |\mathcal M(q,p)|
  \le \widetilde C_{\mathcal M}^{\mathrm{MW}}\,\mathcal V_0(q,p)
  \le \widetilde C_{\mathcal M}^{\mathrm{MW}}\,(1+\mathcal V_0(q,p)).
\end{equation}

\item \textbf{Second-order remainder and the drift-rate expansion.}
Let 
\begin{equation}\label{eq:err-def}
\mathrm{Err}^{(d)}(q,p):=|\mathcal A'\mathcal M(q,p)|+|\Delta_q\mathcal M(q,p)|.
\end{equation}
Then there exist explicit constants $C_2^{\mathrm{MW}}\ge 0$ (dimension-free) and
$C_2^{(d),\mathrm{MW}}= \mathcal O(d)$ such that
\begin{equation}\label{eq:ErrMW-bound}
  \mathrm{Err}^{(d)}(q,p)\le C_2^{\mathrm{MW}}\,\mathcal V_0(q,p)+C_2^{(d),\mathrm{MW}},
  \qquad (q,p)\in\R^{2d}.
\end{equation}
In particular, one may take
\[
  C_2^{\mathrm{MW}}:=2\widetilde C_{\mathcal M}^{\mathrm{MW}},
  \qquad
  C_2^{(d),\mathrm{MW}}
  :=\frac{2+\gamma^2}{2\gamma}\,d.
\]

Finally, define
\begin{equation}\label{eq:deltaMW-ClambdaMW}
  \delta_{\mathrm{MW}}
  :=c_{\mathrm{imp}}-\lambda\,\widetilde C_{\mathcal M}^{\mathrm{MW}},
  \qquad
  C_{\lambda,\mathrm{MW}}
  :=C_2^{\mathrm{MW}}+\widetilde C_{\mathcal M}^{\mathrm{MW}}\,c_{\mathrm{imp}}.
\end{equation}
Then Proposition~\ref{prop:Valpha-drift} applies (for $\lambda$ sufficiently small, if needed by the
baseline constants), and the improved drift rate satisfies
\begin{equation}\label{eq:lambda-alpha-MW}
  \lambda_\alpha \;\ge\; \lambda+\delta_{\mathrm{MW}}\cdot\alpha-C_{\lambda,\mathrm{MW}}\cdot\alpha^2.
\end{equation}
In particular, $\lambda_\alpha>\lambda$ for all sufficiently small $\alpha>0$ whenever $\delta_{\mathrm{MW}}>0$.
\end{enumerate}
\end{proposition}

\begin{proof}
    We provide the proof in Appendix~\ref{app:MW-explicit-M}.
\end{proof}

To apply the quantitative acceleration result of Corollary~\ref{cor:global-acceleration}
to the multi-well model, one needs the one-dimensional condition $\delta_{\mathrm{MW}}>\gamma\lambda$.
The next lemma shows that this condition is not restrictive: for any fixed $\gamma>0$,
it can be enforced by choosing the dissipativity parameter $\lambda$ sufficiently small.

\begin{lemma}[Feasibility of the quantitative condition $\delta_{\mathrm{MW}}>\gamma\lambda$]
\label{lem:MW-feasibility}
Fix $\gamma>0$ and consider the one-dimensional double-well potential $v:\R\to\R$
defined in \eqref{eq:MW-1D-component}. Then Assumption~\ref{assump:potential}(iii)
(dissipativity) holds for \emph{any} $\lambda\in(0,1/4]$ (with an additive constant
depending on $\lambda$). Moreover, with $\delta_{\mathrm{MW}}$ defined in
\eqref{eq:deltaMW-ClambdaMW},
there exists $\lambda_\star(\gamma)\in(0,1/4]$ such that for every
$\lambda\in(0,\lambda_\star(\gamma)]$ we have
\[
  \delta_{\mathrm{MW}}>\gamma\lambda .
\]
\end{lemma}

\begin{proof}
  We provide the proof in Appendix~\ref{app:MW-feasibility}.
\end{proof}

Based on the explicit construction of the quadratic corrector $\mathcal{M}$ in Proposition~\ref{prop:MW-explicit-M}, we state the main acceleration result for the $d$-dimensional multi-well potential.
Before we proceed, let $\rho_{\alpha,1}$ be the one-dimensional cost used in Corollary~\ref{cor:global-acceleration},
and define the tensorized cost on $\R^{2d}$ by
\[
  \rho_{\alpha,d}(z,z'):=\sum_{i=1}^d \rho_{\alpha,1}(z_i,z_i'),
  \qquad z_i=(q_i,p_i)\in\R^2 .
\]
Let $\cW_{\rho_{\alpha,d}}$ be the Wasserstein distance induced by $\rho_{\alpha,d}$.
Then, we have the following result.

\begin{theorem}[HFHR acceleration for a multi-well potential]
\label{thm:HFHR-MW-high-dim}
Consider the separable potential $U(q)=\sum_{i=1}^d v(q_i)$ in \eqref{eq:MW-high-dim},
and let $P_t^{\alpha,(d)}$ be the semigroup of the corresponding $d$-dimensional HFHR dynamics.
Fix $\gamma>0$ and choose $\lambda\in(0,\lambda_\star(\gamma)]$ as in
Lemma~\ref{lem:MW-feasibility}, so that $\delta_{\mathrm{MW}}>\gamma\lambda$.
Assume in addition that the remaining one-dimensional quantitative conditions of
Corollary~\ref{cor:global-acceleration} hold for the multi-well model
(in particular, $\Lambda_0>1/2$).
Then there exist explicit constants $\alpha_{\mathrm{MW}}>0$ and $\kappa_{\mathrm{MW}}>0$
(independent of $d$), depending only on the one-dimensional double-well model
(and on $\gamma$), such that for every $d\ge 1$ and every $\alpha\in(0,\alpha_{\mathrm{MW}}]$,
\[
  \cW_{\rho_{\alpha,d}}\!\left(\mu P_t^{\alpha,(d)},\nu P_t^{\alpha,(d)}\right)
  \le e^{-(c_0+\kappa_{\mathrm{MW}}\alpha)t}\,\cW_{\rho_{\alpha,d}}(\mu,\nu),
  \qquad t\ge0,
\]
for all probability measures $\mu,\nu$ on $\R^{2d}$,
where $c_0$ denotes the \emph{one-dimensional} kinetic Langevin contraction rate at $\alpha=0$
associated with the cost $\rho_{0,1}$.
Moreover, one may choose explicitly
\[
  \alpha_{\mathrm{MW}}
  := \min\left\{ \alpha_{\mathrm{branch,acc}}^{(1)},\ \alpha_{\mathrm{metric,acc}}^{(1)} \right\},
\]
where $\alpha_{\mathrm{branch,acc}}^{(1)}$ and $\alpha_{\mathrm{metric,acc}}^{(1)}$ are the explicit thresholds from Theorem~\ref{thm:lyapunov-acceleration} and Theorem~\ref{thm:metric-acceleration} respectively, evaluated for the one-dimensional model (using $L=1$, $\delta=\delta_{\mathrm{MW}}$ and $C_\lambda=C_{\lambda,\mathrm{MW}}$).
Similarly, the explicit gain is given by
\[
  \kappa_{\mathrm{MW}}
  := \kappa_{\mathrm{global}}^{(1)}
  = \min\left\{\kappa^{(1)},\ c_0\,c_2^{(1)},\ c_0\,c_3^{(1)}\right\},
\]
where $\kappa^{(1)}$ is the Lyapunov-branch gain from
Theorem~\ref{thm:lyapunov-acceleration} in dimension $1$,
and $c_2^{(1)},c_3^{(1)}$ are the metric-branch improvement constants from
Theorem~\ref{thm:metric-acceleration} in dimension $1$.

In particular, with the explicit corrector from Proposition~\ref{prop:MW-explicit-M},
one can take
\[
  \kappa^{(1)}=\frac{L(\delta_{\mathrm{MW}}+\gamma\lambda)}{768\,\gamma},
\]
(with $L=1$ for the multi-well model), where
\[
  \delta_{\mathrm{MW}}:=c_{\mathrm{imp}}-\lambda\,\widetilde C_{\mathcal M}^{\mathrm{MW}},
  \qquad
  C_{\lambda,\mathrm{MW}}:=C_2^{\mathrm{MW}}+\widetilde C_{\mathcal M}^{\mathrm{MW}}\,c_{\mathrm{imp}}.
\]
\end{theorem}

\begin{proof}
    We provide the proof in Appendix~\ref{app:HFHR-MW-high-dim}.
\end{proof}



\subsection{Bayesian linear regression}\label{sec:case:L:p}

In this section, we study the example of
a Bayesian linear regression problem. Given the input data $X \in \R^{n \times d}$, and the output data $y \in \R^n$, we consider the following objective function $U: \mathbb{R}^d \rightarrow \mathbb{R}$ with a regularizer function $g: \mathbb{R}^d \rightarrow \mathbb{R}$ in the Bayesian linear regression task \cite{hoff2009first}:
\begin{align}
U(q) = \frac{1}{2\sigma^2} | y - Xq |^2 + g(q),
\label{eq:linear}
\end{align}
such that
$\nabla U(q) = -\frac{X^\top(y - Xq)}{\sigma^2} + \nabla g(q)$,
where parameters $\sigma > 0$.
In particular, we consider $L^p$ regularization. Our use of smoothed $L^p$ regularization can be interpreted from a Bayesian perspective as imposing a prior on the regression coefficients. Such priors interpolate between Gaussian and Laplace distributions when $p<2$;
see \cite{polson2010shrink,gong2013general} for a comprehensive overview. Moreover, the Bayesian Lasso in~\cite{park2008bayesian,polson2014bayesian} arises as a special case corresponding to $p\rightarrow 1$.
We take the regularizer function $g$ as the following $L^p$ function:
\begin{equation}
\label{eq:lp}
g(q) := \iota \sum_{i=1}^d (q_i^2 + \varepsilon^2)^{p/2}, \qquad 1 < p < 2,
\end{equation}
where $\iota > 0$ is the regularization parameter and $\varepsilon^2 > 0$ is a self-tuning parameter. Since $1<p<2$, the regularizer $g(q)$ and hence the potential $U(q)$ is non-convex in general.
We make the following assumption.

\begin{assumption}\label{assump:linear}
Assume that $X^{\top}X \succ m I_d$ for some $m > 0$.
\end{assumption}

Note that Assumption~\ref{assump:linear} is mild and often imposed in the literature; see for example Assumption~9 in~\cite{mei2018landscape}.
Next, we show that under Assumption~\ref{assump:linear}, the Bayesian linear regression problem \eqref{eq:linear} with $L^{p}$ regularizer \eqref{eq:lp} satisfies
both Assumptions~\ref{assump:potential} and \ref{assump:asymptotic-linear-drift} required for our theory.

\begin{proposition}[Bayesian linear regression with smoothed $L^p$ regularizer satisfies the standing assumptions]
\label{prop:linear-lp-assumptions}
Fix $1<p<2$, $\sigma>0$, $\iota>0$, and $\varepsilon>0$.
Let $X\in\mathbb R^{n\times d}$, $y\in\mathbb R^n$ and define $U$ as in \eqref{eq:linear}
with the $L^p$ regularizer $g(q)$ in \eqref{eq:lp}.
Assume that Assumption~\ref{assump:linear} holds.
Denote $M:=\|X^\top X\|_{\mathrm{op}}$.
Then:

\begin{enumerate}
\item[\textnormal{(a)}] $U\in C^\infty(\mathbb R^d)$ and $U\geq 0$. Moreover, $\nabla U$ is $L$-Lipschitz with
$L:=\frac{M}{\sigma^2}+\iota\,p\,\varepsilon^{p-2}$.

\item[\textnormal{(b)}] $U$ is dissipative in the sense of Assumption~\ref{assump:potential}-(iii). In particular,
\[
\langle \nabla U(q),q\rangle
\ge \frac{m}{2\sigma^2}|q|^2-\frac{|X^\top y|^2}{2m\sigma^2}
\quad\text{for all }q\in\mathbb R^d.
\]

\item[\textnormal{(c)}] $U$ satisfies Assumption~\ref{assump:asymptotic-linear-drift} with
$Q_\infty:=\frac{1}{\sigma^2}X^\top X$,
and the function
\[
\varrho(r):=\frac{c^{\mathrm{LR}}_0}{r}+c^{\mathrm{LR}}_1\, r^{p-2},
\qquad r\ge C_{\mathrm{linear}},
\]
where one may take e.g. $C_{\mathrm{linear}}:=1$ and
\begin{equation}\label{eq:c0lrc1lr-def}
    c^{\mathrm{LR}}_0:=\frac{|X^\top y|}{\sigma^2}+\iota p\sqrt d\,\varepsilon^{p-1}, \qquad
    c^{\mathrm{LR}}_1:=\iota p\, d^{\frac{2-p}{2}}.
\end{equation}
\end{enumerate}
\end{proposition}

\begin{proof}
    We provide the proof in Appendix~\ref{app:linear-lp-assumptions}.
\end{proof}

In contrast to the separable multi-well model (where $\nabla U(q)-Q_\infty q$ is uniformly bounded),
the smoothed $L^p$ regularizer yields a \emph{sublinear but unbounded} remainder
$|\nabla U(q)-Q_\infty q|=\mathcal O(|q|^{p-1})$ as $|q|\to\infty$.
Therefore, we obtain explicit acceleration constants by invoking the general first-order improvement
Lemma~\ref{lem:first-order-improvement} (Lyapunov-equation corrector), together with explicit bounds on
$\rho_\nabla(\cdot)$ and $\delta_U(\cdot)$ specialized to \eqref{eq:linear}--\eqref{eq:lp}.

\begin{proposition}[Explicit constants for Lemma~\ref{lem:first-order-improvement} in Bayesian linear regression]
\label{prop:linear-lp-first-order}
Assume the setting of Proposition~\ref{prop:linear-lp-assumptions} and let $\lambda$ be the
dissipativity parameter in Assumption~\ref{assump:potential}.
Let
$Q_\infty:=\frac{1}{\sigma^2}X^\top X$,
$b:=\frac{1}{\sigma^2}X^\top y$,
and $M:=\|X^\top X\|_{\mathrm{op}}$.
Then the following hold.

\begin{enumerate}[label=(\roman*)]
\item \textbf{Spectral bounds for $Q_\infty$.}
\[
  \lambda_{\min}(Q_\infty)=\frac{1}{\sigma^2}\lambda_{\min}\left(X^\top X\right)\ge \frac{m}{\sigma^2},
  \qquad
  \lambda_{\max}(Q_\infty)=\frac{1}{\sigma^2}\lambda_{\max}\left(X^\top X\right)=\frac{M}{\sigma^2}.
\]

\item \textbf{Explicit tail moduli $\rho_\nabla$ and $\delta_U$.}
With $c^{\mathrm{LR}}_0,c^{\mathrm{LR}}_1$ as in \eqref{eq:c0lrc1lr-def}, we have for any $R\ge 1$,
\begin{equation}\label{eq:LR-rhonabla-bound}
  \rho_\nabla(R):=\sup_{|q|\ge R}\frac{|\nabla U(q)-Q_\infty q|}{|q|}\le \frac{c^{\mathrm{LR}}_0}{R}+c^{\mathrm{LR}}_1\,R^{p-2},
\end{equation}
and for any $R\ge 1$,
\begin{equation}\label{eq:LR-deltaU-bound}
  \delta_U(R)
  :=\sup_{|q|\ge R}\frac{\left|U(q)-\frac12\langle Q_\infty q,q\rangle\right|}{1+|q|^2}
  \le
  \frac{|b|}{R}
  + \iota\,d^{1-\frac p2}\,R^{p-2}
  + \frac{\iota\,d\,\varepsilon^p+\frac{1}{2\sigma^2}|y|^2}{R^2}.
\end{equation}

\item \textbf{An explicit admissible cutoff radius $R_0$.}
Let $\mathsf K$ be the Lyapunov-equation matrix from Lemma~\ref{lem:first-order-improvement}
(defined below), and let $\rho_\star$ be as in \eqref{eq:R0-def-lemma}.
Since $c^{\mathrm{LR}}_0,c^{\mathrm{LR}}_1\ge 0$ and $p-2<0$, the right-hand side of \eqref{eq:LR-rhonabla-bound}
is decreasing in $R$. Therefore the choice
\begin{equation}\label{eq:LR-R0-choice}
  R_0:=\max\left\{1,\ C_{\mathrm{linear}},\ \frac{c^{\mathrm{LR}}_0}{\rho_\star},\
  \left(\frac{c^{\mathrm{LR}}_1}{\rho_\star}\right)^{\!\frac{1}{2-p}}\right\}
\end{equation}
ensures $\rho_\nabla(R_0)\le \rho_\star$, hence \eqref{eq:R0-choice-lemma} holds.

\item \textbf{Quadratic corrector via the same Lyapunov equation as the general theory.}
Let
\[
  B:=\begin{pmatrix}0&I_d\\-Q_\infty&-\gamma I_d\end{pmatrix}
\]
be the linearized kinetic Langevin drift matrix at infinity, and let $C_{B_1}$ be the explicit symmetric
matrix $\nabla^2 B_1$ from Lemma~\ref{lem:first-order-improvement}. Then the corrector can be chosen as
\[
  \mathcal M(z)=\frac12 z^\top \mathsf K z,
  \qquad z=(q,p)\in\R^{2d},
\]
where $\mathsf K$ is the (unique) symmetric solution to
\[
  B^\top \mathsf K+\mathsf K B=C_{B_1},
\]
equivalently given by the integral representation
\begin{equation}\label{eq:LR-K-integral}
  \mathsf K=\int_0^\infty e^{tB^\top}\,C_{B_1}\,e^{tB}\,dt .
\end{equation}

\item \textbf{First-order improvement constant $\underline c_{\mathrm{imp}}$ and an explicit upper bound for $C_{\mathrm{imp}}$.}
Let
\[
  a_{\min}:=\lambda_{\min}(Q_\infty)+\frac{\gamma^2}{2}(1-\lambda),
  \qquad
  a_{\max}:=\lambda_{\max}(Q_\infty)+\frac{\gamma^2}{2}(1-\lambda),
\]
so that by (i),
\[
  a_{\min}\ge \frac{m}{\sigma^2}+\frac{\gamma^2}{2}(1-\lambda),
  \qquad
  a_{\max}\le \frac{M}{\sigma^2}+\frac{\gamma^2}{2}(1-\lambda).
\]
Then Lemma~\ref{lem:first-order-improvement} yields the improvement inequality
\eqref{eq:improvement-condition} with $\underline c_{\mathrm{imp}}$ chosen as in \eqref{eq:cimp-lemma},
where $\delta_U(R_0)$ can be bounded explicitly by \eqref{eq:LR-deltaU-bound} (with $R=R_0$ from
\eqref{eq:LR-R0-choice}). In particular, one obtains a fully explicit positive lower bound
$\underline c_{\mathrm{imp}}>0$ in terms of $(m,M,\sigma,\lambda,p,\varepsilon,|y|,|X^\top y|,\gamma)$.
Moreover, the corresponding constant $C_{\mathrm{imp}}$ from Lemma~\ref{lem:first-order-improvement} is finite.

\item \textbf{A convenient explicit remainder bound for drift-rate expansion.}
Write $\mathsf K$ in block form $\mathsf K=\left(\begin{smallmatrix}\mathsf K_{qq}&\mathsf K_{qp}\\
\mathsf K_{pq}&\mathsf K_{pp}\end{smallmatrix}\right)$ and set
\[
  k_q:=\|\mathsf K_{qq}\|_{\mathrm{op}}+\|\mathsf K_{qp}\|_{\mathrm{op}},
  \qquad
  b_0:=|\nabla U(0)|=|b|.
\]
Using $|\nabla U(q)|\le L|q|+b_0$ (with $L$ from Proposition~\ref{prop:linear-lp-assumptions}(a)),
one obtains for all $(q,p)\in\R^{2d}$,
\[
  |\mathcal A'\mathcal M(q,p)|
  \le k_q\left(\frac{3L+1}{2}|q|^2+\frac{L+1}{2}|p|^2+b_0^2\right),
  \qquad
  |\Delta_q\mathcal M(q,p)|=\mathrm{tr}(\mathsf K_{qq})\le d\,\|\mathsf K_{qq}\|_{\mathrm{op}}.
\]
Let $c_1>0$ be the (explicit) quadratic lower bound constant such that
$\mathcal V_0(q,p)\ge c_1(|q|^2+|p|^2)$ after shifting $U$ by an additive constant if needed.
Then the ``error term'' $\mathrm{Err}^{(d)}(q,p):=|\mathcal A'\mathcal M(q,p)|+|\Delta_q\mathcal M(q,p)|$
satisfies
\begin{equation}\label{eq:LR-Err-bound}
  \mathrm{Err}^{(d)}(q,p)
  \le C_2^{\mathrm{LR}}\ \mathcal V_0(q,p) + C_2^{(d),\mathrm{LR}},
\end{equation}
with the explicit choices
\[
  C_2^{\mathrm{LR}}
  :=\frac{k_q}{c_1}\max\left\{\frac{3L+1}{2},\frac{L+1}{2}\right\},
  \qquad
  C_2^{(d),\mathrm{LR}}
  :=k_q b_0^2 + d\,\|\mathsf K_{qq}\|_{\mathrm{op}}.
\]

Finally, define the (explicit) growth constant
\[
  \widetilde C_{\mathcal M}^{\mathrm{LR}}:=\frac{\|\mathsf K\|_{\mathrm{op}}}{2c_1}
\]
(cf.\ \eqref{eq:widetildeCM-def}), and set
\begin{equation}\label{eq:LR-delta-and-Clambda}
  \delta_{\mathrm{LR}}
  :=\underline c_{\mathrm{imp}}-\lambda\,\widetilde C_{\mathcal M}^{\mathrm{LR}},
  \qquad
  C_{\lambda,\mathrm{LR}}
  :=C_2^{\mathrm{LR}}+\widetilde C_{\mathcal M}^{\mathrm{LR}}\,\underline c_{\mathrm{imp}}.
\end{equation}
Then Proposition~\ref{prop:Valpha-drift} applies (for $\alpha$ sufficiently small) and yields the drift-rate expansion
\[
  \lambda_\alpha \;\ge\; \lambda+\delta_{\mathrm{LR}}\cdot\alpha-C_{\lambda,\mathrm{LR}}\cdot\alpha^2.
\]
In particular, $\lambda_\alpha>\lambda$ for all sufficiently small $\alpha>0$ whenever $\delta_{\mathrm{LR}}>0$.
\end{enumerate}
\end{proposition}

\begin{proof}
    We provide the proof in Appendix~\ref{app:linear-lp-first-order}.
\end{proof}

Finally, we confirm that the quantitative condition $\delta_{\mathrm{LR}}>\gamma\lambda$ required for acceleration can be satisfied by choosing the dissipativity parameter $\lambda$ small enough.

\begin{lemma}[Feasibility of $\delta_{\mathrm{LR}}>\gamma\lambda$ with an explicit $\lambda_\star(\gamma)$]
\label{lem:LR-feasibility-explicit-full}
Fix $\gamma>0$ and consider the Bayesian linear regression potential \eqref{eq:linear}--\eqref{eq:lp}
under Assumption~\ref{assump:linear}. Let
\[
  Q_\infty:=\frac{1}{\sigma^2}X^\top X,
  \quad
  m_\infty:=\lambda_{\min}(Q_\infty)=\frac{\lambda_{\min}(X^\top X)}{\sigma^2}\ge \frac{m}{\sigma^2},
  \quad
  M_\infty:=\lambda_{\max}(Q_\infty)=\frac{\lambda_{\max}(X^\top X)}{\sigma^2}.
\]
Set
\[
  \bar\lambda:=\min\left\{\frac14,\ \frac{m}{2\sigma^2}\right\},
  \qquad
  R:=\max\{1,C_{\mathrm{linear}}\}.
\]
Define the explicit tail bound
\begin{equation}\label{eq:LR-deltaU-plus-lemma}
  \delta_U^{+}
  :=
  \frac{|X^\top y|}{\sigma^2\,R}
  + \iota\,d^{1-\frac p2}\,R^{p-2}
  + \frac{\iota\,d\,\varepsilon^p+\frac{1}{2\sigma^2}|y|^2}{R^2},
\end{equation}
and the spectral proxies
\begin{equation}\label{eq:LR-amin-amax-proxy-lemma}
  a_{\min}^{-}:=m_\infty+\frac{\gamma^2}{2}(1-\bar\lambda),
  \qquad
  a_{\max}^{+}:=M_\infty+\frac{\gamma^2}{2}.
\end{equation}
Let $c_1=c_1(\gamma,\bar\lambda)>0$ be the explicit quadratic lower bound constant of the baseline Lyapunov function
$\mathcal V_0$ (up to an additive constant), namely
\begin{equation}\label{eq:LR-c1-def-lemma}
  c_1
  :=
  \frac18\left(\gamma^2(1-\bar\lambda)+2-\sqrt{(\gamma^2(1-\bar\lambda)-2)^2+4\gamma^2}\right).
\end{equation}

Let $B:=\left(\begin{smallmatrix}0&I_d\\-Q_\infty&-\gamma I_d\end{smallmatrix}\right)$ be the linear drift matrix at infinity.
Define
\begin{equation}\label{eq:LR-eta-def-lemma}
  \eta:=\frac{\gamma-\sqrt{(\gamma^2-4m_\infty)_+}}{2}>0,
  \qquad
  C_B:=1+\frac{\gamma}{2\sqrt{m_\infty}}+\sqrt{M_\infty}+\frac{1}{\sqrt{m_\infty}}.
\end{equation}
Finally, define
\begin{equation}\label{eq:LR-CB1-plus-lemma}
  C_{B_1}^{+}
  :=2\left(1+\gamma+M_\infty+\frac{\gamma^2}{2}\right),
  \qquad
  \widetilde C_{\mathcal M}^{+}
  :=
  \frac{1}{2c_1}\cdot \frac{C_B^2}{2\eta}\,C_{B_1}^{+},
\end{equation}
and the explicit lower bound
\begin{equation}\label{eq:LR-cimp-lower-explicit-lemma}
  \underline c_{\mathrm{imp}}^{-}
  :=
  \frac{3}{8}\cdot
  \frac{
    a_{\min}^{-}+1-\sqrt{(a_{\min}^{-}-1)^2+\gamma^2}
  }{
    a_{\max}^{+}+1+\sqrt{(a_{\max}^{+}-1)^2+\gamma^2}
    \;+\;
    8\,\delta_U^{+}
  }.
\end{equation}
Then the explicit choice
\begin{equation}\label{eq:LR-lambda-star-explicit-lemma}
  \lambda_\star(\gamma)
  :=
  \min\left\{
    \bar\lambda,\ 
    \frac{\underline c_{\mathrm{imp}}^{-}}{\gamma+\widetilde C_{\mathcal M}^{+}}
  \right\}
\end{equation}
leads to the following properties.

\begin{enumerate}[label=(\roman*)]
\item For any $\lambda\in(0,\bar\lambda]$, the dissipativity inequality in
Proposition~\ref{prop:linear-lp-assumptions}(b) implies Assumption~\ref{assump:potential}(iii) with this $\lambda$
(up to an additive constant depending on $\lambda$).

\item For every $\lambda\in(0,\lambda_\star(\gamma)]$, the quantitative condition
\[
  \delta_{\mathrm{LR}}>\gamma\lambda
\]
holds, where
\[
  \delta_{\mathrm{LR}}
  :=
  \underline c_{\mathrm{imp}}(\lambda)-\lambda\,\widetilde C_{\mathcal M}^{\mathrm{LR}}(\lambda),
\]
and $\underline c_{\mathrm{imp}}(\lambda)$, $\widetilde C_{\mathcal M}^{\mathrm{LR}}(\lambda)$ are the constants
appearing in Lemma~\ref{lem:first-order-improvement} specialized to the present model.
\end{enumerate}
\end{lemma}

\begin{proof}
    We provide the proof in Appendix~\ref{app:LR-feasibility-explicit-full}.
\end{proof}

Combining Proposition~\ref{prop:linear-lp-first-order} with Corollary~\ref{cor:global-acceleration}
(in dimension $1$, followed by tensorization if desired) yields explicit constants
$\alpha_{\mathrm{LR}}>0$ and $\kappa_{\mathrm{LR}}>0$ such that the HFHR contraction rate satisfies
$c_\alpha\ge c_0^{\mathrm{LR}}+\kappa_{\mathrm{LR}}\alpha$ for all $\alpha\in(0,\alpha_{\mathrm{LR}}]$,
whenever $\lambda\in(0,\lambda_\star(\gamma)]$.

\begin{theorem}[HFHR acceleration for Bayesian linear regression]
\label{thm:HFHR-LR}
Consider the Bayesian linear regression problem defined by \eqref{eq:linear}--\eqref{eq:lp} under Assumption~\ref{assump:linear}.
Let $P_t^{\alpha}$ be the semigroup of the corresponding HFHR dynamics.
Let $\rho_{\mathcal V_\alpha}$ be the Lyapunov-weighted semimetric used in Corollary~\ref{cor:global-acceleration} (constructed using the global Lipschitz constant $L$ and the Lyapunov function $\mathcal V_\alpha$), and let $\cW_{\rho_{\mathcal V_\alpha}}$ be the associated Wasserstein distance.

Fix $\gamma>0$ and choose the dissipativity parameter $\lambda\in(0,\lambda_\star(\gamma)]$ as in Lemma~\ref{lem:LR-feasibility-explicit-full}, so that $\delta_{\mathrm{LR}}>\gamma\lambda$.
Assume in addition that the quantitative conditions of Corollary~\ref{cor:global-acceleration} hold (in particular, $\Lambda_0>1/2$).
Then there exist explicit constants $\alpha_{\mathrm{LR}}>0$ and $\kappa_{\mathrm{LR}}>0$, depending on the model parameters ($X, y, \sigma, \iota, p, \varepsilon$) and $\gamma$, such that for every $\alpha\in(0,\alpha_{\mathrm{LR}}]$,
\[
  \cW_{\rho_{\mathcal V_\alpha}}\!\left(\mu P_t^{\alpha}, \nu P_t^{\alpha}\right)
  \;\le\; e^{-(c_0^\mathrm{LR} + \kappa_{\mathrm{LR}}\alpha)t}\,\cW_{\rho_{\mathcal V_\alpha}}(\mu, \nu),
  \qquad t \ge 0,
\]
for all probability measures $\mu, \nu$ with finite Lyapunov moments, where $c_0^\mathrm{LR}$ denotes the contraction rate of the kinetic Langevin dynamics at $\alpha=0$.
Moreover, one may choose explicitly
\[
  \alpha_{\mathrm{LR}}
  := \min\left\{ \alpha_{\mathrm{branch,acc}}^{\mathrm{LR}},\ \alpha_{\mathrm{metric,acc}}^{\mathrm{LR}} \right\},
\]
where $\alpha_{\mathrm{branch,acc}}^{\mathrm{LR}}$ and $\alpha_{\mathrm{metric,acc}}^{\mathrm{LR}}$ are the explicit thresholds from Theorem~\ref{thm:lyapunov-acceleration} and Theorem~\ref{thm:metric-acceleration} respectively, evaluated using the global constants from Proposition~\ref{prop:linear-lp-assumptions} and Proposition~\ref{prop:linear-lp-first-order}.
Similarly, the explicit gain is given by
\[
  \kappa_{\mathrm{global}}^{\mathrm{LR}}
  := \min\left\{ \kappa_{\mathrm{LR}},\ c_0^{\mathrm{LR}}\,c_2^{\mathrm{LR}},\ c_0^{\mathrm{LR}}\,c_3^{\mathrm{LR}} \right\},
\]
where $\kappa_{\mathrm{LR}}$ is the Lyapunov-branch gain from Theorem~\ref{thm:lyapunov-acceleration}, and $c_2^{\mathrm{LR}}, c_3^{\mathrm{LR}}$ are the metric-branch improvement constants from Theorem~\ref{thm:metric-acceleration}.
In particular, utilizing the explicit constants from Proposition~\ref{prop:linear-lp-first-order},
\[
  \kappa_{\mathrm{LR}} = \frac{L(\delta_{\mathrm{LR}} + \gamma\lambda)}{768\,\gamma},
\]
where
\[
  \delta_{\mathrm{LR}} := \underline c_{\mathrm{imp}} - \lambda\,\widetilde C_{\mathcal M}^{\mathrm{LR}},
  \qquad
  C_{\lambda,\mathrm{LR}} := C_2^{\mathrm{LR}} + \widetilde C_{\mathcal M}^{\mathrm{LR}}\,\underline c_{\mathrm{imp}},
\]
and $L = \frac{\|X^\top X\|_{\mathrm{op}}}{\sigma^2} + \iota p \varepsilon^{p-2}$ is the global Lipschitz constant from Proposition~\ref{prop:linear-lp-assumptions}.
\end{theorem}

\begin{proof}
We provide the proof in Appendix~\ref{proof:thm:HFHR-LR}.    
\end{proof}

\subsection{Bayesian binary classification}\label{sec:case:classification}

In this section, we consider a Bayesian formulation of a binary classification task with data $\left\{\left(x_i, y_i\right)\right\}_{i=1}^n$, where $x_i \in \mathbb{R}^d$ are feature vectors with $\vert x_i \vert < \infty$ and $y_i \in\{0,1\}$ are labels. In a classification task, our aim is to learn a predictive model of the form
$\mathbb{P}(y_i = 1 \mid x_i,q) = h(\langle q, x_i \rangle).$
Let $h: \mathbb{R} \rightarrow \mathbb{R}$ be a prediction function and $\varphi: \mathbb{R} \rightarrow \mathbb{R}_{+}$be a loss function. 
In this setting, we can write the potential function $U: \mathbb{R}^d \rightarrow \mathbb{R}$ as
\begin{equation}
\label{eq:classification}
U(q)=\frac{1}{n} \sum_{i=1}^n \varphi\!\left(y_i-h(\langle q,x_i\rangle)\right)+\frac{\iota}{2}\vert q \vert^2,
\end{equation}
and the associated sampling target is the Gibbs posterior
$\pi(q) \propto \exp (-U(q))$. After iterating $K$ steps for our samples, the classifier can be formulated as 
$
\hat{y} = \mathbf{1}_{\{h(\langle \bar{q}_K,\hat{x}\rangle) \geq 1/2\}} \in \{0,1\}
$, where $\bar{q}_K$ is the average over $M$ chains at $K$-th iterate, $\hat{x} \in \mathbb{R}^d$ is the given new feature for predicting $\hat{y} \in \{0,1\}$. 
We make the following assumptions, also see Assumption 12 in~\cite{GGZ}, Assumption 2 in~\cite{foster2018uniform} and Assumption 9 in~\cite{mei2018landscape}. 

\begin{assumption}\label{assump:classification}
Assume that the following conditions hold.
\begin{itemize}
\item $B_x:=\max_{1\le i\le n}|x_i|<\infty$.
\item $h \in C^2$ such that $H_{1}:=\sup_x\vert h'(x) \vert < \infty$ and $H_{2}:=\sup_x\vert h''(x) \vert < \infty$.
\item $\varphi \geq 0$ and $\varphi \in C^2$  such that $\Phi_{1}:=\sup_x\vert \varphi'(x) \vert < \infty$ and $\Phi_{2}:=\sup_x\vert \varphi''(x) \vert < \infty$.
\end{itemize}    
\end{assumption}

Assumption~\ref{assump:classification} is mild and can be satisfied for many choices of $\varphi,h$. For example, by following \cite{mei2018landscape,foster2018uniform,GGZ}, we consider $h(z) := z$ with Tukey's bisquare loss: 
\begin{equation}
\label{eq:tukey}
\varphi_{\text {Tukey }}(t):=
\begin{cases}
1-\left(1-\left(t / t_0\right)^2\right)^3 & \text {for $|t| \leq t_0$}, \\
1 & \text {for $|t|>t_0$}.
\end{cases}
\end{equation}
Then Assumption~\ref{assump:classification} is satisfied and the potential $U(q)$ is non-convex in general.
The non-convex examples of $\varphi$
that are either bounded or slowly growing near infinity have also been considered in~\cite{foster2018uniform,mei2018landscape}.
Next, we show that under Assumption~\ref{assump:classification}, the Bayesian binary classification problem \eqref{eq:classification} satisfies
both Assumptions~\ref{assump:potential} and \ref{assump:asymptotic-linear-drift} required for our theory.

\begin{proposition}[Bayesian binary classification potentials satisfy the standing assumptions]
\label{prop:classification-assumptions}
Consider the potential $U$ in \eqref{eq:classification}.
Assume Assumption~\ref{assump:classification} holds.
Then:

\begin{enumerate}
\item[\textnormal{(a)}] $U\in C^2(\mathbb R^d)$ and $U\ge 0$.
Moreover, $\nabla U$ is $L$-Lipschitz with
$L:=\iota + \left(\Phi_2 H_1^2+\Phi_1 H_2\right) B_x^2$.

\item[\textnormal{(b)}] $U$ is dissipative in the sense of Assumption~\ref{assump:potential}-(iii). In particular,
\[
\langle \nabla U(q),q\rangle
\ge \frac{\iota}{2}|q|^2 - \frac{C_0^2}{2\iota}
\quad\text{for all }q\in\mathbb R^d,
\]
where one may take
$C_0:=\Phi_1 H_1 B_x$.

\item[\textnormal{(c)}] $U$ satisfies Assumption~\ref{assump:asymptotic-linear-drift} with
\[
Q_\infty:=\iota I_d,
\qquad
\varrho(r):=\frac{C_0}{r},\quad r\ge 1.
\]
\end{enumerate}
\end{proposition}

\begin{proof}
    We provide the proof in Appendix~\ref{app:classification-assumptions}.
\end{proof}

In contrast to Bayesian linear regression with smoothed $L^p$ regularization (where
$|\nabla U(q)-Q_\infty q|=\mathcal O(|q|^{p-1})$ is unbounded), the present classification potentials satisfy a
\emph{uniformly bounded} remainder:
\[
  \nabla U(q)=\iota q+r(q),
  \qquad |r(q)|\le C_0:=\Phi_1H_1B_x .
\]
As a consequence, the tail moduli $\rho_\nabla(\cdot)$ and $\delta_U(\cdot)$ in
Lemma~\ref{lem:first-order-improvement} admit particularly simple explicit bounds, and one can obtain an explicit
range of $\lambda$ ensuring the quantitative metric-branch condition $\delta_{\mathrm{BC}}>\gamma\lambda$.

\begin{proposition}[Explicit constants for Lemma~\ref{lem:first-order-improvement} in Bayesian binary classification]
\label{prop:classification-first-order}
Assume the setting of Proposition~\ref{prop:classification-assumptions}, and let $\lambda\in(0,1/4]$ denote the
dissipativity parameter appearing in Assumption~\ref{assump:potential} (not the ridge coefficient $\iota$ in
\eqref{eq:classification}). Let
$Q_\infty:=\iota I_d$
and  
$C_0:=\Phi_1H_1B_x$.
Then the following hold.

\begin{enumerate}[label=(\roman*)]
\item \textbf{Spectral bounds for $Q_\infty$:}
$\lambda_{\min}(Q_\infty)=\lambda_{\max}(Q_\infty)=\iota$.

\item \textbf{Explicit tail modulus $\rho_\nabla$.}
For any $R\geq 1$,
\begin{equation}\label{eq:BC-rhonabla}
  \rho_\nabla(R):=\sup_{|q|\ge R}\frac{|\nabla U(q)-Q_\infty q|}{|q|}\le \frac{C_0}{R}.
\end{equation}

\item \textbf{Explicit tail modulus $\delta_U$.}
For any $R\ge 1$,
\begin{equation}\label{eq:BC-deltaU}
  \delta_U(R):=\sup_{|q|\ge R}\frac{\left|U(q)-\frac12\langle Q_\infty q,q\rangle\right|}{1+|q|^2}
  \le \frac{C_0}{R}+\frac{A_\varphi}{R^2},
\end{equation}
where $A_\varphi := |\varphi(0)|+\Phi_1(1+|h(0)|)$.

\item \textbf{An explicit admissible cutoff radius $R_0$.}
Let $\rho_\star$ be defined as in \eqref{eq:R0-def-lemma}. Since $R\mapsto C_0/R$ is decreasing,
the choice
\begin{equation}\label{eq:BC-R0}
  R_0:=\max\left\{1,\ C_{\mathrm{linear}},\ \frac{C_0}{\rho_\star}\right\}
\end{equation}
ensures $\rho_\nabla(R_0)\le \rho_\star$, and hence \eqref{eq:R0-choice-lemma} holds.

\item \textbf{Quadratic corrector via the same Lyapunov equation as the general theory.}
Let
\[
  B:=\begin{pmatrix}0&I_d\\-Q_\infty&-\gamma I_d\end{pmatrix}
  =\begin{pmatrix}0&I_d\\-\iota I_d&-\gamma I_d\end{pmatrix}
\]
and let $C_{B_1}$ be the explicit symmetric matrix $\nabla^2 B_1$ from Lemma~\ref{lem:first-order-improvement}.
Then one may take
$\mathcal M(z)=\frac12 z^\top \mathsf K z$, for any $z=(q,p)\in\R^{2d}$,
where $\mathsf K$ is the (unique) symmetric solution to
$
  B^\top\mathsf K+\mathsf K B=C_{B_1},
$
equivalently given by
\begin{equation}\label{eq:BC-K-integral}
  \mathsf K=\int_0^\infty e^{tB^\top}\,C_{B_1}\,e^{tB}\,dt.
\end{equation}

\item \textbf{First-order improvement constant $\underline c_{\mathrm{imp}}$ (explicit lower bound).}
With
\[
  a_{\min}=a_{\max}=:a(\lambda)
  =\iota+\frac{\gamma^2}{2}(1-\lambda),
\]
Lemma~\ref{lem:first-order-improvement} yields \eqref{eq:improvement-condition} with
\begin{equation}\label{eq:BC-cimp}
  \underline c_{\mathrm{imp}}
  :=
  \frac{3}{8}\cdot
  \frac{
    a(\lambda)+1-\sqrt{(a(\lambda)-1)^2+\gamma^2}
  }{
    a(\lambda)+1+\sqrt{(a(\lambda)-1)^2+\gamma^2}
    \;+\;
    8\,\delta_U(R_0)
  } \;>\;0,
\end{equation}
where $\delta_U(R_0)$ can be bounded explicitly using \eqref{eq:BC-deltaU} and $R_0$ from \eqref{eq:BC-R0}.
The corresponding $C_{\mathrm{imp}}$ in Lemma~\ref{lem:first-order-improvement} is finite.

\item \textbf{A drift-rate expansion bound for Proposition~\ref{prop:Valpha-drift}.}
Write $\mathsf K$ in block form
$\mathsf K=\left(\begin{smallmatrix}\mathsf K_{qq}&\mathsf K_{qp}\\
\mathsf K_{pq}&\mathsf K_{pp}\end{smallmatrix}\right)$ and set
\[
  k_q:=\|\mathsf K_{qq}\|_{\mathrm{op}}+\|\mathsf K_{qp}\|_{\mathrm{op}},
  \qquad
  b_0:=|\nabla U(0)|.
\]
Under Proposition~\ref{prop:classification-assumptions}, we have
$\nabla U(q)=\iota q+r(q)$ with $|r(q)|\le C_0$ for all $q$, hence $b_0\le C_0$ and
\[
  |\nabla U(q)|\le \iota |q|+C_0 \le L|q|+b_0,
\]
where $L=\iota + (\Phi_2H_1^2+\Phi_1H_2)B_x^2$ is the global Lipschitz constant from
Proposition~\ref{prop:classification-assumptions}(a).

Since $\mathcal M(z)=\frac12 z^\top \mathsf K z$, we have
$\nabla_q\mathcal M(q,p)=\mathsf K_{qq}q+\mathsf K_{qp}p$, hence
$|\nabla_q\mathcal M(q,p)|\le k_q(|q|+|p|)$.
Therefore
\[
  |\mathcal A'\mathcal M(q,p)|
  =|\langle \nabla U(q),\nabla_q\mathcal M(q,p)\rangle|
  \le k_q(L|q|+b_0)(|q|+|p|)
  \le k_q\!\left(\frac{3L+1}{2}|q|^2+\frac{L+1}{2}|p|^2+b_0^2\right),
\]
and moreover
\[
  |\Delta_q\mathcal M(q,p)|
  =\mathrm{tr}(\mathsf K_{qq})
  \le d\,\|\mathsf K_{qq}\|_{\mathrm{op}}.
\]
Let $c_1>0$ be a quadratic lower bound constant such that, up to an additive constant shift of $U$,
\begin{equation}\label{eq:BC-V0-lower}
  \mathcal V_0(q,p)\ge c_1\left(|q|^2+|p|^2\right),
  \qquad (q,p)\in\mathbb R^{2d}.
\end{equation}
Then the ``error term''
\[
  \mathrm{Err}^{(d)}(q,p)
  :=|\mathcal A'\mathcal M(q,p)|+|\Delta_q\mathcal M(q,p)|
\]
satisfies, for all $(q,p)\in\mathbb R^{2d}$,
\begin{equation}\label{eq:BC-Err-bound}
  \mathrm{Err}^{(d)}(q,p)
  \le C_2^{\mathrm{BC}}\ \mathcal V_0(q,p) + C_2^{(d),\mathrm{BC}},
\end{equation}
with the explicit choices
\[
  C_2^{\mathrm{BC}}
  :=\frac{k_q}{c_1}\max\left\{\frac{3L+1}{2},\frac{L+1}{2}\right\},
  \qquad
  C_2^{(d),\mathrm{BC}}
  :=k_q b_0^2 + d\,\|\mathsf K_{qq}\|_{\mathrm{op}}
  \ \le\ k_q C_0^2 + d\,\|\mathsf K_{qq}\|_{\mathrm{op}}.
\]

Finally define the (explicit) growth constant
\[
  \widetilde C_{\mathcal M}^{\mathrm{BC}}:=\frac{\|\mathsf K\|_{\mathrm{op}}}{2c_1}
\]
(cf.\ \eqref{eq:widetildeCM-def}), and set
\begin{equation}\label{eq:BC-delta-and-Clambda}
  \delta_{\mathrm{BC}}
  :=\underline c_{\mathrm{imp}}-\lambda\,\widetilde C_{\mathcal M}^{\mathrm{BC}},
  \qquad
  C_{\lambda,\mathrm{BC}}
  :=C_2^{\mathrm{BC}}+\widetilde C_{\mathcal M}^{\mathrm{BC}}\,\underline c_{\mathrm{imp}}.
\end{equation}
Then Proposition~\ref{prop:Valpha-drift} applies (for $\alpha$ sufficiently small) and yields the drift-rate expansion
\[
  \lambda_\alpha \;\ge\; \lambda+\delta_{\mathrm{BC}}\cdot\alpha-C_{\lambda,\mathrm{BC}}\cdot\alpha^2.
\]
In particular, $\lambda_\alpha>\lambda$ for all sufficiently small $\alpha>0$ whenever $\delta_{\mathrm{BC}}>0$.
\end{enumerate}
\end{proposition}

\begin{proof}
    We provide the proof in Appendix~\ref{app:classification-first-order}.
\end{proof}

Finally, we confirm that the quantitative condition $\delta_{\mathrm{BC}}>\gamma\lambda$ required for acceleration can be satisfied by choosing the dissipativity parameter $\lambda$ small enough relative to the ridge coefficient $\iota$.

\begin{lemma}[Feasibility of $\delta_{\mathrm{BC}}>\gamma\lambda$ with an explicit $\lambda_\star(\gamma)$]
\label{lem:BC-feasibility-explicit-full}
Assume the setting of Proposition~\ref{prop:classification-assumptions} and fix $\gamma>0$.
Set
\[
  \bar\lambda:=\min\left\{\frac14,\ \frac{\iota}{2}\right\},
  \qquad
  R:=\max\{1,C_{\mathrm{linear}}\}.
\]
Define
\begin{equation}\label{eq:BC-deltaU-plus}
  \delta_U^{+}
  :=
  \frac{C_0}{R}+\frac{A_\varphi}{R^2},
  \qquad
  C_0:=\Phi_1H_1B_x,
  \qquad
  A_\varphi:=|\varphi(0)|+\Phi_1(1+|h(0)|).
\end{equation}
Let
\[
  a^{-}:=\iota+\frac{\gamma^2}{2}(1-\bar\lambda).
\]
Let $c_1=c_1(\gamma,\bar\lambda)>0$ be the baseline quadratic lower bound constant for $\mathcal V_0$ (up to an additive
constant), i.e.
\begin{equation}\label{eq:BC-c1-def}
  c_1
  :=
  \frac18\left(\gamma^2(1-\bar\lambda)+2-\sqrt{(\gamma^2(1-\bar\lambda)-2)^2+4\gamma^2}\right).
\end{equation}

Let $B=\left(\begin{smallmatrix}0&I_d\\-\iota I_d&-\gamma I_d\end{smallmatrix}\right)$ and define the explicit decay proxies
\[
  \eta:=\frac{\gamma-\sqrt{(\gamma^2-4\iota)_+}}{2}>0,
  \qquad
  C_B:=1+\frac{\gamma}{2\sqrt{\iota}}+\sqrt{\iota}+\frac{1}{\sqrt{\iota}}.
\]
Define also
\[
  C_{B_1}^{+}:=2\left(1+\gamma+\iota+\frac{\gamma^2}{2}\right),
  \qquad
  \widetilde C_{\mathcal M}^{+}
  :=
  \frac{1}{2c_1}\cdot\frac{C_B^2}{2\eta}\,C_{B_1}^{+}.
\]
Finally define the explicit lower bound
\begin{equation}\label{eq:BC-cimp-minus}
  \underline c_{\mathrm{imp}}^{-}
  :=
  \frac{3}{8}\cdot
  \frac{
    a^{-}+1-\sqrt{(a^{-}-1)^2+\gamma^2}
  }{
    a^{-}+1+\sqrt{(a^{-}-1)^2+\gamma^2}
    \;+\;
    8\,\delta_U^{+}
  }.
\end{equation}
Then the explicit choice
\begin{equation}\label{eq:BC-lambda-star}
  \lambda_\star(\gamma)
  :=
  \min\left\{
    \bar\lambda,\ 
    \frac{\underline c_{\mathrm{imp}}^{-}}{\gamma+\widetilde C_{\mathcal M}^{+}}
  \right\}
\end{equation}
suffices to guarantee that for every $\lambda\in(0,\lambda_\star(\gamma)]$ we have
\[
  \delta_{\mathrm{BC}}>\gamma\lambda,
  \qquad
  \delta_{\mathrm{BC}}
  :=
  \underline c_{\mathrm{imp}}(\lambda)-\lambda\,\widetilde C_{\mathcal M}^{\mathrm{BC}}(\lambda),
\]
where $\underline c_{\mathrm{imp}}(\lambda)$ and $\widetilde C_{\mathcal M}^{\mathrm{BC}}(\lambda)$ are the constants from
Lemma~\ref{lem:first-order-improvement} applied to \eqref{eq:classification}.
\end{lemma}

\begin{proof}
We provide the proof in Appendix~\ref{app:BC-feasibility-explicit-full}.
\end{proof}

Combining the explicit first-order improvement established in Proposition~\ref{prop:classification-first-order} with the parameter selection strategy from Lemma~\ref{lem:BC-feasibility-explicit-full}, we obtain the following quantitative acceleration result for the HFHR dynamics in the context of Bayesian binary classification.

\begin{theorem}[HFHR acceleration for Bayesian binary classification]
\label{thm:HFHR-BC}
Consider the Bayesian binary classification problem defined by the potential \eqref{eq:classification} under Assumption~\ref{assump:classification}.
Let $P_t^{\alpha}$ be the semigroup of the corresponding HFHR dynamics.
Let $\rho_{\mathcal V_\alpha}$ be the Lyapunov-weighted semimetric used in Corollary~\ref{cor:global-acceleration} (constructed using the global Lipschitz constant $L$ and the Lyapunov function $\mathcal V_\alpha$), and let $\cW_{\rho_{\mathcal V_\alpha}}$ be the associated Wasserstein distance.

Fix $\gamma>0$ and choose the dissipativity parameter $\lambda\in(0,\lambda_\star(\gamma)]$ as in Lemma~\ref{lem:BC-feasibility-explicit-full}, so that the quantitative condition $\delta_{\mathrm{BC}}>\gamma\lambda$ holds.
Assume in addition that the quantitative conditions of Corollary~\ref{cor:global-acceleration} hold (in particular, $\Lambda_0>1/2$).
Then there exist explicit constants $\alpha_{\mathrm{BC}}>0$ and $\kappa_{\mathrm{BC}}>0$, depending on the model parameters ($B_x, H_1, H_2, \Phi_1, \Phi_2, \iota$) and $\gamma$, such that for every $\alpha\in(0,\alpha_{\mathrm{BC}}]$,
\[
  \cW_{\rho_{\mathcal V_\alpha}}\!\left(\mu P_t^{\alpha}, \nu P_t^{\alpha}\right)
  \;\le\; e^{-(c_0^{\mathrm{BC}} + \kappa_{\mathrm{BC}}\alpha)t}\,\cW_{\rho_{\mathcal V_\alpha}}(\mu, \nu),
  \qquad t \ge 0,
\]
for all probability measures $\mu, \nu$ with finite Lyapunov moments, where $c_0^{\mathrm{BC}}$ denotes the contraction rate of the kinetic Langevin dynamics at $\alpha=0$.
Moreover, one may choose explicitly
\[
  \alpha_{\mathrm{BC}}
  := \min\left\{ \alpha_{\mathrm{branch,acc}}^\mathrm{BC},\ \alpha_{\mathrm{metric,acc}}^\mathrm{BC} \right\},
\]
where $\alpha_{\mathrm{branch,acc}}^\mathrm{BC}$ and $\alpha_{\mathrm{metric,acc}}^\mathrm{BC}$ are the explicit thresholds from Theorem~\ref{thm:lyapunov-acceleration} and Theorem~\ref{thm:metric-acceleration} respectively, evaluated using the global constants from Proposition~\ref{prop:classification-assumptions} and Proposition~\ref{prop:classification-first-order}.
Similarly, the explicit gain is given by
\[
  \kappa_{\mathrm{global}}^{\mathrm{BC}}
  := \min\left\{ \kappa_{\mathrm{BC}},\ c_0^\mathrm{BC}\,c_2^\mathrm{BC},\ c_0^\mathrm{BC}\,c_3^\mathrm{BC} \right\},
\]
where $\kappa_{\mathrm{BC}}$ is the Lyapunov-branch gain from Theorem~\ref{thm:lyapunov-acceleration}, and $c_2^\mathrm{BC}, c_3^\mathrm{BC}$ are the metric-branch improvement constants from Theorem~\ref{thm:metric-acceleration}.
In particular, utilizing the explicit constants from Proposition~\ref{prop:classification-first-order}, \
one can take
\[
  \kappa_{\mathrm{BC}} = \frac{L(\delta_{\mathrm{BC}} + \gamma\lambda)}{768\,\gamma},
\]
where $L = \iota + (\Phi_2 H_1^2+\Phi_1 H_2) B_x^2$ is the global Lipschitz constant from Proposition~\ref{prop:classification-assumptions}
and
\[
  \delta_{\mathrm{BC}} := \underline c_{\mathrm{imp}} - \lambda\,\widetilde C_{\mathcal M}^{\mathrm{BC}},
  \qquad
  C_{\lambda,\mathrm{BC}} := C_2^{\mathrm{BC}} + \widetilde C_{\mathcal M}^{\mathrm{BC}}\,\underline c_{\mathrm{imp}}.
\]
\end{theorem}

\begin{proof}
We provide the proof in Appendix~\ref{proof:thm:HFHR-BC}.
\end{proof}

\section{Numerical Experiments}

In this section, we conduct numerical experiments of \textit{Hessian-free high-resolution Monte Carlo} (HFHRMC), which is based on the Euler-Maruyama discretization of HFHR dynamics in~\eqref{eq:HFHR-SDE}.
We introduce the iterates of HFHRMC as follows:
\begin{align}  
\label{eq:hfhrmc}
q_{k+1} & =q_k +\left(p_k - \alpha \nabla U\left(q_k\right) \right) \eta+\sqrt{2 \alpha \eta} \xi^{q}_{k+1},
\nonumber
\\ 
p_{k+1} & =p_k + \left(-\gamma p_k - \nabla U\left(q_k\right) \right) \eta + \sqrt{2 \gamma \eta} \xi^{p}_{k+1},
\end{align}
where $\xi_{k}^{q},\xi_{k}^{p}$
are i.i.d. Gaussian random vectors
$\mathcal{N}(0,I_{d})$, and $\xi_{k}^{q},\xi_{k}^{p}$ are independent of each other.
We also perform our experiments using \textit{kinetic Langevin Monte Carlo} (KLMC), which is based on the Euler-Maruyama discretization of kinetic Langevin dynamics \eqref{eqn:underdamped}
whose iterates are given by:
\begin{align} 
\label{eq:klmc}
x_{k+1} & = x_k + v_{k} \eta,
\\
v_{k+1} & = v_k + \left(-\gamma v_k - \nabla U\left(x_{k}\right)\right) \eta+\sqrt{2 \gamma \eta} \xi_{k+1}, 
\end{align}
where $\xi_{k}$
are i.i.d. Gaussian random vectors
$\mathcal{N}(0,I_{d})$.

In the following sections, we will conduct numerical experiments using HFHRMC and KLMC. First, we will conduct numerical experiments for a toy example, the multi-well potential case in \eqref{eq:MW-high-dim} (Section~\ref{sec:num:multi}). 
Next, we will conduct Bayesian linear regression with $L^p$ regularizer 
(Section~\ref{sec:num:linear}). 
We will also apply the algorithms to Bayesian binary classification 
(Section~\ref{sec:num:classification}). In all these examples, the potential function $U$ is non-convex and satisfies both Assumption~\ref{assump:potential}
and Assumption~\ref{assump:asymptotic-linear-drift}. Finally, we will study another numerical example, Bayesian logistic regression with ridge regularizer, where the potential function $U$ is non-convex that may not satisfy Assumptions~\ref{assump:potential}
and \ref{assump:asymptotic-linear-drift} (Section~\ref{sec:num:neural}). 

\subsection{Multi-well potential}\label{sec:num:multi}

In this section, we conduct numerical experiments based on a toy example, the multi-well potential that is considered
in Section~\ref{sec:case:multi}, which satisfies both Assumption~\ref{assump:potential} and Assumption~\ref{assump:asymptotic-linear-drift}. We consider the multi-well potential in dimension $d = 8$, and choose different values of $\alpha$: $0.01$, $0.05$, $0.1$, $0.2$, $0.5$, $0.8$, $1.0$ for HFHRMC~\eqref{eq:hfhrmc}, and choose $\gamma = 2.0$ for both HFHRMC in~\eqref{eq:hfhrmc} and KLMC in~\eqref{eq:klmc}. We iterate both algorithms $10000$ steps with step size $\eta = 10^{-3}$ and compute over $M = 2000$ chains. We obtain the plot in Figure~\ref{fig:multi-well} where the $x$-axis represents the iteration $k$ and the $y$-axis represents the logarithm of the Wasserstein distance between the empirical distribution driven by the algorithm and the Gibbs distribution.

\begin{figure}[htpb]
    \centering
    \includegraphics[width=0.8\linewidth]{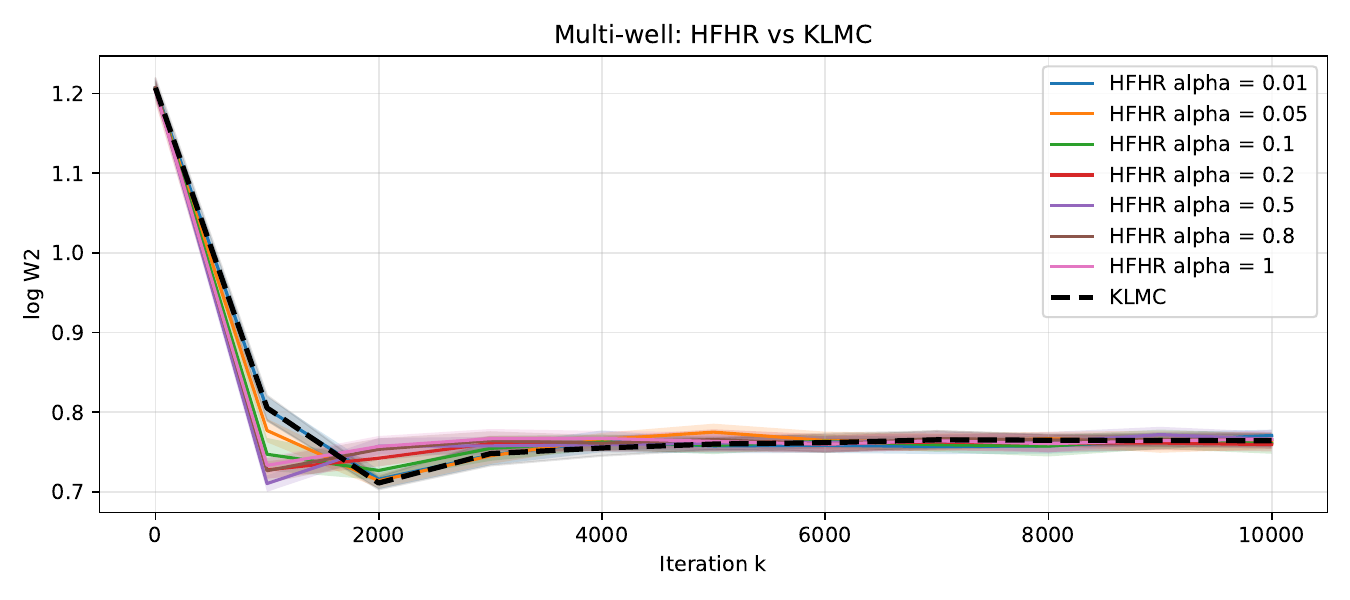}
    \caption{Multi-well potential in dimension $d = 8$.}
    \label{fig:multi-well}
\end{figure}

We can observe from Figure~\ref{fig:multi-well} that HFHRMC achieves better performance compared to KLMC in this multi-well potential example. We find that for $\alpha = 0.01$, HFHRMC and KLMC achieve comparable convergence performance. However, HFHRMC performs better for larger values of $\alpha$. In particular, we observe that increasing $\alpha$ accelerates the convergence of HFHRMC, which is consistent with our theory in Section~\ref{sec:acceleration}. As $\alpha$ approaches $1$, the convergence of HFHRMC slows, which corresponds to a smaller contraction rate as shown in Corollary~\ref{cor:W2-convergence-baseline}.

\subsection{Bayesian linear regression with synthetic data}\label{sec:num:linear}

In this section, we consider the Bayesian linear regression model as follows:
\begin{equation}
y_j=x_j^{\top}\beta_{*}+\delta_j,\quad \delta_j \sim \mathcal{N}(0,\sigma^2), \quad x_j \sim \mathcal{N}\left(0, 0.5^2 I_d\right), \quad j = 1,\ldots,n,
\end{equation}
where 
$\beta_*=[1.0, -0.5, 0.7, 1.2, -3.0,   5.4]^{\top}$
is a fixed ground-truth coefficient vector. Our goal is to sample the posterior distribution given by
$\pi(q) \propto \exp \left\{-U(q)\right\}$,
where $U(q)$ is the negative log-posterior i.e. the squared loss with a regularizer that we will choose. In order to present the performance of convergence of the algorithms, we compute the MSE at the $k$-th iterate defined by the following formula:
$\mathrm{MSE}_k:=\frac{1}{n} \sum_{j=1}^n\left(y_j-\left(x_j\right)^{\top} q_k\right)^2$, and the mean of the paramater after $K$ iterates over $M$ chains is given as
$
\bar{q}_K = \frac{1}{M}\sum_{m=1}^Mq_k^{(m)}
$.

We follow the Bayesian linear regression with $L^p$ regularizer introduced in Section~\ref{sec:case:L:p} and consider the the objective function of as in~\eqref{eq:linear}. As discussed in Section~\ref{sec:case:L:p}, the corresponding objective function $U$ satisfies our Assumption~\ref{assump:potential} and Assumption~\ref{assump:asymptotic-linear-drift}. By choosing the parameters in HFHRMC~\eqref{eq:hfhrmc} and KLMC~\eqref{eq:klmc} such that $\alpha =0.1, \gamma_{\mathrm{HFHRMC}}=1.0, \gamma_{\mathrm{KLMC}}=10.0$, and the parameters in linear regression~\eqref{eq:linear} such that $\sigma = 0.4, \lambda = 0.1, \ep = 0.001, p = 1.2$, we take $n = 1000$ samples, $M = 10$ chains, choose $\eta = 10^{-4}$ with $10000$ steps and we obtain the following plot in Figure~\ref{fig:linear}. As shown in the figure, HFHRMC with $\gamma_{\mathrm{HFHRMC}} = 1.0$ converges significantly faster and achieves a lower MSE compared to KLMC with $\gamma_{\mathrm{KLMC}} = 10.0$. The latter exhibits oscillatory behavior and only begins to converge after approximately $6000$ steps. It is also worth noting that when $\gamma_{\mathrm{KLMC}}$ is set to $1.0$, KLMC fails to converge.

\begin{figure}[htpb]
    \centering
    \includegraphics[width=0.8\linewidth]{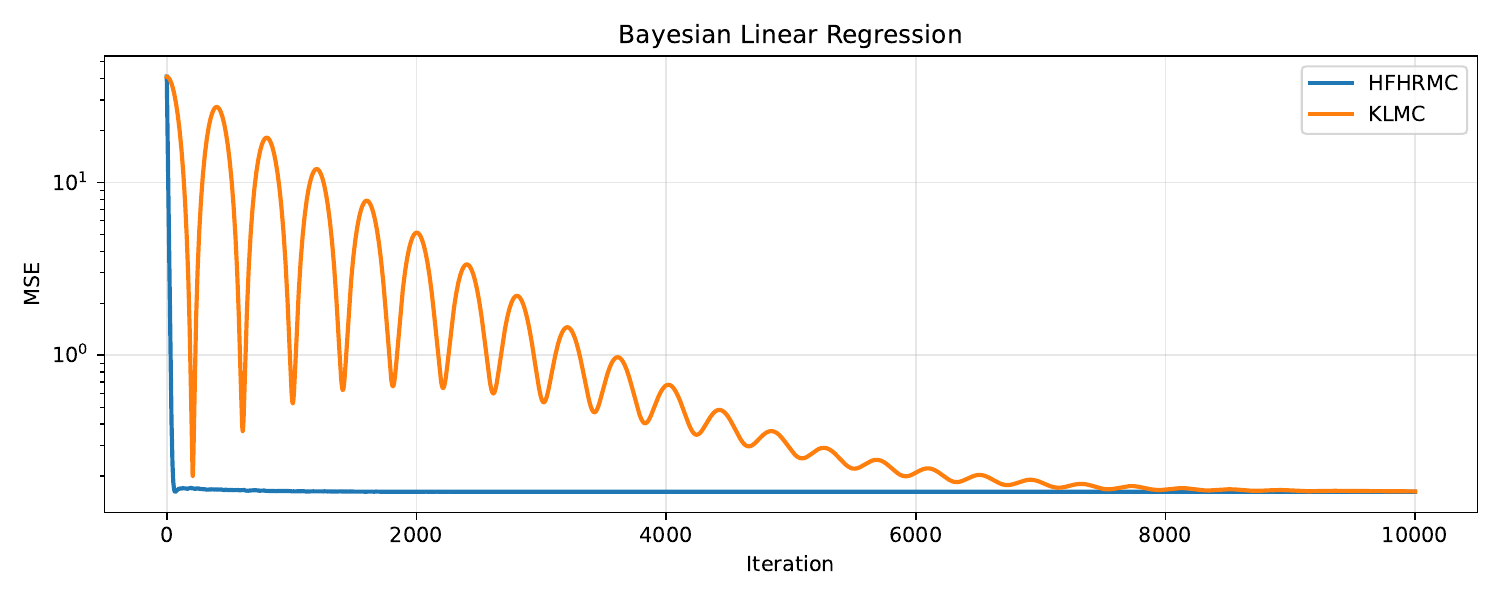}
    \caption{Bayesian linear regression.}
        \label{fig:linear}
\end{figure}

\subsection{Bayesian binary classification with real data}\label{sec:num:classification}

In this section, we test the performance of our algorithms in Bayesian binary classification problems in~\eqref{eq:classification} with Tukey bisquare loss in~\eqref{eq:tukey} that are introduced in Section~\ref{sec:case:classification}. As discussed in Section~\ref{sec:case:classification}, the corresponding objective function $U$ satisfies our Assumptions~\ref{assump:potential} and ~\ref{assump:asymptotic-linear-drift}. We apply HFHRMC and KLMC algorithms to this Bayesian binary classification task with real data (Breast Cancer\footnote{Breast Cancer - UCI Machine Learning Repository, \url{ https://archive.ics.uci.edu/dataset/14/breast+cancer}}). The Breast Cancer Wisconsin (Diagnostic) dataset, consisting of $n=569$ samples and $d=30$ real-valued features. The binary response indicates whether the tumor is malignant (labled as $1$) or benign (labled as $0$). We split the dataset into training and test subsets (70/30). The goal is achieve binary classification such that given $x\in\mathbb{R}^{30}$, we are able to predict $y\in\{0,1\}$.

In order to present the performance of convergence of the algorithms, we first compute the mean of the paramater after $K$ iterates over $M$ chains. Then, for a test feature $\hat{x}$, we compute the predicted label
$\hat{y} = \mathbf{1}_{\{h(\langle \bar{q}_K,\hat{x}\rangle) \geq 1/2\}} \in \{0,1\}$,
where $h$ is the predictive function defined in Section~\ref{sec:case:classification}.
The classification performance is evaluated using the test accuracy of the form: 
\begin{equation}\label{eq:acc}
\mathrm{Acc} := \frac{1}{n_{\mathrm{test}}}\sum_{i=1}^{n_{\mathrm{test}}}\mathbf{1}_{\hat{y}_i = y_i}.
\end{equation}
To process our experiment, we use the objective function $U$ of the Bayesian binary classification problem in~\eqref{eq:classification} with Tukey's bisquare loss in~\eqref{eq:tukey}, and choose the parameters in HFHRMC~\eqref{eq:hfhrmc} and KLMC~\eqref{eq:klmc} such that $\alpha =0.05, \gamma=1.0$, and the parameters in binary classification with Tukey's bisquare loss such that $\iota = 0.05$ and $t_0 = 2.0$. Moreover, we take $M = 50$ chains, and choose $\eta = 10^{-4}$ with $20000$ steps. As a result, we obtain Figure~\ref{fig:classification}.

\begin{figure}[htpb]
    \centering
    \includegraphics[width=0.8\linewidth]{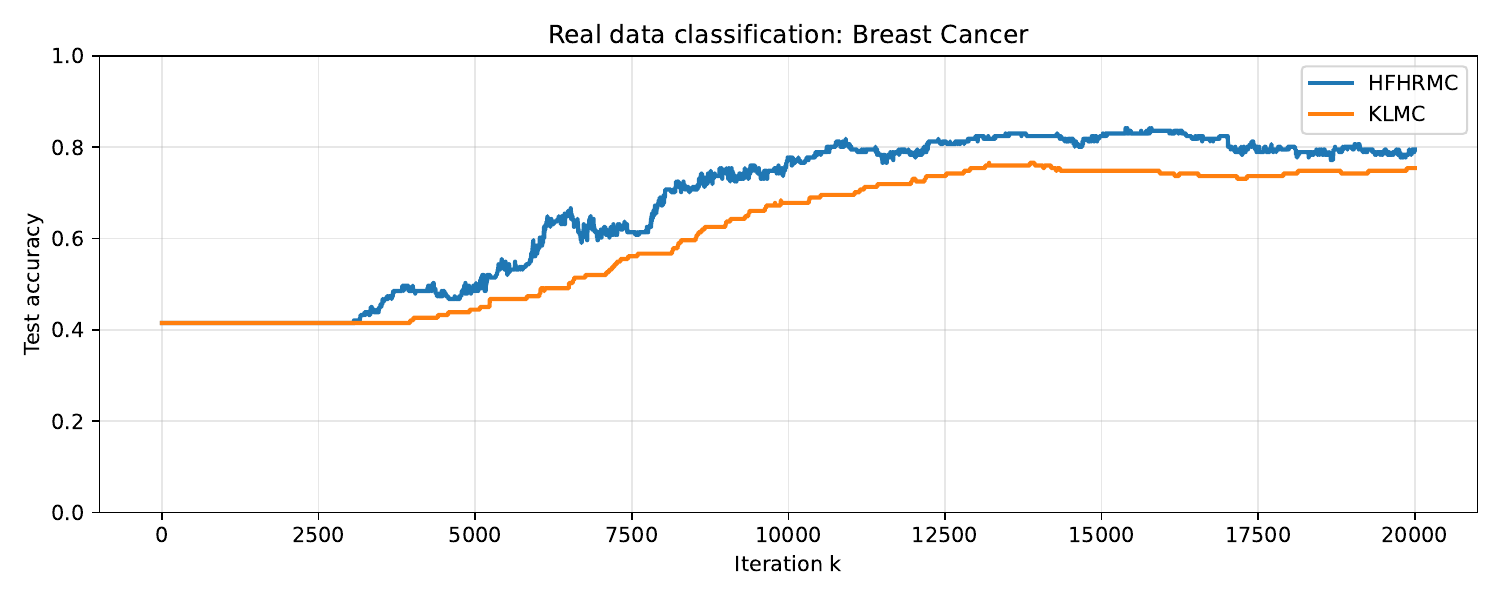}
    \caption{Bayesian binary classification.}
    \label{fig:classification}
\end{figure}

We can observe from Figure~\ref{fig:classification} that HFHRMC produces a higher test accuracy around $80\%$ than the one produced by KLMC; moreover, we observe that the convergence of HFHRMC is slightly faster than KLMC in the task of Bayesian binary classification.

\subsection{Bayesian logistic regression with real data}\label{sec:num:neural}

In this section, we consider Bayesian logistic regression with real data (Iris\footnote{Iris - UCI Machine Learning Repository, \url{ https://archive.ics.uci.edu/dataset/53/iris}}) processed by the neural networks. Iris dataset consists $n=150$ samples, each with $d=4$ real-valued features. To fit the Bayesian binary logistic regression framework, we select two classes, versicolor and virginica, from data set and relabel the observations as $y_i \in \{0,1\}$ and $x_i \in \mathbb{R}^4$. We split the dataset into training and test subsets (80/20) and our goal is to model the conditional distribution of the label given the features and parameter vector $q\in\mathbb{R}^d$ as $\mathbb{P}(y_i=1 \vert x_i,q)
= \sigma\left(q^\top x_i\right)$ with sigmoid function $\sigma(z)=\frac{1}{1+e^{-z}}$. We impose a Gaussian prior on the regression coefficients, $q \sim \mathcal{N}(0,\iota^{-1} I_d)$, such that it gives Gibbs potential $\pi(q) \propto e^{-U(q)}$ with
$
U(q)
=
\frac{1}{n}\sum_{i=1}^n
\left(
\log\left(1+e^{q^\top x_i}\right)
-
y_i\,q^\top x_i\right) +
\frac{\iota}{2}|q|^2,
$ where the first term is the negative log-entropy loss and the second term is the ridge regularizer.

For a new feature vector $\hat{x}$, the posterior predictive probability is $\bar{p}(\hat{x})=\mathbb{E}_{q \sim \pi}\left[\sigma\left(q^{\top} \hat{x}\right)\right]$ which can be approximated over $M$ chains in the form of $\frac{1}{M} \sum_{m=1}^M \sigma\left((q^{(m)})^{\top} \hat{x}\right)$ and the predicted label is $\hat{y} = \mathbf{1}_{\{\bar{p}(\hat{x}) \geq 1 / 2\}}$. We process a feedforward neural network and use HFHRMC and KLMC samples from the Gibbs posterior to compute the predictive quantities.

The study of Bayesian logistic regression with real data processed by neural networks has also appeared in \cite{blundell2015weight,osawa2019practical,gurbuzbalaban2025anchored}. Even though in the presence of neural networks, it does not seem easy to verify Assumptions~\ref{assump:potential} and \ref{assump:asymptotic-linear-drift}, we will nevertheless show the efficiency of our proposed algorithm. In particular, we consider a fully connected feedforward neural network with $L = 3$ hidden layers, and each hidden layers have same number of neurons in $N_{\mathrm{neurons}} = 32$, the neural network is parameterized by
$\theta = (W_1,\dots,W_L)\in\mathbb{R}^D$,
with $W_\ell\in\mathbb{R}^{m_{\ell-1}\times m_\ell}$,
where $m_0=d$ and $m_L=1$. To ensure smoothness of the potential, we employ a Gaussian-smoothed ReLU activation such that 
$\phi_\nu(z) = \mathbb{E}_{\xi\sim\mathcal{N}(0,1)}\left[(z+\nu\xi)^{+}\right] $. We can check that its derivative is bounded and Lipschitz continuous. The network forward map is defined recursively as
$$
h_0(x) = x, \quad h_\ell(x;\theta) = \phi_\nu\left(h_{\ell-1}(x;\theta) W_\ell\right), \quad \ell=1,\dots,L-1,
$$
The predicted probability is now given by the sigmoid function parameterized by $\theta$ such that 
$p_\theta(x)=\sigma(z_\theta(x))
=
\frac{1}{1+e^{-z_\theta(x)}}.
$
As a result, we define the potential function as
$$
U(\theta)
=
\frac{1}{n}\sum_{i=1}^n
\mathcal{L}\left(y_i,p_\theta(x_i)\right)
+
\frac{\iota}{2}|\theta|^2,
$$
where the loss function is and the loss function is 
$
\mathcal{L}(y, p_{\theta}(x)) = -y\log p_{\theta}(x)-(1-y)\log(1-p_{\theta}(x))
$,
and the quadratic term $\frac{\iota}{2}|\theta|^2$ corresponds to a Gaussian prior $\theta \sim \mathcal{N}(0,\iota^{-1}I_d)$.

To proceed with the binary logistic regression task, we choose Gaussian smoothing factor $\nu = 32$, ridge strength $\iota = 10^{-3}$, and parameters for HFHRMC and KLMC with $\alpha = 1.0$ and $\gamma = 2.0$, then we implement algorithms $M = 50$ chains, $5000$ iterates with the step size $\eta = 0.5 \times 10^{-4}$, we get Figure~\ref{fig:nn:logistic}.

\begin{figure}[htbp]
    \centering
    \begin{minipage}[b]{0.48\textwidth}
        \centering
        \includegraphics[width=\linewidth]{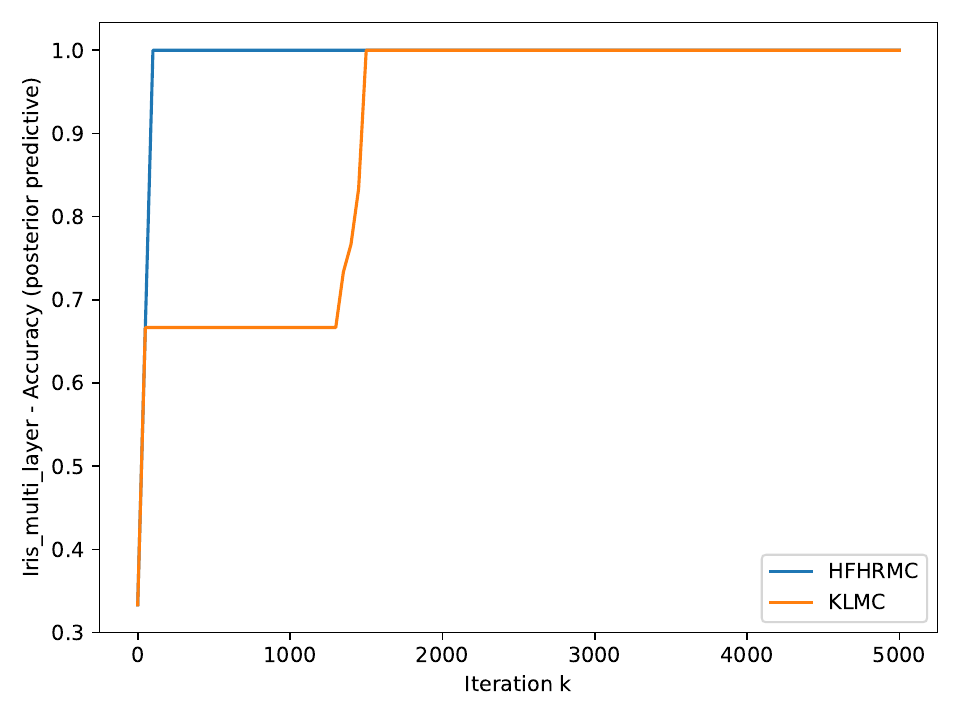}
    \end{minipage}
    \hfill
    \begin{minipage}[b]{0.48\textwidth}
        \centering
        \includegraphics[width=\linewidth]{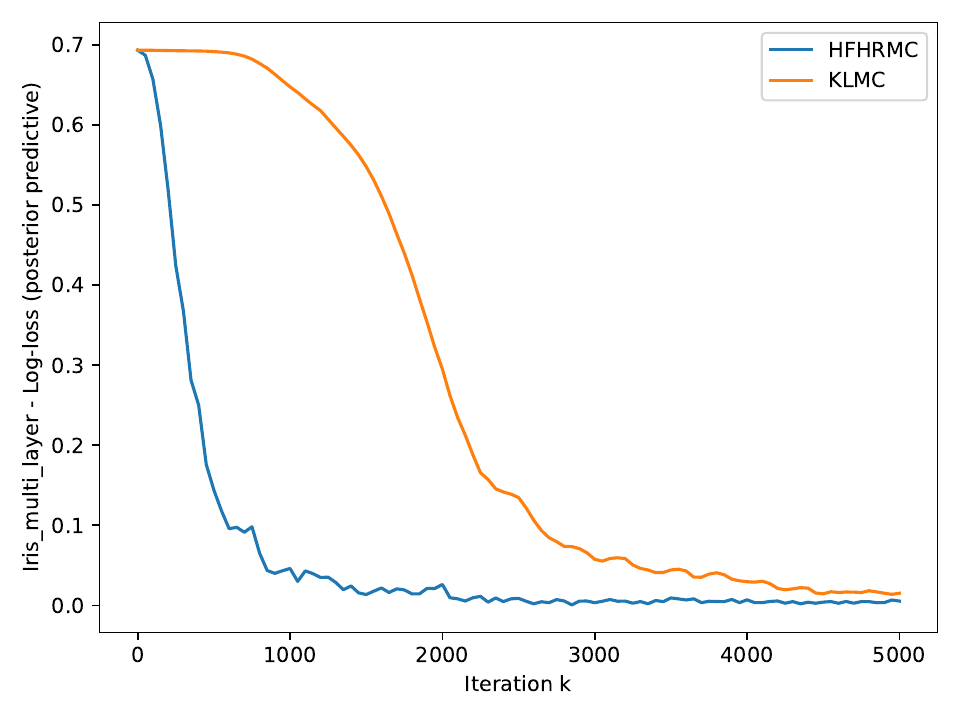}
    \end{minipage}
    \caption{Bayesian logistic regression processed by feedforward neural network with $L = 3$ layers.}
    \label{fig:nn:logistic}
\end{figure}

The left plot in Figure~\ref{fig:nn:logistic} is the test accuracy computed by~\eqref{eq:acc} and the right plot in Figure~\ref{fig:nn:logistic} is the log-loss of the predictive posterior. We can observe from the plots that both HFHRMC and KLMC achieve a high accuracy of prediction, and moreover HFHRMC achieves acceleration and has a superior performance where the log-loss decreases faster.

\section{Conclusion}

In this paper, we provided a theoretical analysis of the Hessian-free high-resolution (HFHR) dynamics for sampling from target distributions $\pi(q)\propto e^{-U(q)}$ with non-convex potentials. While HFHR dynamics has demonstrated empirical success in various settings, existing theory was largely restricted to strongly-convex cases. Our work bridges this gap between theory and practice by establishing convergence guarantees in the non-convex regime.
By adopting the reflection/synchronous coupling framework and constructing appropriate Lyapunov functions, under smoothness and dissipativity assumptions, we proved that the HFHR semigroup is exponentially contractive in a Lyapunov-weighted Wasserstein distance for all sufficiently small resolution parameters $\alpha>0$. 
Crucially, we went beyond basic convergence to demonstrate quantitative acceleration. Under an additional assumption that asymptotically $\nabla U$ has linear growth at infinity, we showed that HFHR dynamics achieves a strictly better contraction rate than kinetic Langevin dynamics. We established an explicit linear-in-$\alpha$ gain that applies not only when the convergence is limited by the Lyapunov drift (recurrence from infinity) but also when it is dominated by the metric coupling (barrier crossing). 
We illustrated these theoretical results 
through three concrete examples: a multi-well potential, Bayesian linear regression with $L^p$ regularizer and Bayesian binary classification.
We conducted numerical experiments based
on these examples, as well as an additional example of Bayesian logistic regression with real data processed by the neural networks. Our numerical experiments corroborated the theory
and illustrated the efficiency of the algorithms based on HFHR dynamics. 
Our numerical results showed the acceleration and superior performance compared to kinetic Langevin dynamics.

\section*{Acknowledgments}

Xiaoyu Wang is supported by the Guangzhou-HKUST(GZ) Joint Funding Program (No.2024A03J0630), Guangzhou Municipal Key Laboratory of Financial Technology Cutting-Edge Research.
Lingjiong Zhu is partially supported by the grants NSF DMS-2053454 and DMS-2208303. 


\bibliographystyle{alpha}
\bibliography{bibtex}

\appendix

\section{Proofs for the Results in Section~\ref{sec:setup}}

\subsection{Proof of Proposition~\ref{prop:V0-drift-HFHR}}\label{app:V0-drift-HFHR}
\begin{proof}
Using the decomposition \eqref{eq:Lalpha-decomp}, we have
\begin{align}\label{L:alpha:V:0}
  \cL_\alpha \mathcal V_0
  = \cL_0 \mathcal V_0
    + \alpha\,\mathcal A'\mathcal V_0
    + \alpha\,\Delta_q \mathcal V_0.
\end{align}
By \eqref{eq:L0-V0-Lyapunov}, we have
\begin{align}\label{L:0:V:0:ineq}
  \cL_0 \mathcal V_0(q,p)
  \;\le\; \gamma\left(d + A - \lambda\,\mathcal V_0(q,p)\right).
\end{align}
It remains to control the perturbation terms $\mathcal A'\mathcal V_0$ and $\Delta_q \mathcal V_0$ in \eqref{L:alpha:V:0}.

\medskip

First, we aim to obtain an explicit bound on $\mathcal A'\mathcal V_0$.
From the definition \eqref{eq:V0-general-quadratic} we obtain
\[
  \nabla_q\mathcal V_0(q,p)
  = \nabla U(q)
    + \frac{\gamma^2}{2}\left(q+\gamma^{-1}p\right)
    - \frac{\gamma^2\lambda}{2}\,q
  = \nabla U(q)
    + \frac{\gamma^2}{2}\left((1-\lambda)q+\gamma^{-1}p\right).
\]
Hence
\begin{align*}
  \mathcal A'\mathcal V_0(q,p)
  = -\nabla U(q)\cdot\nabla_q\mathcal V_0(q,p) = -|\nabla U(q)|^2
     - \frac{\gamma^2}{2}\,\nabla U(q)\cdot
      \left((1-\lambda)q+\gamma^{-1}p\right).
\end{align*}
Using Cauchy--Schwarz and Young's inequalities with
\[
  a := |\nabla U(q)|,
  \qquad
  b := \frac{\gamma^2}{2}\left|(1-\lambda)q+\gamma^{-1}p\right|,
\]
we get
\[
  \frac{\gamma^2}{2}\,|\nabla U(q)|\,
   \left|(1-\lambda)q+\gamma^{-1}p\right|
  \;\le\; \frac12 a^2 + \frac12 b^2
  = \frac12|\nabla U(q)|^2
      + \frac{\gamma^4}{8}\left|(1-\lambda)q+\gamma^{-1}p\right|^2.
\]
Therefore
\[
  \mathcal A'\mathcal V_0(q,p)
  \;\le\;
  -|\nabla U(q)|^2
  + \frac12|\nabla U(q)|^2
  + \frac{\gamma^4}{8}\left|(1-\lambda)q+\gamma^{-1}p\right|^2
  \;\le\;
  \frac{\gamma^4}{8}\left|(1-\lambda)q+\gamma^{-1}p\right|^2.
\]
Next, using
\[
  \left|(1-\lambda)q+\gamma^{-1}p\right|^2
  \le 2(1-\lambda)^2|q|^2 + 2\gamma^{-2}|p|^2,
\]
we obtain
\[
  \mathcal A'\mathcal V_0(q,p)
  \;\le\;
  \frac{\gamma^4}{4}(1-\lambda)^2|q|^2
  + \frac{\gamma^2}{4}|p|^2.
\]
From \eqref{eq:V0-equivalent} and $U\ge0$, we have, for some $c_1>0$,
\[
  c_1\left(1+|q|^2+|p|^2\right)
  \;\le\; 1+\mathcal V_0(q,p),
\]
which implies
\[
  |q|^2+|p|^2
  \;\le\; \frac{1}{c_1}\left(1+\mathcal V_0(q,p)\right).
\]
Therefore
\begin{align}
  \mathcal A'\mathcal V_0(q,p)
  &\le
  \left[\frac{\gamma^4}{4}(1-\lambda)^2+\frac{\gamma^2}{4}\right]
     \left(|q|^2+|p|^2\right)
  \le
  \frac{1}{c_1}\left[\frac{\gamma^4}{4}(1-\lambda)^2+\frac{\gamma^2}{4}\right]
  \left(1+\mathcal V_0(q,p)\right).\label{eq:AprimeV0-explicit}
\end{align}
By introducing
\begin{equation}\label{eq:KA-def}
  K_A
  := \frac{1}{c_1}\left[\frac{\gamma^4}{4}(1-\lambda)^2+\frac{\gamma^2}{4}\right],
\end{equation}
we conclude that
\begin{equation}\label{eq:AprimeV0-bound-explicit}
  \mathcal A'\mathcal V_0(q,p)
  \;\le\; K_A\left(1+\mathcal V_0(q,p)\right).
\end{equation}

\medskip

Next, we derive an explicit bound on $\Delta_q\mathcal V_0$.
Again from \eqref{eq:V0-general-quadratic},
\[
  \Delta_q\mathcal V_0(q,p)
  = \Delta U(q) + \frac{\gamma^2}{2}d(1-\lambda).
\]
Since $\nabla U$ is globally Lipschitz with constant $L$, the operator norm of the Hessian of $U$
is bounded by~$L$ almost everywhere, and hence
$|\Delta U(q)|\le Ld$. Therefore
\begin{equation}\label{eq:DqV0-explicit}
  \left|\Delta_q\mathcal V_0(q,p)\right|
  \;\le\; Ld + \frac{\gamma^2}{2}d(1-\lambda)
  =: K_\Delta.
\end{equation}
Using $1+\mathcal V_0\ge1$, it follows from \eqref{eq:DqV0-explicit} that
\begin{equation}\label{eq:DqV0-bound-explicit}
  \left|\Delta_q\mathcal V_0(q,p)\right|
  \;\le\; K_\Delta\left(1+\mathcal V_0(q,p)\right),
  \qquad
  K_\Delta := Ld + \frac{\gamma^2}{2}d(1-\lambda).
\end{equation}

\medskip

Combining \eqref{eq:AprimeV0-bound-explicit} and \eqref{eq:DqV0-bound-explicit}, we obtain
\begin{align}\label{pert:ineq}
  \mathcal A'\mathcal V_0(q,p) + \Delta_q\mathcal V_0(q,p)
  \;\le\; J_1\left(1+\mathcal V_0(q,p)\right),
\end{align}
where $J_1$ is defined in \eqref{eq:K-explicit}.

Therefore, combining the bounds for $\cL_0$ in \eqref{L:0:V:0:ineq} and the perturbation terms in \eqref{pert:ineq}:
\begin{align}
  \cL_\alpha\mathcal V_0(q,p)
  &\le
  \gamma\left(d+A-\lambda\mathcal V_0(q,p)\right)
  + \alpha J_1\left(1+\mathcal V_0(q,p)\right)\nonumber\\
  &= \gamma(d+A) - \gamma\lambda\mathcal V_0(q,p)
     + \alpha J_1 + \alpha J_1\mathcal V_0(q,p). \label{eq:drift-expansion-step}
\end{align}
To cast this into the standard drift form $\gamma(d+A_\alpha - \hat\lambda_\alpha \mathcal V_0)$, we group the constant terms and the $\mathcal V_0$ terms. We factor out $\gamma$ from the entire expression to obtain:
\begin{align*}
  \cL_\alpha\mathcal V_0(q,p)
  &\le \left[\gamma(d+A) + \alpha J_1\right] - \left[\gamma\lambda - \alpha J_1\right]\mathcal V_0(q,p) \\
  &= \gamma \left( d + A + \frac{J_1}{\gamma}\alpha \right) 
     - \gamma \left( \lambda - \frac{J_1}{\gamma}\alpha \right) \mathcal V_0(q,p).
\end{align*}
This matches the desired inequality \eqref{eq:Lalpha-V0-Lyapunov} with
$A_\alpha := A + \frac{J_1}{\gamma}\alpha$
and
$\hat\lambda_\alpha := \lambda - \frac{J_1}{\gamma}\alpha$.
The explicit expansion of $\hat\lambda_\alpha$ and the choice of $\alpha_0$ then follow directly from substituting the expression for $J_1$. Specifically, to ensure $\hat\lambda_\alpha \ge \lambda/2$, we require
$\frac{J_1}{\gamma}\alpha \le \frac{\lambda}{2}$ which is equivalent to $\alpha \le \frac{\gamma\lambda}{2J_1}$,
which corresponds to the definition of $\alpha_0$ in \eqref{eq:alphastar-explicit}. The proof is complete.
\end{proof}


\section{Proofs for the Results in Section~\ref{sec:global-contractivity}}
\label{app:global-contractivity}

\subsection{Proof of Lemma~\ref{lem:r-equivalent}}\label{app:r-equivalent}

\begin{proof}
Let $\Delta z := z-z' = (\Delta q, \Delta p)$. We denote the standard Euclidean norm on $\R^{2d}$ by $|\cdot|$. Note that $|\Delta z| \le |\Delta q| + |\Delta p|$ and $|\Delta z|^2 = |\Delta q|^2 + |\Delta p|^2$.

\medskip\noindent\emph{Upper bound ($k_2$):}
By the triangle inequality, and the definition of $r(z,z')$, we have
\begin{align*}
  r(z,z')
  \le \theta |\Delta q| + |\Delta q| + \gamma^{-1} |\Delta p|
  = (\theta+1) |\Delta q| + \gamma^{-1} |\Delta p|.
\end{align*}
Applying the Cauchy--Schwarz inequality to the vectors $((\theta+1), \gamma^{-1})$ and $(|\Delta q|, |\Delta p|)$, we obtain
\[
  r(z,z')
  \le \sqrt{(\theta+1)^2 + \gamma^{-2}} \sqrt{|\Delta q|^2 + |\Delta p|^2}
  = k_2\,|\Delta z|.
\]

\medskip\noindent\emph{Lower bound ($k_1$):}
From the definition of $r(z,z')$, we immediately have explicit control on $\Delta q$:
\begin{equation}\label{eq:bound-dq}
  |\Delta q| \le \frac{1}{\theta} r(z,z').
\end{equation}
To control $\Delta p$, we rewrite it as $\Delta p = \gamma \left( (\Delta q + \gamma^{-1}\Delta p) - \Delta q \right)$. Using the triangle inequality:
\[
  |\Delta p|
  \le \gamma \left( \left|\Delta q + \gamma^{-1}\Delta p\right| + |\Delta q| \right).
\]
Since $\left|\Delta q + \gamma^{-1}\Delta p\right| \le r(z,z')$ (by dropping the first nonnegative term in the definition of $r$) and using \eqref{eq:bound-dq}, we get
\begin{equation}\label{eq:bound-dp}
  |\Delta p|
  \le \gamma \left( r(z,z') + \frac{1}{\theta}r(z,z') \right)
  = \frac{\gamma(1+\theta)}{\theta} r(z,z').
\end{equation}
Finally, using the basic inequality $|\Delta z| \le |\Delta q| + |\Delta p|$, we sum \eqref{eq:bound-dq} and \eqref{eq:bound-dp}:
\[
  |\Delta z|
  \le \left( \frac{1}{\theta} + \frac{\gamma(1+\theta)}{\theta} \right) r(z,z')
  = \frac{1 + \gamma(1+\theta)}{\theta} r(z,z').
\]
Rearranging this yields
\[
  r(z,z') \ge \frac{\theta}{1 + \gamma(1+\theta)} |\Delta z| = k_1\,|\Delta z|.
\]
This completes the proof.
\end{proof}


\subsection{Proof of Lemma~\ref{lem:drift-decomp}}\label{app:drift-decomp}

\begin{proof}
Fix $\varepsilon>0$ and $c\in\R$. Recall the coupled HFHR dynamics
\eqref{eq:HFHR-SDE} for $z_t=(q_t,p_t)$ and $z'_t=(q'_t,p'_t)$ and recall from \eqref{defn:r:t} that
$Z_t:=q_t-q'_t$, $W_t:=p_t-p'_t$ and $\mathbf R_t:=Z_t+\gamma^{-1}W_t$.
Let $e_t:=\mathbf R_t/|\mathbf R_t|$ if $\mathbf R_t\neq0$ and an arbitrary
unit vector otherwise, and let $\mathcal P_t:=e_t e_t^\top$.
We recall from \eqref{eq:coupling-noise} that the coupling is defined by
\[
  dB_t^{q'}=dB_t^q,\qquad dB_t^{p'}=(I_d-2\chi(t)\mathcal P_t)\,dB_t^p,
\]
with a control process $\chi(t)\in\{0,1\}$.

\medskip\noindent
\emph{Step 1: Difference dynamics and the noise of $\mathbf R_t$.}
From \eqref{eq:HFHR-SDE},
\[
  dZ_t = \left(W_t-\alpha(\nabla U(q_t)-\nabla U(q'_t))\right)\,dt,
\]
and
\[
  dW_t = \left(-\gamma W_t-(\nabla U(q_t)-\nabla U(q'_t))\right)\,dt
         +\sqrt{2\gamma}\,d\widetilde B_t,
\qquad
  d\widetilde B_t:=dB_t^p-dB_t^{p'}=2\chi(t)\mathcal P_t\,dB_t^p.
\]
Hence
\[
  d\mathbf R_t
  = -\frac{1+\alpha\gamma}{\gamma}\,(\nabla U(q_t)-\nabla U(q'_t))\,dt
    + 2\sqrt{2}\,\gamma^{-1/2}\chi(t)\mathcal P_t\,dB_t^p .
\]
Since $\mathcal P_t$ projects onto $\mathrm{span}\{e_t\}$, the noise acts in
direction $e_t$, and therefore the It\^o correction term in $d|\mathbf R_t|$
vanishes. In particular,
\[
  \left.d|\mathbf R_t|\right|_{\rm noise}
  = \langle e_t,\,d\mathbf R_t\rangle_{\rm noise}
  = 2\sqrt{2}\,\gamma^{-1/2}\chi(t)\,\langle e_t,dB_t^p\rangle,
\qquad
  d\langle |\mathbf R|\rangle_t = 8\,\gamma^{-1}\left(\chi(t)\right)^2\,dt.
\]

\medskip\noindent
\emph{Step 2: Drift bound for $r_t$.}
Recall from \eqref{defn:r:t} that
\[
  r_t:=\theta|Z_t|+|\mathbf R_t|,\qquad
  \theta:=(1+\eta_0)L_{\rm eff}(\alpha)\gamma^{-2}.
\]
Recall the definition of $\delta_\alpha$ from \eqref{eq:delta-alpha-def} such that
$\delta_\alpha:=\frac{\eta_0}{1+\eta_0}-\frac{\alpha L}{\gamma}$.
Throughout this proof, we assume $\delta_\alpha>0$ (equivalently,
$\alpha<\frac{\eta_0}{1+\eta_0}\frac{\gamma}{L}$).

Using $dZ_t=(W_t-\alpha(\nabla U(q_t)-\nabla U(q_t')))\,dt$ and the one-sided Lipschitz bound
\[
  \left\langle \frac{Z_t}{|Z_t|},\,\nabla U(q_t)-\nabla U(q_t')\right\rangle \le L|Z_t|,
\]
together with the standard kinetic estimate for the $W_t$--contribution (with the choice
$\theta=(1+\eta_0)L_{\rm eff}(\alpha)\gamma^{-2}$), we obtain the finite-variation inequality
\begin{equation}\label{eq:drift-r-form}
  dr_t
  \;\le\;
  \gamma\left(\theta|\mathbf R_t|-\delta_\alpha\,\theta|Z_t|\right)\,dt
  + dM_t^{(r)},
\end{equation}
where the continuous local martingale $M^{(r)}$ is given by
\[
  M_t^{(r)}:= 2\sqrt{2}\,\gamma^{-1/2}\int_0^t \chi(s)\,\langle e_s,dB_s^p\rangle.
\]
Moreover,
\begin{equation}\label{eq:qr-r}
  d\langle r\rangle_t = d\langle |\mathbf R|\rangle_t
  = 8\,\gamma^{-1}\left(\chi(t)\right)^2\,dt.
\end{equation}
(Identity \eqref{eq:qr-r} follows from the fact that $Z_t$ has no noise
under our coupling.)

\medskip\noindent
\emph{Step 3: Meyer--It\^o for $f_\lambda(r_t)$.}
Let $f_\lambda$ be the concave profile from
\eqref{eq:phi-def}--\eqref{eq:f-def}.
Using the Meyer--It\^o formula, we obtain
\begin{align}\label{f:lambda:r:t}
  df_\lambda(r_t)
  = f'_{\lambda,-}(r_t)\,dr_t + \frac12 f_\lambda''(r_t)\,d\langle r\rangle_t
    + dM_t^{(f)},
\end{align}
with a continuous local martingale
\[
  M_t^{(f)}:= 2\sqrt{2}\,\gamma^{-1/2}\int_0^t f'_{\lambda,-}(r_s)\chi(s)\,
  \langle e_s,dB_s^p\rangle.
\]
Inserting \eqref{eq:drift-r-form}--\eqref{eq:qr-r} into \eqref{f:lambda:r:t} gives
\begin{equation}\label{eq:df-form}
  df_\lambda(r_t)
  \;\le\;
  \gamma\left[
    4\gamma^{-2}\left(\chi(t)\right)^2 f_\lambda''(r_t)
    +\left(\theta|\mathbf R_t|-\delta_\alpha\,\theta|Z_t|\right)
      f'_{\lambda,-}(r_t)
  \right]\,dt
  + dM_t^{(f)}.
\end{equation}

\medskip\noindent
\emph{Step 4: Dynamics of $G_t$ and the product rule.}
Let $\mathcal V$ be $(\lambda,D)$-admissible and recall from \eqref{defn:G:t} that
$G_t:=1+\varepsilon\mathcal V(z_t)+\varepsilon\mathcal V(z_t')$
and
$\rho_t:=f_\lambda(r_t)G_t$.
By It\^o's formula,
\begin{align}\label{d:G:t}
  dG_t
  = \varepsilon\left(\cL_\alpha\mathcal V(z_t)+\cL_\alpha\mathcal V(z_t')\right)\,dt
    + dM_t^{(G)},
\end{align}
where $M^{(G)}$ is the following continuous local martingale:
\begin{align}\label{eq:MG-explicit}
  M_t^{(G)}
  :=\;& \varepsilon\sqrt{2\alpha}\int_0^t
        \left\langle \nabla_q\mathcal V(z_s)+\nabla_q\mathcal V(z_s'),\,dB_s^q\right\rangle \nonumber\\
      &\;+\varepsilon\sqrt{2\gamma}\int_0^t
        \left\langle \nabla_p\mathcal V(z_s),\,dB_s^p\right\rangle
      +\varepsilon\sqrt{2\gamma}\int_0^t
        \left\langle \nabla_p\mathcal V(z_s'),\,dB_s^{p'}\right\rangle .
\end{align}
Using the coupling relation $dB_t^{p'}=(I_d-2\chi(t)\mathcal P_t)\,dB_t^p$,
the $p$-noise part can equivalently be written as
\begin{equation}\label{eq:MGp-explicit}
  \varepsilon\sqrt{2\gamma}\int_0^t
  \left\langle \nabla_p\mathcal V(z_s) + (I_d-2\chi(s)\mathcal P_s)^\top \nabla_p\mathcal V(z_s'),\,
  dB_s^p \right\rangle .
\end{equation}

Applying It\^o's product rule to $e^{ct}\rho_t=e^{ct}f_\lambda(r_t)G_t$ yields
\begin{align}\label{eq:ito-product}
  d(e^{ct}\rho_t)
  &= c e^{ct}\rho_t\,dt
     + e^{ct}G_t\,df_\lambda(r_t)
     + e^{ct}f_\lambda(r_t)\,dG_t
     + e^{ct}\,d\langle f_\lambda(r),G\rangle_t .
\end{align}
Substituting \eqref{eq:df-form} and the expression for $dG_t$ \eqref{d:G:t} into \eqref{eq:ito-product}, we obtain
\begin{align}
  d(e^{ct}\rho_t)
  \le\;& e^{ct}\gamma\Bigg[
      4\gamma^{-2}\left(\chi(t)\right)^2 f_\lambda''(r_t)G_t
      +\left(\theta|\mathbf R_t|-\delta_\alpha\,\theta|Z_t|\right)
        f'_{\lambda,-}(r_t)G_t\nonumber\\
      &\qquad\qquad
      +\gamma^{-1}\varepsilon f_\lambda(r_t)\left(\cL_\alpha\mathcal V(z_t)+\cL_\alpha\mathcal V(z_t')\right)
      +\gamma^{-1}c f_\lambda(r_t)G_t
    \Bigg]dt\nonumber\\
  &\quad + e^{ct}\,d\langle f_\lambda(r),G\rangle_t
    + dM_t,\label{eq:rho-ito-pre}
\end{align}
where $M_t$ is the continuous local martingale
\begin{equation}\label{eq:Mt-explicit}
  M_t
  := \int_0^t e^{cs} G_s\, dM_s^{(f)}
     + \int_0^t e^{cs} f_\lambda(r_s)\, dM_s^{(G)}.
\end{equation}

\medskip\noindent
\emph{Step 5: Bounding the cross-variation term.}
Only the noise in the $p$-component contributes to the cross-variation $d\langle f_\lambda(r),G\rangle_t$.
Using the coupling relation $dB_t^{p'}=(I_d-2\chi(t)\mathcal P_t)\,dB_t^p$ and the explicit expression for the martingale parts, a direct computation gives:
\begin{equation}\label{cross:variation:term}
d\langle f_\lambda(r),G\rangle_t = 4\varepsilon (\chi(t))^2 f'_{\lambda,-}(r_t) \langle e_t, \nabla_p\mathcal V(z_t) - \nabla_p\mathcal V(z'_t) \rangle \, dt.
\end{equation}
We estimate the gradient difference by exploiting the structure $\mathcal{V} = \mathcal{V}_0 + \mathfrak{Q}$.
First, consider the baseline function $\mathcal{V}_0$. From \eqref{eq:V0-general-quadratic}, the gradient of $\mathcal{V}_{0}$ with respect to $p$ is linear:
\[
\nabla_p \mathcal{V}_0(q,p) = p + \frac{\gamma}{2}q.
\]
Thus, the difference is
\begin{equation}\label{take:norm:1}
\nabla_p \mathcal{V}_0(z_t) - \nabla_p \mathcal{V}_0(z'_t) = \Delta p_t + \frac{\gamma}{2}\Delta q_t.
\end{equation}
Recall that $\mathbf R_t = \Delta q_t + \gamma^{-1}\Delta p_t$, which implies $\Delta p_t = \gamma(\mathbf R_t - \Delta q_t)$. Substituting this back:
\begin{equation}\label{take:norm:2}
\Delta p_t + \frac{\gamma}{2}\Delta q_t = \gamma(\mathbf R_t - \Delta q_t) + \frac{\gamma}{2}\Delta q_t = \gamma \mathbf R_t - \frac{\gamma}{2}\Delta q_t.
\end{equation}
Taking the norms in \eqref{take:norm:1}-\eqref{take:norm:2} and comparing with the distance $r_t = \theta|\Delta q_t| + |\mathbf R_t|$:
\begin{equation}\label{combine:estimate:1}
|\nabla_p \mathcal{V}_0(z_t) - \nabla_p \mathcal{V}_0(z'_t)| 
\le \gamma |\mathbf R_t| + \frac{\gamma}{2}|\Delta q_t|
\le \gamma \max\left\{1, \frac{1}{2\theta}\right\} \left(|\mathbf R_t| + \theta|\Delta q_t|\right)
= \gamma \max\left\{1, (2\theta)^{-1}\right\} r_t.
\end{equation}
Next, for the perturbation term $\mathfrak{Q}(z)=\frac12 z^\top \mathsf A z$,
we compute the $p$-gradient explicitly. Writing $\mathsf A$ in block form
with respect to $z=(q,p)$,
\[
\mathsf A=
\begin{pmatrix}
\mathsf A_{qq} & \mathsf A_{qp}\\
\mathsf A_{pq} & \mathsf A_{pp}
\end{pmatrix},
\qquad \mathsf A_{qp}=\mathsf A_{pq}^\top,
\]
we have
\[
\nabla_p \mathfrak Q(q,p)=\mathsf A_{pq}\,q+\mathsf A_{pp}\,p.
\]
Hence
\[
\left|\nabla_p \mathfrak Q(z_t)-\nabla_p \mathfrak Q(z'_t)\right|
\le \|\mathsf A_{pq}\|_{\mathrm{op}}\,|\Delta q_t|+\|\mathsf A_{pp}\|_{\mathrm{op}}\,|\Delta p_t|
\le \left(\|\mathsf A_{pq}\|_{\mathrm{op}}+\|\mathsf A_{pp}\|_{\mathrm{op}}\right)\,|z_t-z'_t|.
\]
Using the norm equivalence $r_t\ge k_1|z_t-z'_t|$ (Lemma~\ref{lem:r-equivalent}),
we obtain
\begin{equation}\label{combine:estimate:2}
|\nabla_p \mathfrak Q(z_t)-\nabla_p \mathfrak Q(z'_t)|
\le \frac{C_{\mathfrak Q}}{k_1}\,r_t,
\qquad
C_{\mathfrak Q}:=\|\mathsf A_{pp}\|_{\mathrm{op}}+\|\mathsf A_{pq}\|_{\mathrm{op}}.
\end{equation}
Combining the estimates \eqref{combine:estimate:1} and \eqref{combine:estimate:2} yields
\[
|\nabla_p \mathcal{V}(z_t) - \nabla_p \mathcal{V}(z'_t)|
\le \left(\gamma \max\left\{1,(2\theta)^{-1}\right\} + \frac{C_{\mathfrak{Q}}}{k_1}\right) r_t
= \gamma \bar{C}_{\mathcal V}\, r_t.
\]
Substituting this bound into the cross-variation term \eqref{cross:variation:term} gives
\begin{equation*}
  \left|d\langle f_\lambda(r),G\rangle_t\right|
  \le
  4\gamma\,\varepsilon \bar{C}_{\mathcal{V}}
  \left(\chi(t)\right)^2\, r_t\, f'_{\lambda,-}(r_t)\,dt.
\end{equation*}
This matches the form stated in Lemma~\ref{lem:drift-decomp}. 

Absorbing this contribution into the drift term in \eqref{eq:rho-ito-pre},
we conclude that
\[
  e^{ct}\rho_t
  \le \rho_0+\gamma\int_0^t e^{cs}K_s\,ds + M_t,
\]
where $M_t$ is the continuous local martingale defined in
\eqref{eq:Mt-explicit}, and
$K_t$ satisfies \eqref{K:t:ineq}.
\end{proof}

\subsection{Proof of Proposition~\ref{prop:regional-contractivity}}\label{app:regional-contractivity}
\begin{proof}
Fix $\xi>0$ and abbreviate $(z_t,z_t')=\left(z_t^\xi,z_t^{\prime,\xi}\right)$,
$r_t=r(z_t,z_t')$, $G_t=1+\varepsilon\mathcal V(z_t)+\varepsilon\mathcal V(z_t')$,
$\rho_t=f_\lambda(r_t)G_t$, and $\chi(t)=\chi_\xi(t)$.
Assume throughout that \eqref{eq:alpha-small-drift-prop} holds, so that
$\delta_\alpha\ge \kappa_{\mathrm{adjust}}\frac{\eta_0}{1+\eta_0}>0$.
By Lemma~\ref{lem:drift-decomp}, for any $c\in\R$,
\begin{equation}\label{eq:reg-semimart}
  e^{ct}\rho_t \le \rho_0+\gamma\int_0^t e^{cs}K_s\,ds + M_t,
\end{equation}
where $M_t$ is a continuous local martingale and $K_t$ is bounded from above by the
right-hand side in Lemma~\ref{lem:drift-decomp}.
We bound $K_t$ on the two regions $r_t\le R_1(\lambda)$ and $r_t>R_1(\lambda)$.

\medskip
\noindent\textbf{1) Region $r_t\le R_1(\lambda)$.}
Split further into the events $\{|\mathbf R_t|\ge\xi\}$ (reflection active) and
$\{|\mathbf R_t|<\xi\}$ (reflection inactive).

\smallskip
\noindent\emph{(i) If $r_t\le R_1(\lambda)$ and $|\mathbf R_t|\ge\xi$, then $\chi(t)=1$.}
On $(0,R_1(\lambda))$, $f_\lambda$ is $C^2$. Moreover, the construction of $f_\lambda$
(with the choice of $\varphi_\lambda$ in \eqref{eq:phi-def}) ensures that the
combination of the $f_\lambda''$-term, the ``bad'' linear drift term, and the
cross-variation contribution is strictly negative. More precisely, there exists
$C_{\mathrm{conc}}>0$ (depending only on the profile construction) such that
for a.e.\ $t$ on this event,
\begin{align*}
4\gamma^{-2} f_\lambda''(r_t)\,G_t
  +\left(\theta|\mathbf R_t|-\delta_\alpha\,\theta|Z_t|\right)
     f'_{\lambda,-}(r_t)\,G_t
  + 4\varepsilon \bar C_{\mathcal V}\, r_t f'_{\lambda,-}(r_t)
  \le - C_{\mathrm{conc}}\, f_\lambda(r_t)\,G_t .
\end{align*}
Using $(\lambda,D)$-admissibility and choosing $0<c\le c_0$ and
$0<\varepsilon\le\varepsilon_0$ small enough (so that the remaining $\varepsilon$-
and $c$-terms are dominated), we obtain $K_t\le0$ here. In particular
$K_t\le C_{\mathrm{reg}}\xi G_t$ holds.

\smallskip
\noindent\emph{(ii) If $r_t\le R_1(\lambda)$ and $|\mathbf R_t|<\xi$, then $\chi(t)=0$.}
In this case the $f_\lambda''$-term and the $\left(\chi(t)\right)^2 r_t f'_{\lambda,-}(r_t)$-term
vanish. Moreover, since $|\mathbf R_t|<\xi$ and $r_t\le R_1(\lambda)$, we have the crude bound
\[
  \left(\theta|\mathbf R_t|-\delta_\alpha\,\theta|Z_t|\right)
  f'_{\lambda,-}(r_t)\,G_t
  \le \theta\,|\mathbf R_t|\,\sup_{[0,R_1(\lambda)]} f'_{\lambda,-}\; G_t
  \le C\,\xi\,G_t .
\]
with $C$ independent of $\xi$. The remaining Lyapunov and $c$-terms are bounded by
a constant multiple of $G_t$ (since $f_\lambda$ is bounded on $[0,R_1(\lambda)]$), and hence
can be absorbed into $C_{\mathrm{reg}}\xi G_t$ after enlarging $C_{\mathrm{reg}}$.
Therefore, $K_t\le C_{\mathrm{reg}}\xi G_t$ also holds on this event.

\medskip
\noindent\textbf{2) Region $r_t>R_1(\lambda)$.}
By construction, $f_\lambda$ is constant on $[R_1(\lambda),\infty)$, so
$f'_{\lambda,-}(r_t)=0$ and $f_\lambda''(r_t)=0$ a.e.\ on $\{r_t>R_1(\lambda)\}$.
Hence Lemma~\ref{lem:drift-decomp} reduces to
\[
  K_t
  \le \gamma^{-1}\varepsilon f_\lambda(r_t)\left[\cL_\alpha\mathcal V(z_t)+\cL_\alpha\mathcal V(z_t')\right]
     +\gamma^{-1}c f_\lambda(r_t)G_t .
\]
Using $(\lambda,D)$-admissibility,
\[
  \cL_\alpha\mathcal V(z_t)+\cL_\alpha\mathcal V(z_t')
  \le 2\gamma(d+D)-\gamma\lambda\left(\mathcal V(z_t)+\mathcal V(z_t')\right).
\]
Thus
\[
  K_t \le f_\lambda(r_t)\left[2\varepsilon(d+D)+\gamma^{-1}c
     +\varepsilon(\gamma^{-1}c-\lambda)\left(\mathcal V(z_t)+\mathcal V(z_t')\right)\right].
\]
Choose $c_0<\gamma\lambda$ so that $\gamma^{-1}c-\lambda<0$.
Since $r_t>R_1(\lambda)$ implies $|z_t-z_t'|\gtrsim r_t$ by Lemma~\ref{lem:r-equivalent}, at
least one of $|z_t|,|z_t'|$ is $\gtrsim r_t$, and coercivity \eqref{eq:V-coercive}
yields $\mathcal V(z_t)+\mathcal V(z_t')\gtrsim r_t^2$.
Taking $R_1(\lambda)$ (already a free cutoff in the construction) large enough, the negative
term dominates and we get $K_t\le0$ on $\{r_t>R_1(\lambda)\}$, and hence again
$K_t\le C_{\mathrm{reg}}\xi G_t$.

\medskip
Combining the two regions gives \eqref{eq:Kt-xi-bound}. Taking expectations in
\eqref{eq:reg-semimart} (with localization to remove $M_t$) yields
\eqref{eq:rho-exp-bound}. Finally, for each fixed $t$, $\sup_{s\le t}\E[G_s^\xi]<\infty$
and does not blow up as $\xi\downarrow0$ (the marginals are the same HFHR dynamics).
Hence letting $\xi\downarrow0$ gives
$\limsup_{\xi\downarrow0}\E[e^{ct}\rho_t^\xi]\le \E[\rho_0]$. This completes the proof.
\end{proof}

\subsection{Proof of Theorem~\ref{thm:master-contraction}}\label{app:master-contraction}
\begin{proof}
Let $(Z_t,Z_t')$ be the coupling used in Lemma~\ref{lem:drift-decomp} (reflection/synchronous switching),
and set $\rho_t:=\rho_{\mathcal V}(Z_t,Z_t')$, $r_t:=r(Z_t,Z_t')$.
By Lemma~\ref{lem:drift-decomp}, for any $c>0$ the process $e^{ct}\rho_t$ is a supermartingale as long as
the drift term $K_t$ in $d(e^{ct}\rho_t)=e^{ct}K_t\,dt+dM_t$ satisfies $K_t\le 0$ a.s.
Let us choose
\begin{equation}\label{choice:varepsilon}
\varepsilon := \frac{4c}{\gamma(d+D)}.
\end{equation}
We verify $K_t\le 0$ in three regions.

\paragraph{Region I: $r_t\ge R_1(\lambda)$ (large distance).}
Since $f_\lambda$ is constant on $[R_1(\lambda),\infty)$, we have $f'_\lambda=f''_\lambda=0$.
Using $(\lambda,D)$-admissibility,
\[
\cL_\alpha \mathcal V \le \gamma(d+D)-\gamma\lambda \mathcal V.
\]
The choice of $R_1(\lambda)$ in \eqref{eq:R1-condition-v2} 
implies that whenever $r_t\ge R_1(\lambda)$,
\begin{equation}\label{eq:V-large-v2}
\mathcal V(Z_t)+\mathcal V(Z_t') \;\ge\; \frac{12}{5}\frac{d+D}{\lambda},
\end{equation}
and therefore
\begin{equation}\label{eq:LV-neg-v2}
\cL_\alpha\mathcal V(Z_t)+\cL_\alpha\mathcal V(Z_t')
\;\le\; -\frac{1}{6}\gamma\lambda\left(\mathcal V(Z_t)+\mathcal V(Z_t')\right).
\end{equation}
With \eqref{eq:LV-neg-v2} and $\varepsilon=4c/(\gamma(d+D))$ as in \eqref{choice:varepsilon}, we obtain
$K_t\le 0$ in this region provided
\[
c \le \frac{\gamma}{16}\lambda.
\]

\paragraph{Region II: $r_t<R_1(\lambda)$ and reflection is active.}
On this event $\chi(t)=1$. In Lemma~\ref{lem:drift-decomp}, the term
\[
-\delta_\alpha\,\theta|Z_t|\,f'_{\lambda,-}(r_t)\,G_t
\]
is \emph{non-positive} and can be dropped. 
Construction of $\varphi_\lambda$ ensures the following cancellation condition holds for all $r \in (0, R_1(\lambda))$:
\begin{equation}\label{eq:cancellation-condition}
  4\gamma^{-2}\varphi_\lambda'(r)
  + \left(\theta + 4\varepsilon\bar C_{\mathcal V}\right) r\, \varphi_\lambda(r)
  \le 0.
\end{equation}
Note that this condition is defined using the distance $r$ to cover the worst-case drift since $|\mathbf R_t| \le r_t$.

Recall the bound for $K_t$ from Equation~\eqref{K:t:ineq} in Lemma~\ref{lem:drift-decomp}. 
Since $G_t \ge 1$ and $f'_{\lambda,-}(r_t) = \varphi_\lambda(r_t)g_\lambda(r_t) \ge 0$, we can upper bound the cross-variation term by multiplying it by $G_t$:
\[
4\varepsilon \bar C_{\mathcal V} r_t f'_{\lambda,-}(r_t)
\le 
4\varepsilon \bar C_{\mathcal V} r_t f'_{\lambda,-}(r_t) G_t.
\]
Since $r_t<R_1(\lambda)$, we have for a.e.\ $r\in(0,R_1(\lambda))$ that
$f'_{\lambda,-}(r)=\varphi_\lambda(r)g_\lambda(r)$ and
$f_\lambda''(r)=\varphi_\lambda'(r)g_\lambda(r)+\varphi_\lambda(r)g_\lambda'(r)$.
Substituting these identities into \eqref{K:t:ineq} yields
\begin{align*}
K_t
&\le
\left[
  4\gamma^{-2}\varphi_\lambda'(r_t)
  + \theta |\mathbf R_t|\,\varphi_\lambda(r_t)
  + 4\varepsilon\bar C_{\mathcal V}\,r_t\,\varphi_\lambda(r_t)
\right] g_\lambda(r_t)\, G_t  \\
&\quad + 4\gamma^{-2}\varphi_\lambda(r_t)g_\lambda'(r_t)G_t
+ \gamma^{-1}\varepsilon f_\lambda(r_t)\left[\cL_\alpha\mathcal V(z_t)+\cL_\alpha\mathcal V(z'_t)\right]
+\gamma^{-1}c f_\lambda(r_t)G_t.
\end{align*}
Using $|\mathbf R_t|\le r_t$, the bracketed term is bounded above by
\[
\left[4\gamma^{-2}\varphi_\lambda'(r_t)+(\theta+4\varepsilon\bar C_{\mathcal V})\,r_t\,\varphi_\lambda(r_t)\right]g_\lambda(r_t),
\]
which is non-positive by \eqref{eq:cancellation-condition}, and hence can be dropped.
For the remaining terms, by $(\lambda,D)$-admissibility \eqref{eq:generic-drift-again},
\[
\cL_\alpha \mathcal V(z) \le \gamma(d+D - \lambda \mathcal V(z)) \le \gamma(d+D),
\]
so that $\cL_\alpha\mathcal V(z_t)+\cL_\alpha\mathcal V(z'_t)\le 2\gamma(d+D)$.
With $\varepsilon = \frac{4c}{\gamma(d+D)}$, we obtain
\[
\gamma^{-1}\varepsilon f_\lambda(r_t)\cdot 2\gamma(d+D)
= 8\gamma^{-1}c\,f_\lambda(r_t)
\le 8\gamma^{-1}c\,f_\lambda(r_t)\,G_t,
\]
and therefore
\[
K_t \le 4\gamma^{-2}\varphi_\lambda(r_t)\,g'_\lambda(r_t)\,G_t
+ 9\gamma^{-1}c\,f_\lambda(r_t)\,G_t.
\]
By the definition of $g_\lambda$ in \eqref{eq:gint-def},
\[
4\gamma^{-2}\varphi_\lambda(r)\,g'_\lambda(r) = -9\gamma^{-1}c\,\Phi_\lambda(r),
\]
and since $f_\lambda(r)\le \Phi_\lambda(r)$ for $r\in[0,R_1(\lambda)]$ we conclude $K_t\le 0$
as long as $g_\lambda(r)\ge 1/2$ on $[0,R_1(\lambda)]$, i.e.
\[
\frac{9}{4}\,c\,\gamma\int_0^{R_1(\lambda)}\frac{\Phi_\lambda(s)}{\varphi_\lambda(s)}\,ds \;\le\; \frac12.
\]
As in \cite[Theorem~2.3]{Eberle}, the above holds whenever
\[
c \;\le\; \frac{\gamma}{384}\min\left\{
\sqrt{\Lambda_\alpha(\lambda)}\,e^{-\Lambda_\alpha(\lambda)}\frac{L_{\mathrm{eff}}(\alpha)}{\gamma^2},\;
\sqrt{\Lambda_\alpha(\lambda)}\,e^{-\Lambda_\alpha(\lambda)}
\right\}.
\]

\paragraph{Region III: $r_t<R_1(\lambda)$ and synchronous coupling is active.}
In this regime $\chi(t)=0$ and the reflection-noise terms vanish.
The drift bound from Lemma~\ref{lem:drift-decomp} contains the dissipative part
\[
  -\delta_\alpha\,\theta|Z_t|\,f'_\lambda(r_t)\,G_t,
\]
which yields the constraint
\[
c \;\le\; \frac{\gamma}{18}\,\delta_\alpha\,\inf_{s\in(0,R_1(\lambda)]}\frac{s\,\varphi_\lambda(s)}{\Phi_\lambda(s)}.
\]
Using $\delta_\alpha\ge \kappa_{\mathrm{adjust}}\frac{\eta_0}{1+\eta_0}$ and choosing
$\eta_0=\left(\Lambda_0(\lambda)\right)^{-1}$, we estimate the Gaussian ratio as follows.
Since $s\mapsto s\varphi_\lambda(s)/\Phi_\lambda(s)$ is decreasing on $(0,R_1(\lambda)]$,
\[
  \inf_{s\in(0,R_1(\lambda)]}\frac{s\,\varphi_\lambda(s)}{\Phi_\lambda(s)}
  =\frac{R_1(\lambda)\,\varphi_\lambda(R_1(\lambda))}{\Phi_\lambda(R_1(\lambda))}.
\]
Moreover, $\Phi_\lambda(R_1(\lambda))\le \int_0^\infty \varphi_\lambda(s)\,ds
= \frac{\sqrt\pi}{2}\left(\frac{8}{L_{\mathrm{eff}}(\alpha)\,R_1^2(\lambda)}\right)^{1/2}$,
and hence
\[
  \inf_{s\in(0,R_1(\lambda)]}\frac{s\,\varphi_\lambda(s)}{\Phi_\lambda(s)}
  \;\ge\; \frac{2}{\sqrt\pi}\,\sqrt{\Lambda_\alpha(\lambda)}\,e^{-\Lambda_\alpha(\lambda)}.
\]
Therefore, it suffices to impose
\[
c \;\le\; \frac{\gamma}{18}\,\delta_\alpha\,\frac{2}{\sqrt\pi}\,
\sqrt{\Lambda_\alpha(\lambda)}\,e^{-\Lambda_\alpha(\lambda)}.
\]

Taking the minimum of the admissible bounds from the three regions yields
\eqref{eq:rate-explicit} and hence the contraction estimate.
\end{proof}


\subsection{Proof of Corollary~\ref{cor:convergence-V0}}\label{app:convergence-V0}
\begin{proof}
By Proposition~\ref{prop:V0-drift-HFHR}, for $\alpha\in[0,\alpha_0]$ we have 
\[
  \cL_\alpha \mathcal V_0 \le \gamma\left(d + A_\alpha - \lambda_\alpha \mathcal V_0\right).
\]
Hence $\mathcal V_0$ is $(\lambda_\alpha,A_\alpha)$-admissible in the sense of
Definition~\ref{def:admissible-lyapunov}.
Fix $\kappa_{\mathrm{adjust}}\in(0,1)$ and assume \eqref{eq:alpha-small-drift} holds
(with $\eta_0=\left(\Lambda_0(\lambda)\right)^{-1}$ as chosen in Theorem~\ref{thm:master-contraction}).
Therefore, Theorem~\ref{thm:master-contraction} applies with $\mathcal V=\mathcal V_0$
and yields, for the corresponding semimetric $\rho_{\mathcal V_0,\alpha}$,
the contraction estimate
\[
  \cW_{\rho_{\mathcal V_0,\alpha}}(\mu P_t^\alpha,\nu P_t^\alpha)
  \le e^{-c_\alpha t}\,\cW_{\rho_{\mathcal V_0,\alpha}}(\mu,\nu),
  \qquad t\ge0.
\]

We next deduce existence and uniqueness of an invariant measure and exponential
convergence to it. Let
\[
  \mathcal P_{\mathcal V_0}(\R^{2d})
  :=\left\{\mu\ \text{probability measure on }\R^{2d}:\ \int_{\R^{2d}}\mathcal V_0\,d\mu<\infty\right\},
\]
equipped with $\cW_{\rho_{\mathcal V_0,\alpha}}$. As in
\cite[Corollary~2.6]{Eberle}, $\left(\mathcal P_{\mathcal V_0}(\R^{2d}),
\cW_{\rho_{\mathcal V_0,\alpha}}\right)$ is complete, and the Lyapunov drift
implies moment control along the semigroup: for $\mu\in\mathcal P_{\mathcal V_0}$,
\[
  \sup_{t\ge0}\int_{\mathbb{R}^{2d}} \mathcal V_0\,d(\mu P_t^\alpha)
  \le \max\left\{\int_{\mathbb{R}^{2d}} \mathcal V_0\,d\mu,\ \frac{d+A_\alpha}{\lambda_\alpha}\right\}
  <\infty.
\]

Fix $\mu_0\in\mathcal P_{\mathcal V_0}(\R^{2d})$ and set $\mu_t:=\mu_0P_t^\alpha$.
For $s>t$, by the semigroup property and the contraction,
\[
  \cW_{\rho_{\mathcal V_0,\alpha}}(\mu_s,\mu_t)
  =\cW_{\rho_{\mathcal V_0,\alpha}}\!\left((\mu_0P_{s-t}^\alpha)P_t^\alpha,\mu_0P_t^\alpha\right)
  \le e^{-c_\alpha t}\,
     \cW_{\rho_{\mathcal V_0,\alpha}}(\mu_0P_{s-t}^\alpha,\mu_0).
\]
The uniform moment bound above and the structure of $\rho_{\mathcal V_0,\alpha}$
imply $\sup_{u\ge0}\cW_{\rho_{\mathcal V_0,\alpha}}(\mu_0P_u^\alpha,\mu_0)<\infty$
(see \cite[Corollary~2.6]{Eberle}). Hence $(\mu_t)_{t\ge0}$ is a Cauchy family 
with respect to the metric $\cW_{\rho_{\mathcal V_0,\alpha}}$, and by the completeness 
of the space $\left(\mathcal P_{\mathcal V_0}(\R^{2d}), \cW_{\rho_{\mathcal V_0,\alpha}}\right)$, it converges to some $\pi_\alpha\in\mathcal P_{\mathcal V_0}(\R^{2d})$.

The limit $\pi_\alpha$ is invariant: for any $t\ge0$,
\[
  \pi_\alpha P_t^\alpha
  = \lim_{s\to\infty} \mu_s P_t^\alpha
  = \lim_{s\to\infty} \mu_{s+t}
  = \pi_\alpha.
\]
Uniqueness follows from contraction: if $\pi_\alpha'$ is another invariant
measure in $\mathcal P_{\mathcal V_0}$, then
\[
  \cW_{\rho_{\mathcal V_0,\alpha}}(\pi_\alpha,\pi_\alpha')
  = \cW_{\rho_{\mathcal V_0,\alpha}}(\pi_\alpha P_t^\alpha,\pi_\alpha' P_t^\alpha)
  \le e^{-c_\alpha t}\cW_{\rho_{\mathcal V_0,\alpha}}(\pi_\alpha,\pi_\alpha'),
\]
and letting $t\to\infty$ yields $\pi_\alpha=\pi_\alpha'$.
Taking $\nu=\pi_\alpha$ gives the stated convergence to equilibrium.
\end{proof}

\subsection{Proof of Lemma~\ref{lem:quadratic:bounds}}\label{proof:lem:quadratic:bounds}
\begin{proof}
Under Assumption~\ref{assump:potential}(ii), $\nabla U$ is $L$-Lipschitz, which implies the quadratic growth bound 
\begin{equation}\label{U:quadratic:bound}
U(q) \le U(0) + |\nabla U(0)||q| + \frac{L}{2}|q|^2. 
\end{equation}
Combining \eqref{U:quadratic:bound} with the explicit quadratic form of $\mathcal V_0$, we have
\begin{align}
c_1'\left(1+|q|^2+|p|^2\right)
  &\le 1+\mathcal V_0(q,p)\nonumber
  \\
  &\leq\left(\max(1, \mu_{\max}) + \sup_{q\in\mathbb{R}^{d}} \frac{U(q)}{1+|q|^2}\right)\left(1+|q|^2+|p|^2\right)
  \nonumber
  \\
  &\leq\left(\max(1, \mu_{\max}) + \sup_{q\in\mathbb{R}^{d}} \frac{U(0) + |\nabla U(0)||q| + \frac{L}{2}|q|^2}{1+|q|^2}\right)\left(1+|q|^2+|p|^2\right)
  \nonumber
  \\
  &\leq c_2'\left(1+|q|^2+|p|^2\right),
\end{align}
where 
$c_1' = \min(1, \mu_{\min})$ 
and $c_2' = \max(1, \mu_{\max}) + U(0)+\frac{L}{2}+\frac{1}{2}|\nabla U(0)|$,
where $\mu_{\min}$ and $\mu_{\max}$ are the smallest and largest eigenvalues of the symmetric matrix $M$ defined in \eqref{M:matrix}, i.e., the matrix associated with the quadratic form in $(q,p)$ appearing in \eqref{eq:V0-general-quadratic}; see \eqref{mu:min:max} for explicit formulas. 
The proof is complete.
\end{proof}


\subsection{Proof of Lemma~\ref{lem:rho-controls-W2}}\label{app:rho-controls-W2}
\begin{proof}
Let $\Gamma$ be any coupling of $(\mu,\nu)$. By Lemma~\ref{lem:r-equivalent},
\[
  |z-z'|^2 \le k_1^{-2}\, \left(r(z,z')\right)^2 .
\]
Set $r:=r(z,z')$.  Since $\varphi_\lambda$ is positive and nonincreasing, and
$g_\lambda(s)\ge g_*$ on $[0,R_1(\lambda)]$, for $0\le r\le R_1(\lambda)$ we have
\[
  f_\lambda(r)=\int_0^{r}\varphi_\lambda(s)g_\lambda(s)\,ds
  \ge g_*\int_0^r \varphi_\lambda(s)\,ds
  \ge g_*\,c_r\, r.
\]
Also, by definition $f_\lambda(r)=f_\lambda(R_1(\lambda))=c_0$ for all $r\ge R_1(\lambda)$.

\medskip\noindent\emph{Case 1: $r\le R_1(\lambda)$.}
Then $r^2\le R_1(\lambda) r$ and therefore
\[
  |z-z'|^2 \le k_1^{-2} r^2 \le k_1^{-2}R_1(\lambda) r
  \le \frac{k_1^{-2}R_1(\lambda)}{g_*c_r}\, f_\lambda(r).
\]
Since $\mathcal V\ge 1$, we have $1\le 1+\mathcal V(z)+\mathcal V(z')$, hence
\begin{align}\label{z:z:prime:case:1}
  |z-z'|^2
  \le \frac{k_1^{-2}R_1(\lambda)}{g_*c_r}\,
      f_\lambda(r)\left(1+\mathcal V(z)+\mathcal V(z')\right).
\end{align}

\medskip\noindent\emph{Case 2: $r>R_1(\lambda)$.}
Using \eqref{eq:V-quadratic-lower-new},
\[
  |z-z'|^2 \le 2|z|^2+2|z'|^2
  \le 4C_V\left(1+\mathcal V(z)+\mathcal V(z')\right).
\]
Since $f_\lambda(r)\ge c_0$ on $\{r>R_1(\lambda)\}$, we get
\begin{align}\label{z:z:prime:case:2}
  |z-z'|^2
  \le \frac{4C_V}{c_0}\,
      f_\lambda(r)\left(1+\mathcal V(z)+\mathcal V(z')\right).
\end{align}

\medskip
Combining both cases \eqref{z:z:prime:case:1}-\eqref{z:z:prime:case:2} yields
\[
  |z-z'|^2 \le C\, f_\lambda(r)\left(1+\mathcal V(z)+\mathcal V(z')\right),
  \quad
  C:=\max\left\{\frac{k_1^{-2}R_1(\lambda)}{g_*c_r},\,\frac{4C_V}{c_0}\right\}.
\]
Finally, since $\varepsilon\le 1$ and $\mathcal V\ge 0$,
\[
  1+\mathcal V(z)+\mathcal V(z')
  \le \frac1\varepsilon\left(1+\varepsilon\mathcal V(z)+\varepsilon\mathcal V(z')\right),
\]
which, together with the definition of $\rho_{\mathcal{V}}(z,z')$ in \eqref{eq:rho-V-def-again}, implies
\[
  |z-z'|^2 \le \frac{C}{\varepsilon}\,\rho_{\mathcal V}(z,z').
\]
Integrate w.r.t.\ $\Gamma$ and take the infimum over all couplings to obtain
$\mathcal W_2^2(\mu,\nu)\le C_\rho\,\mathcal W_{\rho_{\mathcal V}}(\mu,\nu)$
with $C_\rho:=C/\varepsilon$. The proof is complete.
\end{proof}

\subsection{Proof of Corollary~\ref{cor:W2-contraction}}\label{app:W2-contraction}
\begin{proof}
Let $c>0$ and $\varepsilon=\frac{4c}{\gamma(d+D)}$ be as in
\eqref{eq:epsilon-def}, and let $C_\rho$ be the constant in
Lemma~\ref{lem:rho-controls-W2} computed with this $\varepsilon$.
Applying Lemma~\ref{lem:rho-controls-W2} to the pair
$(\mu P_t^\alpha,\nu P_t^\alpha)$ gives
\[
  \mathcal W_2^2(\mu P_t^\alpha,\nu P_t^\alpha)
  \le C_\rho\,
      \cW_{\rho_{\mathcal V}}(\mu P_t^\alpha,\nu P_t^\alpha).
\]
Using the contraction property \eqref{eq:W-rho-contraction},
\[
  \cW_{\rho_{\mathcal V}}(\mu P_t^\alpha,\nu P_t^\alpha)
  \le e^{-ct}\,\cW_{\rho_{\mathcal V}}(\mu,\nu),
\]
we obtain
\[
  \mathcal W_2^2(\mu P_t^\alpha,\nu P_t^\alpha)
  \le C_\rho\,e^{-ct}\,\cW_{\rho_{\mathcal V}}(\mu,\nu).
\]
Taking square roots yields the claimed bound. The final statement follows by
choosing $\nu=\pi_\alpha$ whenever $\pi_\alpha$ exists and satisfies
$\int_{\mathbb R^{2d}} \mathcal V\,d\pi_\alpha<\infty$. This completes the proof.
\end{proof}

\subsection{Proof of Corollary~\ref{cor:W2-convergence-baseline}}\label{proof:cor:W2-convergence-baseline}
\begin{proof}
By Proposition~\ref{prop:V0-drift-HFHR}, $\mathcal V_0$ is $(\hat\lambda_\alpha,A_\alpha)$-admissible.
Applying Theorem~\ref{thm:master-contraction} with $\mathcal V=\mathcal V_0$ and
$(\lambda,D)=(\hat\lambda_\alpha,A_\alpha)$ yields a contraction rate $c_\alpha>0$
and the associated choice
\[
  \varepsilon_\alpha=\frac{4c_\alpha}{\gamma(d+A_\alpha)}.
\]
Let $\rho_{\mathcal V_0,\alpha}$ be the corresponding weighted semimetric, and let
$C_{\rho,\alpha}$ denote the constant from Lemma~\ref{lem:rho-controls-W2} associated
with $\mathcal V_0$ and computed with $\varepsilon=\varepsilon_\alpha$.
Then Corollary~\ref{cor:W2-contraction} yields, for any $\nu$ with finite $\mathcal V_0$-moment,
\[
  \mathcal W_2(\mu P_t^\alpha,\nu P_t^\alpha)
  \le C_{\rho,\alpha}^{1/2} e^{-\frac12 c_\alpha t}
     \left(\cW_{\rho_{\mathcal V_0,\alpha}}(\mu,\nu)\right)^{1/2}.
\]
By Corollary~\ref{cor:convergence-V0}, the invariant measure $\pi_\alpha$ exists, is unique, and satisfies
$\int_{\mathbb R^{2d}} \mathcal V_0\,d\pi_\alpha<\infty$. Taking $\nu=\pi_\alpha$ and using
$\pi_\alpha P_t^\alpha=\pi_\alpha$ completes the proof.
\end{proof}

\section{Proofs for the Results in Section~\ref{sec:acceleration}}
\label{app:acceleration-proofs}

\subsection{Proof of Lemma~\ref{lem:exact-drift-decomp}}\label{app:exact-drift-decomp}
\begin{proof}
Recall from \eqref{eq:Aprime-def} that the interaction operator is given by
$\mathcal A' = -\nabla U(q)\cdot\nabla_q$,
so that for any smooth test function $f$,
\[
  \mathcal A' f(q,p)
  = - \nabla U(q)\cdot\nabla_q f(q,p).
\]
By appling this to the Lyapunov function $\mathcal V_0$ defined in
\eqref{eq:V0-general-quadratic}, we get:
\[
  \mathcal V_0(q,p)
  = U(q)
    + \frac{\gamma^2}{4}
      \left(
        |q+\gamma^{-1}p|^2
        + |\gamma^{-1}p|^2
        - \lambda |q|^2
      \right).
\]

\smallskip\noindent
\emph{Step 1: Compute the $q$–gradient of $\mathcal V_0$.}
We first differentiate $\mathcal V_0$ with respect to $q$:
\[
  \nabla_q \mathcal V_0(q,p)
  = \nabla U(q)
    + \frac{\gamma^2}{4}\left(
        2(q+\gamma^{-1}p)
        - 2\lambda q
      \right),
\]
since $|\gamma^{-1}p|^2$ does not depend on $q$. Hence,
\[
  \nabla_q \mathcal V_0(q,p)
  = \nabla U(q)
    + \frac{\gamma^2}{2}
      \left(
        q + \gamma^{-1}p - \lambda q
      \right)
  = \nabla U(q)
    + \frac{\gamma^2}{2}
      \left(
        (1-\lambda) q + \gamma^{-1}p
      \right).
\]

\smallskip\noindent
\emph{Step 2: Apply $\mathcal A'$ to $\mathcal V_0$.}
By the definition of $\mathcal A'$, we obtain
\begin{align*}
  \mathcal A'\mathcal V_0(q,p)
  &= - \nabla U(q)\cdot\nabla_q \mathcal V_0(q,p)
  \\
  &= - \nabla U(q)\cdot
     \left[
       \nabla U(q)
       + \frac{\gamma^2}{2}
         \left(
           (1-\lambda) q + \gamma^{-1}p
         \right)
     \right]
  \\
  &= -|\nabla U(q)|^2
     - \frac{\gamma^2}{2}(1-\lambda)\,\nabla U(q)\cdot q
     - \frac{\gamma}{2}\,\nabla U(q)\cdot p.
\end{align*}
This is exactly the claimed identity \eqref{eq:exact-drift-identity}.
\end{proof}



\subsection{Proof of Lemma~\ref{lem:first-order-improvement}}\label{app:first-order-improvement}
\begin{proof}
Throughout the proof we write $z=(q,p)\in\R^{2d}$ and use the notation
$\langle x,y\rangle = x^\top y$ for the Euclidean inner product.

\medskip\noindent
\emph{Step 1: A limiting Ornstein--Uhlenbeck operator and a quadratic control of $U$ at infinity.}
By Assumption~\ref{assump:asymptotic-linear-drift}, there exist a symmetric positive definite matrix
$Q_\infty\in\R^{d\times d}$ and a nonincreasing function $\varrho:[0,\infty)\to[0,\infty)$ with $\varrho(r)\to 0$
as $r\to\infty$ such that
\begin{equation}\label{eq:gradU-asymptotic-proof}
\left|\nabla U(q)-Q_\infty q\right|
\le \varrho(|q|)\,|q|,
\qquad |q|\ge C_{\mathrm{linear}}.
\end{equation}
Define $r(q):=\nabla U(q)-Q_\infty q$ and, for $R\ge 1$,
recall from \eqref{defn:tail:modulus} the definition of the tail modulus:
\[
\rho_\nabla(R):=\sup_{|q|\ge R}\frac{|r(q)|}{|q|}.
\]
Then $\rho_\nabla(R)<\infty$ for $R\ge \max\{1,C_{\mathrm{linear}}\}$ and, since $\rho_\nabla(R)\le \varrho(R)$ for
$R\ge C_{\mathrm{linear}}$, we have $\rho_\nabla(R)\to 0$ as $R\to\infty$. Moreover, $\rho_\nabla(\cdot)$ is nonincreasing.

We now derive a quadratic control of $U(q)-\frac12\langle Q_\infty q,q\rangle$ at infinity.
Since $U$ satisfies Assumption~\ref{assump:potential}, $\nabla U$ is continuous. Hence
\[
B_{\mathrm{lin}}:=\sup_{|x|\le C_{\mathrm{linear}}}|r(x)|<\infty.
\]
Fix any $q\in\R^d$ with $q\neq 0$ and write $\theta:=q/|q|\in\mathbb S^{d-1}$. Define
\[
g_\theta(s):=U(s\theta)-\frac12\langle Q_\infty(s\theta),s\theta\rangle,\qquad s\ge 0.
\]
By the fundamental theorem of calculus,
\[
g_\theta(|q|)-g_\theta(0)=\int_0^{|q|}\langle r(s\theta),\theta\rangle\,ds.
\]
By splitting at $C_{\mathrm{linear}}$, we get:
\[
\left|g_\theta(|q|)-g_\theta(0)\right|
\le \int_0^{C_{\mathrm{linear}}}|r(s\theta)|\,ds+\int_{C_{\mathrm{linear}}}^{|q|}|r(s\theta)|\,ds
\le C_{\mathrm{linear}}B_{\mathrm{lin}}+\int_{C_{\mathrm{linear}}}^{|q|}\rho_\nabla(s)\,s\,ds,
\]
where we used $|r(s\theta)|\le \rho_\nabla(s)\,|s\theta|=\rho_\nabla(s)\,s$ for $s\ge C_{\mathrm{linear}}$.

Therefore, for all $|q|\ge C_{\mathrm{linear}}$,
\begin{equation}\label{eq:U-quadratic-control-tail}
\left|U(q)-U(0)-\frac12\langle Q_\infty q,q\rangle\right|
\le C_{\mathrm{linear}}B_{\mathrm{lin}}+\int_{C_{\mathrm{linear}}}^{|q|}\rho_\nabla(s)\,s\,ds .
\end{equation}

Consequently, for every $R\ge \max\{1,C_{\mathrm{linear}}\}$ and every $|q|\ge R$,
\begin{align*}
\int_{C_{\mathrm{linear}}}^{|q|}\rho_\nabla(s)s\,ds
=&\int_{C_{\mathrm{linear}}}^{R}\rho_\nabla(s)s\,ds+\int_{R}^{|q|}\rho_\nabla(s)s\,ds\\
\le &\int_{C_{\mathrm{linear}}}^{R}\rho_\nabla(s)s\,ds+\rho_\nabla(R)\int_{R}^{|q|}s\,ds
\le \int_{C_{\mathrm{linear}}}^{R}\rho_\nabla(s)s\,ds+\frac12\rho_\nabla(R)|q|^2,
\end{align*}
using that $\rho_\nabla$ is nonincreasing. Dividing by $1+|q|^2$ and taking the supremum over $|q|\ge R$ yields
\begin{equation}\label{eq:deltaU-bound-from-rho}
\delta_U(R):=\sup_{|q|\ge R}\frac{\left|U(q)-\frac12\langle Q_\infty q,q\rangle\right|}{1+|q|^2}
\le \frac{|U(0)|+C_{\mathrm{linear}}B_{\mathrm{lin}}+\int_{C_{\mathrm{linear}}}^{R}\rho_\nabla(s)s\,ds}{1+R^2}
+\frac12\,\rho_\nabla(R).
\end{equation}
To conclude, it remains to show that
\begin{align}\label{claim:limit:zero}
\frac{1}{R^2}\int_{C_{\mathrm{linear}}}^{R}\rho_\nabla(s)\,s\,ds \xrightarrow[R\to\infty]{} 0.
\end{align}
Fix any $\varepsilon>0$ and choose $S\ge C_{\mathrm{linear}}$ such that $\rho_\nabla(S)\le \varepsilon$ (possible since $\rho_\nabla(R)\to0$).
Then for all $R\ge S$,
\begin{align*}
\frac{1}{R^2}\int_{C_{\mathrm{linear}}}^{R}\rho_\nabla(s)\,s\,ds
\le &\frac{1}{R^2}\int_{C_{\mathrm{linear}}}^{S}\rho_\nabla(s)\,s\,ds
     +\frac{1}{R^2}\int_{S}^{R}\rho_\nabla(s)\,s\,ds\\
\le &\frac{1}{R^2}\int_{C_{\mathrm{linear}}}^{S}\rho_\nabla(s)\,s\,ds+\varepsilon\cdot\frac{R^2-S^2}{2R^2}
\le \frac{1}{R^2}\int_{C_{\mathrm{linear}}}^{S}\rho_\nabla(s)\,s\,ds+\frac{\varepsilon}{2}.
\end{align*}
Letting $R\to\infty$ gives $\limsup_{R\to\infty}\frac{1}{R^2}\int_{C_{\mathrm{linear}}}^{R}\rho_\nabla(s)\,s\,ds\le \varepsilon/2$,
and since $\varepsilon>0$ is arbitrary, the limit is $0$ and the claim \eqref{claim:limit:zero} is proved.

Next, introduce the ``limiting'' kinetic Ornstein--Uhlenbeck drift operator
\[
\mathcal A_\infty f(q,p)
:= \left\langle p,\nabla_q f(q,p)\right\rangle
- \left\langle \gamma p + Q_\infty q,\nabla_p f(q,p)\right\rangle,
\qquad (q,p)\in\R^{2d},
\]
and write
\[
\mathcal A_0 = \mathcal A_\infty + \mathcal A_{\mathrm{pert}},
\qquad
\mathcal A_{\mathrm{pert}} f(q,p)
:= -\langle r(q),\nabla_p f(q,p)\rangle .
\]

\medskip\noindent
\emph{Step 2: $\mathcal A_\infty$ is invertible on quadratic polynomials.}
Let $\mathsf{Q}_2$ denote the vector space of quadratic polynomials on $\R^{2d}$.
For $\mathcal M(z)=\frac12 z^\top \mathsf K z$ with $\mathsf K=\mathsf K^\top$, one has
\[
  (\mathcal A_\infty\mathcal M)(z)
  = \frac12\,z^\top\left(B^\top \mathsf K + \mathsf K B\right)z,
  \qquad
  B :=
  \begin{pmatrix}
    0       & I_d\\
    -Q_\infty & -\gamma I_d
  \end{pmatrix}.
\]
Since $Q_\infty$ is positive definite and $\gamma>0$, $B$ is Hurwitz.
Hence the Lyapunov equation $B^\top \mathsf K + \mathsf K B = C$ has a unique symmetric
solution for any symmetric $C$ (see, e.g., \cite{horn2012matrix}).
Therefore, the linear map $\mathsf Q_2\ni \mathcal M\mapsto \mathcal A_\infty \mathcal M\in \mathsf Q_2$
is an isomorphism.

\medskip\noindent
\emph{Step 3: An explicit expansion for $\mathcal A'\mathcal V_0$ and an explicit upper bound.}
Recall from \eqref{eq:V0-general-quadratic} that 
\[
  \mathcal V_0(q,p)
  = U(q)
    + \frac{\gamma^2}{4}\left(
         |q+\gamma^{-1}p|^2
         + |\gamma^{-1}p|^2
         - \lambda|q|^2
      \right).
\]
A direct computation yields the explicit $q$--gradient
\begin{equation}\label{eq:gradqV0-explicit}
  \nabla_q \mathcal V_0(q,p)
  = \nabla U(q) + \frac{\gamma^2}{2}(1-\lambda)\,q + \frac{\gamma}{2}\,p .
\end{equation}
Recall from Lemma~\ref{lem:exact-drift-decomp}, insert $\nabla U(q)=Q_\infty q+r(q)$ into \eqref{eq:exact-drift-identity}.
Define the quadratic form
\begin{equation}\label{eq:Q-explicit}
  Q(q,p)
  := -|Q_\infty q|^2
     -\frac{\gamma^2}{2}(1-\lambda)\,\langle Q_\infty q,q\rangle
     -\frac{\gamma}{2}\,\langle Q_\infty q,p\rangle,
\end{equation}
and the remainder:
\begin{align}
  \mathcal R(q,p)
  :=\;& -2\langle Q_\infty q,r(q)\rangle - |r(q)|^2
       -\frac{\gamma^2}{2}(1-\lambda)\,\langle r(q),q\rangle
       -\frac{\gamma}{2}\,\langle r(q),p\rangle .\label{eq:R1-explicit}
\end{align}
Then
\begin{equation}\label{eq:AprimeV0-split-explicit}
  \mathcal A'\mathcal V_0(q,p) = Q(q,p) + \mathcal R(q,p).
\end{equation}

We now provide an upper bound on $|\mathcal R|$ in the tail.
Using $|Q_\infty q|\le \lambda_{\max}(Q_\infty)|q|$, the elementary bounds
$|q|\le |z|$, $|q||p|\le \frac12(|q|^2+|p|^2)\le |z|^2$,
and the tail estimate $|r(q)|\le \rho_\nabla(|q|)\,|q|$ valid for $|q|\ge C_{\mathrm{linear}}$,
for all $|q|\ge \max\{1,C_{\mathrm{linear}}\}$ and all $p\in\R^d$,
\begin{equation}\label{eq:R1-tail}
  |\mathcal R(q,p)|\le \rho_1(|q|)\,(1+|z|^2),
\end{equation}
where $\rho_1$ is \emph{explicitly defined} by
\begin{equation}\label{eq:rho1-def}
  \rho_1(r)
:= \left(4\lambda_{\max}(Q_\infty)+\gamma^2|1-\lambda|+\gamma\right)\rho_\nabla(r)+(\rho_\nabla(r))^2,
  \qquad r\ge 0.
\end{equation}
In particular, $\rho_1(r)\to 0$ as $r\to\infty$.

\medskip\noindent
\emph{Step 4: Construct $\mathcal M$ and obtain an explicit lower bound $\underline c_{\mathrm{imp}}$.}
Define a quadratic form
\begin{equation}\label{eq:B1-def}
  B_1(q,p) := -Q(q,p) - \left(
      \frac12|p|^2 + \frac{\gamma}{2}\langle q,p\rangle
      + \frac12\left\langle\left(Q_\infty+\frac{\gamma^2}{2}(1-\lambda)I_d\right)q,q\right\rangle
    \right).
\end{equation}
Let $C_{B_1}:=\nabla^2 B_1$ so that $B_1(z)=\frac12 z^\top C_{B_1} z$.
By Step 2, there exists a unique quadratic polynomial $\mathcal M\in \mathsf Q_2$
such that
\begin{equation}\label{eq:AinfM-B1}
  \mathcal A_\infty\mathcal M(z) = B_1(z),
  \qquad z\in\R^{2d}.
\end{equation}
Equivalently, writing $\mathcal M(z)=\frac12 z^\top \mathsf K z$ with $\mathsf K=\mathsf K^\top$,
the matrix $\mathsf K$ is the unique symmetric solution of the Lyapunov equation
\begin{equation}\label{eq:K-def}
  B^\top \mathsf K + \mathsf K B = C_{B_1},
  \qquad
  B :=
  \begin{pmatrix}
    0       & I_d\\
    -Q_\infty & -\gamma I_d
  \end{pmatrix}.
\end{equation}

Define
\begin{equation}\label{eq:CM-CDelta-explicit-viaK}
  C_{\mathcal M}
  := \frac{\|\mathsf K\|_{\mathrm{op}}}{2},
  \qquad
  C_\Delta
  := 2d\,\|\mathsf K\|_{\mathrm{op}}.
\end{equation}
Then \eqref{eq:M-growth}--\eqref{eq:lapM-growth} hold.

Writing $\mathsf K=\begin{pmatrix}\mathsf K_{qq}&\mathsf K_{qp}\\ \mathsf K_{pq}&\mathsf K_{pp}\end{pmatrix}$, we have
$\nabla_p\mathcal M(q,p)=\mathsf K_{pq}q+\mathsf K_{pp}p$, and hence
\begin{equation}\label{eq:gradpM-explicit}
  |\nabla_p\mathcal M(q,p)|
  \le \left(\|\mathsf K_{pq}\|_{\mathrm{op}}+\|\mathsf K_{pp}\|_{\mathrm{op}}\right)\,(|q|+|p|)
  \le \left(\|\mathsf K_{pq}\|_{\mathrm{op}}+\|\mathsf K_{pp}\|_{\mathrm{op}}\right)\,(1+|z|).
\end{equation}

Next, using $\mathcal A_0=\mathcal A_\infty+\mathcal A_{\mathrm{pert}}$ and
$\mathcal A'\mathcal V_0 = Q+\mathcal R$, we can compute that
\begin{equation}\label{eq:A0M-plus-AprimeV0-split}
    \begin{aligned}
      \mathcal A_0\mathcal M(z)+\mathcal A'\mathcal V_0(z)
      &= \mathcal A_\infty\mathcal M(z) + \mathcal A_{\mathrm{pert}}\mathcal M(z) + Q(z) + \mathcal R(z) \\
      &= B_1(z) + \mathcal A_{\mathrm{pert}}\mathcal M(z) + Q(z) + \mathcal R(z)
         \qquad\text{by \eqref{eq:AinfM-B1}}\\
      &= -\Xi(z) + \mathcal A_{\mathrm{pert}}\mathcal M(z) + \mathcal R(z),
    \end{aligned}
\end{equation}
where
\[
  \Xi(z)=\Xi(q,p)
  :=\frac12|p|^2 + \frac{\gamma}{2}\langle q,p\rangle
      + \frac12\left\langle\left(Q_\infty+\frac{\gamma^2}{2}(1-\lambda)I_d\right)q,q\right\rangle .
\]

For $|q|\ge C_{\mathrm{linear}}$, using \eqref{eq:gradpM-explicit} and
$|r(q)|\le \rho_\nabla(|q|)\,|q|$, we can compute that
\begin{align}\label{eq:ApertM-bound}
  |\mathcal A_{\mathrm{pert}}\mathcal M(z)|
  &= |\langle r(q),\nabla_p\mathcal M(z)\rangle| \nonumber\\
  &\le \rho_\nabla(|q|)\,|q|\cdot \left(\|\mathsf K_{pq}\|_{\mathrm{op}}+\|\mathsf K_{pp}\|_{\mathrm{op}}\right)(1+|z|)\nonumber\\
  &\le 2\left(\|\mathsf K_{pq}\|_{\mathrm{op}}+\|\mathsf K_{pp}\|_{\mathrm{op}}\right)\rho_\nabla(|q|)\,(1+|z|^2).
\end{align}
Combining \eqref{eq:R1-tail} and \eqref{eq:ApertM-bound} yields: for all $|q|\ge \max\{1,C_{\mathrm{linear}}\}$,
\begin{equation}\label{eq:A0M-plus-AprimeV0-pre}
  \mathcal A_0\mathcal M(z)+\mathcal A'\mathcal V_0(z)
  \le -\Xi(z) + \left[2(\|\mathsf K_{pq}\|_{\rm op}+\|\mathsf K_{pp}\|_{\rm op})\rho_\nabla(|q|)
  +\rho_1(|q|)\right](1+|z|^2).
\end{equation}

We now derive coercivity bounds on $\Xi$.
The lower bound will be used to absorb the tail perturbation in \eqref{eq:A0M-plus-AprimeV0-pre},
while the upper bound will be used later to relate $\Xi$ to $\mathcal V_0$ with explicit constants.
Set
\[
  a_{\min}:=\lambda_{\min}(Q_\infty)+\frac{\gamma^2}{2}(1-\lambda),\qquad
  a_{\max}:=\lambda_{\max}(Q_\infty)+\frac{\gamma^2}{2}(1-\lambda).
\]
Hence, we obtain the global bounds
\begin{equation}\label{eq:Xi-coercive}
  \Xi(z)\ge \underline{a}\,|z|^2,
  \qquad
  \Xi(z)\le \overline{a}\,|z|^2,
\end{equation}
with
\begin{equation}\label{eq:alpha-bounds}
  \underline{a}
  :=\frac14\left(a_{\min}+1-\sqrt{(a_{\min}-1)^2+\gamma^2}\right),\qquad
  \overline{a}
  :=\frac14\left(a_{\max}+1+\sqrt{(a_{\max}-1)^2+\gamma^2}\right).
\end{equation}

\medskip\noindent
\emph{Step 4.5: A closed-form cutoff ensuring absorption.}
To obtain a \emph{computable} cutoff, we recall the tail modulus defined in \eqref{defn:tail:modulus}:
$\rho_\nabla(R) := \sup_{|q|\ge R}\frac{|r(q)|}{|q|}$.
By definition, for any $|q|\ge R$, we have $|r(q)| \le \rho_\nabla(|q|)|q|$.
Using the expression for $\rho_1(r)$ in \eqref{eq:rho1-def}, the condition for the perturbation term to be absorbed is
\begin{equation}\label{term:in:bracket:unified}
  \sup_{r\ge R}\left[
    \left(2\left(\|\mathsf K_{pq}\|_{\mathrm{op}}+\|\mathsf K_{pp}\|_{\mathrm{op}}\right) + 4\lambda_{\max}(Q_\infty)+\gamma^2|1-\lambda|+\gamma\right)\rho_\nabla(r) + (\rho_\nabla(r))^2
  \right]
  \le \frac{5}{16}\underline{a}.
\end{equation}
Let $A$ be the coefficient of the linear term:
\[
  A:= 2\left(\|\mathsf K_{pq}\|_{\mathrm{op}}+\|\mathsf K_{pp}\|_{\mathrm{op}}\right)+4\lambda_{\max}(Q_\infty)+\gamma^2|1-\lambda|+\gamma.
\]
A sufficient condition for \eqref{term:in:bracket:unified} to hold is $A\rho_\nabla(R)+(\rho_\nabla(R))^2\le \frac{5}{16}\underline{a}$.
Consider the quadratic equation $x^2+Ax-\frac{5}{16}\underline{a} = 0$. The positive root is
\[
  \rho_\star:=\frac{-A+\sqrt{A^2+\frac54\,\underline{a}}}{2}>0.
\]
We now set
\begin{equation}\label{eq:R0-choice}
  R_0:=\inf\{R\ge \max\{1,C_{\mathrm{linear}}\}:\ \rho_\nabla(R)\le \rho_\star\}.
\end{equation}
Then $R_0<\infty$ (since $\rho_\nabla(R)\to 0$). By our choice of $R_0$, for all $z=(q,p)$ with $|q|\ge R_0$, the bracketed term
in \eqref{term:in:bracket:unified} is bounded by $\frac{5}{16}\underline{a}$. Using $\Xi(z)\ge \underline{a} |z|^2$ and $1+|z|^2 \le 2|z|^2$ (since $R_0\ge 1$), we explicitly obtain:
\begin{equation}\label{eq:absorb-tail}
  \left[2\left(\|\mathsf K_{pq}\|_{\mathrm{op}}+\|\mathsf K_{pp}\|_{\mathrm{op}}\right)\rho_\nabla(|q|)+\rho_1(|q|)\right](1+|z|^2)
  \le \frac{5}{8}\,\Xi(z).
\end{equation}

Plugging \eqref{eq:absorb-tail} into \eqref{eq:A0M-plus-AprimeV0-pre} yields
for $|q|\ge R_0$:
\begin{equation}\label{eq:outside-ball-Xi}
  \mathcal A_0\mathcal M(z)+\mathcal A'\mathcal V_0(z)\le -\frac{3}{8}\Xi(z).
\end{equation}

We now convert \eqref{eq:outside-ball-Xi} into a drift improvement of the form $-\;\underline c_{\mathrm{imp}}\mathcal V_0 + C$ with the explicit rate claimed in \eqref{eq:cimp-lemma}.
Observe that by the definitions of $\mathcal V_0$ and $\Xi$, we have the exact identity:
\[
  \mathcal V_0(z) - \Xi(z) = U(q) - \frac12\langle Q_\infty q,q\rangle.
\]
For $|q|\ge R_0 \ge 1$, the definition of $\delta_U(R_0)$ and the fact
$1+|q|^2 \le 2|q|^2 \le 2|z|^2$ imply
\[
  U(q) - \frac12\langle Q_\infty q,q\rangle
  \le \delta_U(R_0)(1+|q|^2)
  \le 2\delta_U(R_0)|z|^2.
\]
Using the upper bound $\Xi(z)\le \overline{a}|z|^2$ from \eqref{eq:Xi-coercive}, we obtain
\[
  \mathcal V_0(z)
  \le \Xi(z) + 2\delta_U(R_0)|z|^2
  \le \left(\overline{a} + 2\delta_U(R_0)\right)|z|^2.
\]
Finally, using the lower bound $|z|^2 \le \frac{1}{\underline{a}}\Xi(z)$ from \eqref{eq:Xi-coercive}, we arrive at the explicit control
\begin{equation}\label{eq:V0-controlled-by-Xi}
  \mathcal V_0(z)
  \le \frac{\overline{a} + 2\delta_U(R_0)}{\underline{a}}\,\Xi(z),
  \qquad |q|\ge R_0.
\end{equation}

Combining \eqref{eq:outside-ball-Xi} and \eqref{eq:V0-controlled-by-Xi} yields
\[
  \mathcal A_0\mathcal M(z)+\mathcal A'\mathcal V_0(z)
  \le -\frac{3}{8}\cdot \frac{\underline{a}}{\overline{a}+2\delta_U(R_0)}\,\mathcal V_0(z),
  \qquad |q|\ge R_0.
\]
Recall from \eqref{eq:alpha-bounds} that $4\underline{a} = a_{\min}+1-\sqrt{(a_{\min}-1)^2+\gamma^2}$ and $4\overline{a} = a_{\max}+1+\sqrt{(a_{\max}-1)^2+\gamma^2}$. Multiplying the numerator and denominator of the coefficient by $4$, we recover exactly the constant $\underline c_{\mathrm{imp}}$ defined in \eqref{eq:cimp-lemma}:
\[
  \frac{3}{8}\cdot \frac{4\underline{a}}{4\overline{a}+8\delta_U(R_0)} = \underline c_{\mathrm{imp}}.
\]

\medskip\noindent
\emph{Step 5: Control on the region $|q|\le R_0$ and global extension.}

Fix $R_0$ as in \eqref{eq:R0-choice}. Since $\nabla U$ is continuous, the function
$r(q)=\nabla U(q)-Q_\infty q$ is continuous. Hence
\[
  B_0:=\sup_{|q|\le R_0}|r(q)|<\infty.
\]
For $|q|\le R_0$, we bound the perturbation term using \eqref{eq:gradpM-explicit}:
\[
  |\mathcal A_{\mathrm{pert}}\mathcal M(q,p)|
  =|\langle r(q),\mathsf K_{pq}q+\mathsf K_{pp}p\rangle|
  \le B_0\|\mathsf K_{pq}\|_{\mathrm{op}}R_0 + B_0\|\mathsf K_{pp}\|_{\mathrm{op}}|p|.
\]
Moreover, since $|q|\le R_0$ and $|r(q)|\le B_0$, the remainder term \eqref{eq:R1-explicit} satisfies
\[
  |\mathcal R(q,p)|
  \le 2|Q_\infty q|\,|r(q)| + |r(q)|^2
      +\frac{\gamma^2}{2}|1-\lambda|\,|r(q)|\,|q|
      +\frac{\gamma}{2}|r(q)|\,|p|
  \le C_{R_0}^{(1)} + C_{R_0}^{(2)}|p|,
\]
where we can take
\[
  C_{R_0}^{(1)}
  := 2\,\lambda_{\max}(Q_\infty)\,R_0\,B_0 + B_0^2
     + \frac{\gamma^2}{2}|1-\lambda|\,R_0\,B_0,
  \qquad
  C_{R_0}^{(2)}:=\frac{\gamma}{2}B_0.
\]

Combining this with
\[
  |\mathcal A_{\mathrm{pert}}\mathcal M(q,p)|
  \le B_0\|\mathsf K_{pq}\|_{\mathrm{op}}R_0 + B_0\|\mathsf K_{pp}\|_{\mathrm{op}}|p|,
\]
we obtain, for $|q|\le R_0$,
\[
  |\mathcal A_{\mathrm{pert}}\mathcal M(q,p)|+|\mathcal R(q,p)|
  \le A_0 + L_0\,|p|,
\]
with
\[
  A_0:=B_0\|\mathsf K_{pq}\|_{\mathrm{op}}R_0 + C_{R_0}^{(1)},
  \qquad
  L_0:=B_0\|\mathsf K_{pp}\|_{\mathrm{op}} + C_{R_0}^{(2)}.
\]
Using Young's inequality $L_0|p|\le \eta |p|^2+\frac{L_0^2}{4\eta}$, for any $\eta\in(0,1)$ we get
\[
  |\mathcal A_{\mathrm{pert}}\mathcal M(q,p)|+|\mathcal R(q,p)|
  \le \eta |p|^2 + C_{R_0,\eta},
  \qquad |q|\le R_0,
\]
where an explicit choice of $C_{R_0,\eta}$ is given by
\[
  C_{R_0,\eta}:=A_0+\frac{L_0^2}{4\eta}.
\]

Recalling \eqref{eq:A0M-plus-AprimeV0-split},
we deduce that for $|q|\le R_0$,
\[
  \mathcal A_0\mathcal M(z)+\mathcal A'\mathcal V_0(z)
  \le -\Xi(z) + \eta |p|^2 + C_{R_0,\eta}.
\]
Using $\Xi(z)\ge \underline{a}(|q|^2+|p|^2)$ from \eqref{eq:Xi-coercive} and choosing
$\eta:=\underline{a}/2$, we get
\[
  \mathcal A_0\mathcal M(z)+\mathcal A'\mathcal V_0(z)
  \le -\frac12\,\Xi(z) + C_{R_0},
  \qquad |q|\le R_0,
\]
with $C_{R_0}:=C_{R_0,\underline{a}/2}$.

Together with \eqref{eq:outside-ball-Xi} (valid on $|q|\ge R_0$), we have the global bound
\[
  \mathcal A_0\mathcal M(z)+\mathcal A'\mathcal V_0(z)
  \le -\frac{3}{8}\Xi(z) + C_{R_0},
  \qquad z\in\mathbb R^{2d}.
\]

Finally, we convert $\Xi$ into $\mathcal V_0$ as in \eqref{eq:V0-controlled-by-Xi} on $|q|\ge R_0$,
while on $|q|\le R_0$ we use the identity $\mathcal V_0-\Xi = U(q)-\frac12\langle Q_\infty q,q\rangle$
and the bound $|U(q)-\frac12\langle Q_\infty q,q\rangle|\le \delta_U(1)(1+R_0^2)$ (cf. \eqref{eq:deltaU-bound-from-rho}) to conclude that
$\mathcal V_0(z)\le \Xi(z)+C_{R_0}'$ on $|q|\le R_0$.
This yields \eqref{eq:improvement-condition} for all $z$, with the same $\underline c_{\mathrm{imp}}$
as in \eqref{eq:cimp-lemma} and with $C_{\mathrm{imp}}$ defined as in the lemma statement. The finiteness of $C_{\mathrm{imp}}$ holds since for each fixed $q$, the map 
\[
    p\mapsto
    \mathcal A_0\mathcal M(q,p) + \mathcal A'\mathcal V_0(q,p)
    + \underline c_{\mathrm{imp}}\mathcal V_0(q,p)
\]
is a concave quadratic polynomial in $p$. Indeed, for fixed $q$, the $p$-quadratic coefficient matrix of
$\mathcal A_0\mathcal M + \mathcal A'\mathcal V_0 + \underline c_{\rm imp}\mathcal V_0$
is negative definite (uniformly in $|q|\le R_0$), and hence the supremum over $p$ is finite. This completes the proof.
\end{proof}



\subsection{Proof of Lemma~\ref{lem:Valpha-V0-equivalence}}\label{app:Valpha-V0-equivalence}
\begin{proof}
Recall from \eqref{eq:Valpha-def} that $\mathcal V_\alpha=\mathcal V_0+\alpha\mathcal M$.
From \eqref{eq:V0-equivalent} and $U\ge0$, we have
\[
  1+\mathcal V_0(q,p)
  \;\ge\;
  c_1\left(1+U(q)+|q|^2+|p|^2\right)
  \;\ge\;
  c_1\left(1+|q|^2+|p|^2\right),
\]
and hence
\begin{equation}\label{eq:qp2-controlled-by-V0}
  1+|q|^2+|p|^2
  \;\le\;
  \frac{1}{c_1}\,\left(1+\mathcal V_0(q,p)\right).
\end{equation}

By Lemma~\ref{lem:first-order-improvement} (growth bound \eqref{eq:M-growth}),
\[
  |\mathcal M(q,p)|
  \;\le\;
  C_{\mathcal M}\left(1+|q|^2+|p|^2\right).
\]
Combining with \eqref{eq:qp2-controlled-by-V0} yields
\begin{equation}\label{eq:M-controlled-by-V0}
  |\mathcal M(q,p)|
  \;\le\;
  \frac{C_{\mathcal M}}{c_1}\,\left(1+\mathcal V_0(q,p)\right).
\end{equation}

Therefore, for any $\alpha\ge0$,
\[
  |\alpha\mathcal M(q,p)|
  \;\le\;
  \alpha\,\frac{C_{\mathcal M}}{c_1}\,\left(1+\mathcal V_0(q,p)\right).
\]
Let
$\alpha_*:=\frac{c_1}{2C_{\mathcal M}}$.
Then for all $\alpha\in[0,\alpha_*]$ we have
$\alpha\,\frac{C_{\mathcal M}}{c_1}\le \frac12$, and thus
\begin{align*}
  1+\mathcal V_\alpha(q,p)
  &= 1+\mathcal V_0(q,p)+\alpha\mathcal M(q,p) \\
  &\ge 1+\mathcal V_0(q,p)-|\alpha\mathcal M(q,p)| \\
  &\ge \left(1-\alpha\frac{C_{\mathcal M}}{c_1}\right)\left(1+\mathcal V_0(q,p)\right)
  \;\ge\;
  \frac12\,\left(1+\mathcal V_0(q,p)\right),
\end{align*}
and similarly,
\begin{align*}
  1+\mathcal V_\alpha(q,p)
  &= 1+\mathcal V_0(q,p)+\alpha\mathcal M(q,p) \\
  &\le 1+\mathcal V_0(q,p)+|\alpha\mathcal M(q,p)| \\
  &\le \left(1+\alpha\frac{C_{\mathcal M}}{c_1}\right)\left(1+\mathcal V_0(q,p)\right)
  \;\le\;
  \frac32\,\left(1+\mathcal V_0(q,p)\right).
\end{align*}
This proves \eqref{eq:Valpha-V0-equivalence}.
\end{proof}


\subsection{Proof of Lemma~\ref{lem:Valpha-drift-expansion}}\label{app:Valpha-drift-expansion}

\begin{proof}
We use $\cL_\alpha=\mathcal A_0+\alpha\mathcal A'+\alpha\Delta_q+\gamma\Delta_p$ from \eqref{eq:Lalpha-decomp}
and $\mathcal V_\alpha=\mathcal V_0+\alpha\mathcal M$ from \eqref{eq:Valpha-def} to write
\begin{equation}\label{eq:Valpha-expansion-fixed}
  \cL_\alpha\mathcal V_\alpha
  = \cL_0\mathcal V_0
    + \alpha\left(\cL_0\mathcal M+\mathcal A'\mathcal V_0+\Delta_q\mathcal V_0\right)
    + \alpha^2\left(\mathcal A'\mathcal M+\Delta_q\mathcal M\right),
\end{equation}
where $\cL_0=\mathcal A_0+\gamma\Delta_p$.
From \eqref{eq:L0-V0-Lyapunov}, we get
\begin{align}\label{L:0:V:0:bound}
  \cL_0\mathcal V_0 \le \gamma(d+A)-\lambda\mathcal V_0.
\end{align}

\medskip
\noindent\emph{The $\alpha$-term in \eqref{eq:Valpha-expansion-fixed}.}
Since $\cL_0\mathcal M=\mathcal A_0\mathcal M+\gamma\Delta_p\mathcal M$, we have
\[
  \cL_0\mathcal M+\mathcal A'\mathcal V_0
  = \mathcal A_0\mathcal M+\mathcal A'\mathcal V_0+\gamma\Delta_p\mathcal M
  \le C_{\mathrm{imp}}-\underline c_{\mathrm{imp}}\mathcal V_0+\gamma\Delta_p\mathcal M
\]
by \eqref{eq:improvement-condition}. Since $\mathcal M(z)=\frac12 z^\top \mathsf K z$ is quadratic, we have
$\nabla_{pp}^2\mathcal M=\mathsf K_{pp}$ and hence
\[
  \Delta_p\mathcal M=\mathrm{tr}(\mathsf K_{pp}).
\]
In particular, $\Delta_p\mathcal M$ is independent of $(q,p)$, so the term $\gamma\Delta_p\mathcal M$
can be absorbed into the constant. Therefore,
\[
  \cL_0\mathcal M+\mathcal A'\mathcal V_0
  \le \left(C_{\mathrm{imp}}+\gamma\,\mathrm{tr}(\mathsf K_{pp})\right)
      -\underline c_{\mathrm{imp}}\mathcal V_0 .
\]
Moreover, by Assumption~\ref{assump:potential}, $\nabla U$ is Lipschitz, so that $\nabla^2U$ exists a.e.
and $\|\nabla^2U(q)\|_{\mathrm{op}}\le L$ a.e. Hence, $|\Delta U(q)|\le dL$ a.e. Therefore,
\[
  |\Delta_q\mathcal V_0(q,p)|
  =\left|\Delta U(q)+\frac{\gamma^2}{2}d(1-\lambda)\right|
  \le dL+\frac{\gamma^2}{2}d|1-\lambda|
  =K_\Delta
  \quad\text{a.e.}
\]
Combining these bounds yields
\begin{align}\label{alpha:term:bound}
  \cL_0\mathcal M+\mathcal A'\mathcal V_0+\Delta_q\mathcal V_0
  \le \left(C_{\mathrm{imp}}+\gamma\,\mathrm{tr}(\mathsf K_{pp})+K_\Delta\right)
      -\underline c_{\mathrm{imp}}\mathcal V_0.
\end{align}

\medskip
\noindent\emph{The $\alpha^2$-term in \eqref{eq:Valpha-expansion-fixed}.}
Since $\mathcal A'=-\nabla U(q)\cdot\nabla_q$, using $|\nabla U(q)|\le L|q|+|\nabla U(0)|$
and \eqref{eq:gradM-growth}, we obtain
\[
  |\mathcal A'\mathcal M|
  \le |\nabla U(q)|\,|\nabla_q\mathcal M(q,p)|
  \le (L|q|+|\nabla U(0)|)\,C_{\mathcal M}(1+|q|+|p|).
\]
Using $(1+|q|+|p|)^2\le 3(1+|q|^2+|p|^2)$ and
$L|q|+|\nabla U(0)|\le (L+|\nabla U(0)|)(1+|q|)$ gives
\[
  |\mathcal A'\mathcal M|
  \le 3\,C_{\mathcal M}(L+|\nabla U(0)|)\,(1+|q|^2+|p|^2).
\]
By \eqref{eq:V0-equivalent} and $U\ge0$, $1+|q|^2+|p|^2\le c_1^{-1}(1+\mathcal V_0)$. Hence
\[
  |\mathcal A'\mathcal M|
  \le 3\,C_{\mathcal M}\frac{(L+|\nabla U(0)|)}{c_1}\,(1+\mathcal V_0).
\]
Finally, since $\Delta_q\mathcal M=\mathrm{tr}(\mathsf K_{qq})$ is a constant and $1+\mathcal V_0\ge1$,
\[
  |\Delta_q\mathcal M|
  =|\mathrm{tr}(\mathsf K_{qq})|
  \le |\mathrm{tr}(\mathsf K_{qq})|\,(1+\mathcal V_0).
\]
Therefore,
\begin{align}\label{alpha:2:term:bound}
  \mathcal A'\mathcal M+\Delta_q\mathcal M
  \le \left(|\mathrm{tr}(\mathsf K_{qq})|+3\,C_{\mathcal M}\frac{(L+|\nabla U(0)|)}{c_1}\right)(1+\mathcal V_0).
\end{align}

\medskip
Putting the $\cL_0\mathcal V_0$ bound \eqref{L:0:V:0:bound}, the $\alpha$-term bound \eqref{alpha:term:bound}, and the $\alpha^2$-term bound \eqref{alpha:2:term:bound}
into \eqref{eq:Valpha-expansion-fixed} yields \eqref{eq:Lalpha-Valpha-bound-fixed}
with $C_1,C_2$ given by \eqref{eq:C1C2-explicit-fixed}.
\end{proof}


\subsection{Proof of Proposition~\ref{prop:Valpha-drift}}\label{app:Valpha-drift}

\begin{proof}
We start from Lemma~\ref{lem:Valpha-drift-expansion}: for all $\alpha\in(0,1]$ and all $z=(q,p)\in\R^{2d}$,
\begin{equation}\label{eq:drift-inequality}
  \begin{aligned}
  \cL_\alpha\mathcal V_\alpha(z)
  &\le \gamma(d+A) - \lambda\mathcal V_0(z)
      + \alpha\left( C_1 - \underline c_{\mathrm{imp}}\mathcal V_0(z)\right)
      + C_2\alpha^2(1+\mathcal V_0(z)) \\
  &= \left[\gamma(d+A) + \alpha C_1 + C_2\alpha^2\right]
     - \left[\lambda+\alpha\underline c_{\mathrm{imp}}-C_2\alpha^2\right]\mathcal V_0(z).
  \end{aligned}
\end{equation}
Set
\[
  \varsigma(\alpha):=\lambda+\alpha\underline c_{\mathrm{imp}}-C_2\alpha^2.
\]

\medskip\noindent
\emph{Step 1: Compare $\mathcal V_0$ and $\mathcal V_\alpha$.}
By \eqref{eq:M-growth}, $|\mathcal M|\le C_{\mathcal M}(1+|q|^2+|p|^2)$.
By \eqref{eq:V0-equivalent}, $1+\mathcal V_0\ge c_1(1+|q|^2+|p|^2)$, and hence
$1+|q|^2+|p|^2\le c_1^{-1}(1+\mathcal V_0)$. Therefore
\[
  |\mathcal M|\le \widetilde C_{\mathcal M}\,(1+\mathcal V_0),
  \qquad \widetilde C_{\mathcal M}:=\frac{C_{\mathcal M}}{c_1}.
\]
Consequently, for $\alpha\in(0,1]$,
\[
  \mathcal V_\alpha
  = \mathcal V_0+\alpha\mathcal M
  \le \mathcal V_0+\alpha|\mathcal M|
  \le \mathcal V_0+\alpha \widetilde C_{\mathcal M}(1+\mathcal V_0)
  = (1+\alpha \widetilde C_{\mathcal M})\mathcal V_0+\alpha \widetilde C_{\mathcal M}.
\]
Rearranging gives the lower bound
\begin{equation}\label{eq:V0-lower-bound-refined}
  \mathcal V_0 \ge \frac{\mathcal V_\alpha - \alpha \widetilde C_{\mathcal M}}{1+\alpha \widetilde C_{\mathcal M}}.
\end{equation}

\medskip\noindent
\emph{Step 2: Drift bound in terms of $\mathcal V_\alpha$.}
Plugging \eqref{eq:V0-lower-bound-refined} into \eqref{eq:drift-inequality} yields
\begin{align*}
  \cL_\alpha\mathcal V_\alpha
  &\le \left[\gamma(d+A) + \alpha C_1 + C_2\alpha^2\right]
      - \varsigma(\alpha)\left(\frac{\mathcal V_\alpha - \alpha \widetilde C_{\mathcal M}}{1+\alpha \widetilde C_{\mathcal M}}\right)\\
  &= \left[\gamma(d+A) + \alpha C_1 + C_2\alpha^2
      + \frac{\alpha \widetilde C_{\mathcal M}\varsigma(\alpha)}{1+\alpha \widetilde C_{\mathcal M}}\right]
     - \frac{\varsigma(\alpha)}{1+\alpha \widetilde C_{\mathcal M}}\,\mathcal V_\alpha.
\end{align*}
Define
\[
  \lambda_\alpha := \frac{\varsigma(\alpha)}{1+\alpha \widetilde C_{\mathcal M}}.
\]

\smallskip\noindent
\emph{Constant term and admissibility.}
Using $\alpha\le 1$, $C_2\alpha^2\le \alpha(\alpha C_2)$, and
\[
  \frac{\alpha \widetilde C_{\mathcal M}\varsigma(\alpha)}{1+\alpha \widetilde C_{\mathcal M}}
  \le \alpha \widetilde C_{\mathcal M}\varsigma(\alpha)
  \le \alpha \widetilde C_{\mathcal M}(\lambda+\underline c_{\mathrm{imp}}),
\]
we obtain
\begin{align*}
  &\gamma(d+A) + \alpha C_1 + C_2\alpha^2
      + \frac{\alpha \widetilde C_{\mathcal M}\varsigma(\alpha)}{1+\alpha \widetilde C_{\mathcal M}}\\
  &\le \gamma(d+A)
     + \alpha\left[C_1+\alpha C_2 + \widetilde C_{\mathcal M}(\lambda+\underline c_{\mathrm{imp}})\right]
  = \gamma\left(d + A_\alpha'\right),
\end{align*}
where $A_\alpha'$ is given by \eqref{eq:Aalpha-explicit-fixed}. Hence
\[
  \cL_\alpha\mathcal V_\alpha \le \gamma\left(d + A_\alpha' - \lambda_\alpha \mathcal V_\alpha\right),
\]
i.e.\ $\mathcal V_\alpha$ is $(\lambda_\alpha,A_\alpha')$-admissible for $\cL_\alpha$.

\smallskip\noindent
\emph{Rate expansion.}
Since $\frac{1}{1+x}\ge 1-x$ for any $x\ge 0$,
\[
  \lambda_\alpha=\frac{\varsigma(\alpha)}{1+\alpha \widetilde C_{\mathcal M}}
  \ge \varsigma(\alpha)\,\left(1-\alpha \widetilde C_{\mathcal M}\right)
  = \left(\lambda+\alpha\underline c_{\mathrm{imp}}-C_2\alpha^2\right)\left(1-\alpha \widetilde C_{\mathcal M}\right).
\]
Expanding the product gives
\[
  \lambda_\alpha
  \ge \lambda
     + \alpha\left(\underline c_{\mathrm{imp}}-\lambda \widetilde C_{\mathcal M}\right)
     - \alpha^2\left(C_2+\underline c_{\mathrm{imp}}\widetilde C_{\mathcal M}\right)
     + \alpha^3 C_2\widetilde C_{\mathcal M}.
\]
Dropping the nonnegative cubic term yields
\[
  \lambda_\alpha
  \ge \lambda
     + \alpha\left(\underline c_{\mathrm{imp}}-\lambda \widetilde C_{\mathcal M}\right)
     - \alpha^2\left(C_2+\underline c_{\mathrm{imp}}\widetilde C_{\mathcal M}\right)
  = \lambda + \delta\,\alpha - C_\lambda\,\alpha^2,
\]
with $\delta$ and $C_\lambda$ defined in \eqref{eq:delta-explicit-again-fixed}--\eqref{eq:Clambda-explicit-fixed}.
The choice of $\alpha_1$ in \eqref{eq:alpha1-explicit} ensures the auxiliary constraints
(from Proposition~\ref{prop:V0-drift-HFHR} and the positivity requirement on $\lambda_\alpha$)
hold simultaneously, and hence the claim follows.
\end{proof}


\subsection{Proof of Theorem~\ref{thm:metric-acceleration}}
\label{app:metric-acceleration}

\begin{proof}
Fix $\alpha\in(0,\alpha_1]$ and recall from \eqref{L:eff} that $L_{\mathrm{eff}}(\alpha)=(1+\alpha\gamma)L$.
We use the metric parameter $\Lambda_\alpha(\lambda)$ defined in \eqref{eq:Lambda-metric-def}, namely
\[
  \Lambda_\alpha(\lambda)=J_2\,\frac{(1+\alpha\gamma)L}{\lambda},
  \qquad
  \Lambda_0:=\Lambda_0(\lambda)=J_2\frac{L}{\lambda}.
\]
Also set $\mathcal S_h:=1-\frac{1}{2\Lambda_0}>0$ and
\[
  \Delta_\Lambda:=J_2L\frac{\delta-\gamma\lambda}{\lambda^2}>0.
\]

\medskip
By Proposition~\ref{prop:Valpha-drift}, there exist $\alpha_1>0$ and $C_\lambda\ge 0$ such that
for all $\alpha\in(0,\alpha_1]$ one can choose $\lambda_\alpha>0$ satisfying
\[
  \lambda_\alpha \;\ge\; \underline\lambda_\alpha
  := \lambda+\delta\alpha-C_\lambda\alpha^2.
\]
Define the proxy
\[
  \bar\Lambda_\alpha
  := J_2\,\frac{(1+\alpha\gamma)L}{\underline\lambda_\alpha}.
\]
Since $\lambda_\alpha\ge\underline\lambda_\alpha$ and $\Lambda_\alpha(\lambda)$ is decreasing in $\lambda$,
\begin{equation}\label{eq:Lambda-upper-by-barLambda-repl}
  \Lambda_\alpha(\lambda_\alpha)
  = J_2\,\frac{(1+\alpha\gamma)L}{\lambda_\alpha}
  \le
  J_2\,\frac{(1+\alpha\gamma)L}{\underline\lambda_\alpha}
  = \bar\Lambda_\alpha.
\end{equation}

\medskip
Let $D:=\delta-\gamma\lambda>0$ and define
\[
  g(\alpha):=\frac{1+\alpha\gamma}{\underline\lambda_\alpha}
  =\frac{1+\alpha\gamma}{\lambda+\delta\alpha-C_\lambda\alpha^2}.
\]
A direct computation gives
\begin{align}\label{follow:one}
  g(\alpha)-g(0)
  =\frac{1+\alpha\gamma}{\underline\lambda_\alpha}-\frac{1}{\lambda}
  =\frac{-D\alpha+C_\lambda\alpha^2}{\lambda\,\underline\lambda_\alpha}.
\end{align}
Choose
\begin{equation}\label{eq:alpha3a-explicit-repl}
  \alpha_{3,a}:=\min\left\{\alpha_1,\ 1,\ \frac{D}{4C_\lambda}\ \text{(if $C_\lambda>0$)},\ \sqrt{\frac{\lambda}{2C_\lambda}}\ \text{(if $C_\lambda>0$)}\right\},
\end{equation}
with the convention that the terms involving $C_\lambda$ are removed when $C_\lambda=0$.
Then for all $\alpha\in(0,\alpha_{3,a}]$ we have
\begin{align}\label{follow:two}
  -D\alpha+C_\lambda\alpha^2 \le -\frac{3D}{4}\alpha,
  \qquad
  \underline\lambda_\alpha\ge \frac{\lambda}{2},
\end{align}
and therefore it follows from \eqref{follow:one} and \eqref{follow:two} that
\[
  g(\alpha)-g(0)
  \le
  -\frac{3D}{4}\alpha\cdot\frac{2}{\lambda^2}
  \le
  -\frac{D}{8\lambda^2}\alpha.
\]
Multiplying by $J_2L$ and using $g(0)=1/\lambda$ yields
\[
  \bar\Lambda_\alpha
  =J_2L\,g(\alpha)
  \le
  J_2L\left(\frac{1}{\lambda}-\frac{D}{8\lambda^2}\alpha\right)
  =\Lambda_0-\frac{\Delta_\Lambda}{8}\alpha.
\]
Combining with \eqref{eq:Lambda-upper-by-barLambda-repl} proves (i) with the explicit choice
$c_\Lambda:=\Delta_\Lambda/8$.

\medskip
Recall \eqref{eq:hlambda-def}.
Since $\Lambda_0>\frac12$, we have $h'(\Lambda_0)<0$ and
\[
  h'(\Lambda_0)=-h(\Lambda_0)\left(1-\frac{1}{2\Lambda_0}\right)=-h(\Lambda_0)\mathcal S_h.
\]
Moreover
\[
  h''(\Lambda)=\frac{\Lambda^2-\Lambda-\frac14}{\Lambda^{3/2}}\,e^{-\Lambda}.
\]
Define the (finite) constant
\[
  M_h:=\sup_{\Lambda\in[\Lambda_0/2,\Lambda_0]}|h''(\Lambda)|.
\]
Let $t:=c_\Lambda\alpha$.
To ensure $\Lambda_0-t\in[\Lambda_0/2,\Lambda_0]$ and that the quadratic remainder is dominated by the linear term,
set
\begin{equation}\label{eq:alpha3b-explicit-repl}
  \alpha_{3,b}:=\min\left\{\frac{\Lambda_0}{2c_\Lambda},\ \frac{h(\Lambda_0)\mathcal S_h}{M_h\,c_\Lambda}\right\}.
\end{equation}
Finally define
\begin{equation}\label{eq:alpha3-explicit}
  \alpha_{\mathrm{metric,acc}}:=\min\{\alpha_{3,a},\alpha_{3,b}\}.
\end{equation}
For $\alpha\in(0,\alpha_{\mathrm{metric,acc}}]$, part (i) implies
$\Lambda_\alpha(\lambda_\alpha)\le \Lambda_0-t$ with $t=c_\Lambda\alpha$.
Since $h$ is decreasing on $[\Lambda_0/2,\Lambda_0]$, this yields
\[
  h(\Lambda_\alpha(\lambda_\alpha))\ge h(\Lambda_0-t).
\]
A second-order Taylor expansion of $h$ at $\Lambda_0$ with Lagrange remainder yields
\[
  h(\Lambda_0-t)
  \ge
  h(\Lambda_0)+|h'(\Lambda_0)|\,t-\frac12 M_h t^2
  \ge
  h(\Lambda_0)+\frac12 h(\Lambda_0)\mathcal S_h\,t,
\]
where the last inequality uses $t\le h(\Lambda_0)\mathcal S_h/M_h$
(by the definition of $\alpha_{3,b}$). Hence
\[
  h(\Lambda_\alpha(\lambda_\alpha))
  \ge
  h(\Lambda_0)\left(1+\frac{\mathcal S_h}{2}\,c_\Lambda\,\alpha\right).
\]
Therefore, for any $c_3<c_3^*:=\frac{\mathcal S_h}{2}c_\Lambda$, we have
\[
  \widetilde\Lambda_{3,\alpha}(\lambda_\alpha)
  =\kappa_{\mathrm{adjust}}\,h(\Lambda_\alpha(\lambda_\alpha))
  \ge
  \kappa_{\mathrm{adjust}}\,h(\Lambda_0)\,(1+c_3\alpha)
  =\widetilde\Lambda_{3,0}(\lambda)\,(1+c_3\alpha).
\]

Finally,
\begin{align*}
  \widetilde\Lambda_{2,\alpha}(\lambda_\alpha)
  = h(\Lambda_\alpha(\lambda_\alpha))\,\frac{L_{\mathrm{eff}}(\alpha)}{\gamma^2}
  &= \frac{L}{\gamma^2}(1+\alpha\gamma)\,h(\Lambda_\alpha(\lambda_\alpha))
  \\
  &\ge
  \widetilde\Lambda_{2,0}(\lambda)\,(1+\alpha\gamma)(1+c_3\alpha)
  \ge
  \widetilde\Lambda_{2,0}(\lambda)\left(1+(\gamma+c_3)\alpha\right),
\end{align*}
so that the bound holds with $c_2:=\gamma+c_3$.
This completes the proof.
\end{proof}

\subsection{Proof of Lemma~\ref{lem:branch-continuity}}\label{app:branch-continuity}
\begin{proof}
By the strict activity at $\alpha=0$ we have $\Delta(0)>0$.
By continuity of $\Delta$, there exists $\varepsilon>0$ such that $\Delta(\alpha)>0$ for all $\alpha\in[0,\varepsilon]$.
Hence $\alpha_{\mathrm{branch}}\ge\varepsilon>0$ and, by definition of $\alpha_{\mathrm{branch}}$,
for all $\alpha\in(0,\alpha_{\mathrm{branch}}]$ we have
\[
  \widetilde\Lambda_{1,\alpha}(\lambda_\alpha)\le
  \min\left\{\widetilde\Lambda_{2,\alpha}(\lambda_\alpha),\,\widetilde\Lambda_{3,\alpha}(\lambda_\alpha)\right\}.
\]
Recalling from Theorem~\ref{thm:master-contraction} (see \eqref{eq:rate-explicit}) that
\[
c(\lambda_\alpha)=\frac{\gamma}{384}\min\left\{\widetilde\Lambda_{1,\alpha}(\lambda_\alpha),\widetilde\Lambda_{2,\alpha}(\lambda_\alpha),\widetilde\Lambda_{3,\alpha}(\lambda_\alpha)\right\},
\]
the above inequality implies that the minimum is attained at $\widetilde\Lambda_{1,\alpha}(\lambda_\alpha)$, and thus
\[
c_\alpha=\frac{\gamma}{384}\widetilde\Lambda_{1,\alpha}(\lambda_\alpha)
\qquad\text{for all }\alpha\in(0,\alpha_{\mathrm{branch}}].
\]
\end{proof}


\subsection{Proof of Theorem~\ref{thm:lyapunov-acceleration}}\label{app:acceleration}

\begin{proof}
By Lemma~\ref{lem:branch-continuity}, for all $\alpha\in(0,\alpha_{\mathrm{branch}}]$,
the Lyapunov branch remains active. Hence
\[
  c_\alpha
  = \frac{\gamma}{384}\,\widetilde\Lambda_{1,\alpha}(\lambda_\alpha)
  = \frac{\gamma}{384}\,\frac{\lambda_\alpha L_{\mathrm{eff}}(\alpha)}{\gamma^{2}},
\]
where $L_{\mathrm{eff}}(\alpha)=(1+\gamma\alpha)L$. At $\alpha=0$,
\[
  c_0
  = \frac{\gamma}{384}\,\widetilde\Lambda_{1,0}(\lambda)
  = \frac{\gamma}{384}\,\frac{\lambda L}{\gamma^{2}}
  = \frac{L}{384\,\gamma}\,\lambda.
\]

By Proposition~\ref{prop:Valpha-drift}, for all $\alpha\in(0,\alpha_1]$,
\[
  \lambda_\alpha \;\ge\; \lambda + \delta\alpha - C_\lambda\alpha^2.
\]
Hence for $\alpha\in(0,\alpha_{\mathrm{branch}}]$,
\begin{align*}
  c_\alpha
  = \frac{L}{384\,\gamma}\,\lambda_\alpha(1+\gamma\alpha)
  &\ge \frac{L}{384\,\gamma}\,(\lambda + \delta\alpha - C_\lambda\alpha^2)(1+\gamma\alpha) \\
  &= \frac{L}{384\,\gamma}\left[
      \lambda + (\delta+\gamma\lambda)\alpha
      + (\gamma\delta-C_\lambda)\alpha^2
      - \gamma C_\lambda\alpha^3
     \right].
\end{align*}
Dropping the possibly positive term $(\gamma\delta)\alpha^2$ and using $\alpha\le 1$ to bound
$-\gamma C_\lambda\alpha^3\ge -\gamma C_\lambda\alpha^2$, we obtain
\[
  c_\alpha
  \ge \frac{L}{384\,\gamma}\left[
      \lambda + (\delta+\gamma\lambda)\alpha
      - (1+\gamma)C_\lambda\alpha^2
     \right]
  = c_0 + \tilde\kappa\,\alpha - C'\alpha^2,
\]
with $\tilde\kappa=\frac{L(\delta+\gamma\lambda)}{384\gamma}$ and
\[
    C'=\frac{L}{384\gamma}(1+\gamma)C_\lambda=\frac{L}{384\gamma}(1+\gamma)\left(C_2 + \widetilde C_{\mathcal M}\,\underline c_{\mathrm{imp}}\right).
\]
Choose $\alpha_{\mathrm{branch,acc}}:=\min\{\alpha_{\mathrm{branch}},1,\tilde\kappa/(2C')\}$.
Then for all $\alpha\in(0,\alpha_{\mathrm{branch,acc}}]$ we have $C'\alpha\le \tilde\kappa/2$. Hence, 
setting $\kappa:=\tilde\kappa/2$ gives \eqref{eq:acceleration-inequality}:
\[
  c_\alpha \ge c_0 + \frac{\tilde\kappa}{2}\alpha=c_0+\kappa\alpha.
\]
\end{proof}

\subsection{Proof of Corollary~\ref{cor:global-acceleration}}
\label{app:global-acceleration}

\begin{proof}
Recall from \eqref{eq:rate-explicit} in Theorem~\ref{thm:master-contraction} that 
\[
  c(\lambda)=\frac{\gamma}{384}\min\left\{\widetilde\Lambda_{1,\alpha}(\lambda),\widetilde\Lambda_{2,\alpha}(\lambda),\widetilde\Lambda_{3,\alpha}(\lambda)\right\}.
\]
For the HFHR dynamics, set
\[
  f_1(\alpha):=\frac{\gamma}{384}\widetilde\Lambda_{1,\alpha}(\lambda_\alpha),\qquad
  f_2(\alpha):=\frac{\gamma}{384}\widetilde\Lambda_{2,\alpha}(\lambda_\alpha),\qquad
  f_3(\alpha):=\frac{\gamma}{384}\widetilde\Lambda_{3,\alpha}(\lambda_\alpha),
\]
so that $c_\alpha=\min\{f_1(\alpha),f_2(\alpha),f_3(\alpha)\}$ and $c_0=\min\{f_1(0),f_2(0),f_3(0)\}$.
In particular, $f_i(0)\ge c_0$ for each $i$.

\medskip
\noindent\emph{Step 1: Lyapunov branch lower bound.}
By Theorem~\ref{thm:lyapunov-acceleration}, for all $\alpha\in(0,\alpha_{\mathrm{branch,acc}}]$,
\begin{equation}\label{eq:f1-global}
  f_1(\alpha)\;\ge\; c_0+\kappa\,\alpha.
\end{equation}

\medskip
\noindent\emph{Step 2: Metric branch lower bounds.}
By Theorem~\ref{thm:metric-acceleration}, for all $\alpha\in(0,\alpha_{\mathrm{metric,acc}}]$,
\[
  f_i(\alpha)\ge f_i(0)(1+c_i\alpha)=f_i(0)+c_i f_i(0)\alpha,
  \qquad i=2,3,
\]
with $c_2,c_3>0$. Using $f_i(0)\ge c_0$ yields
\begin{equation}\label{eq:fi-global}
  f_i(\alpha)\ge c_0 + c_0 c_i\,\alpha,\qquad i=2,3.
\end{equation}

\medskip
\noindent\emph{Step 3: Take the minimum.}
For all $\alpha\in(0,\alpha_{\mathrm{global}}]$, both \eqref{eq:f1-global} and \eqref{eq:fi-global} hold. Hence
\[
  c_\alpha=\min_{i=1,2,3} f_i(\alpha)
  \;\ge\;
  \min\left\{c_0+\kappa\alpha,\; c_0+c_0c_2\alpha,\; c_0+c_0c_3\alpha\right\}
  \;=\;
  c_0 + \kappa_{\mathrm{global}}\alpha,
\]
with $\kappa_{\mathrm{global}}:=\min\{\kappa,c_0c_2,c_0c_3\}>0$.
\end{proof}

\subsection{Proof of Corollary~\ref{cor:W2-acceleration}}
\label{app:W2-acceleration}

\begin{proof}
Fix $\alpha\in(0,\alpha_{\mathrm{W2}}]$ and choose $\lambda_\alpha=\underline\lambda_\alpha$.
By $\alpha\le \alpha_{\mathrm{pos}}$ and \eqref{eq:alpha-pos-explicit}, we have $\lambda_\alpha\ge\lambda_-=\lambda/2$.
Also, since $\lambda_\alpha=\lambda+\delta\alpha-C_\lambda\alpha^2\le \lambda+\delta\alpha\le \lambda_+$,
we have $\lambda_\alpha\in I_\lambda$.

Apply Lemma~\ref{lem:rho-controls-W2} with $\mathcal V=\mathcal V_\alpha$ and $\varepsilon=\varepsilon_\alpha$.
Using the uniform bounds $k_1^-,R_1(\lambda)^+,g_*^-,c_r^-,c_0^-$ and $C_V^{\mathrm{unif}}$,
and the definition of $\varepsilon^-$, the explicit constant \eqref{eq:Crho-explicit} in
Lemma~\ref{lem:rho-controls-W2} yields
\[
  \mathcal W_2^2(\mu,\nu)\le C_{\rho}^{\mathrm{unif}}\,\cW_{\rho_{\mathcal V_\alpha}}(\mu,\nu),
\]
and hence we obtain:
\[
  \mathcal W_2(\mu P_t^\alpha,\nu P_t^\alpha)
  \le \left(C_{\rho}^{\mathrm{unif}}\right)^{1/2}
       \left(\cW_{\rho_{\mathcal V_\alpha}}(\mu P_t^\alpha,\nu P_t^\alpha)\right)^{1/2}.
\]
By Corollary~\ref{cor:global-acceleration},
\[
  \cW_{\rho_{\mathcal V_\alpha}}(\mu P_t^\alpha,\nu P_t^\alpha)
  \le e^{-c_\alpha t}\,\cW_{\rho_{\mathcal V_\alpha}}(\mu,\nu),
\]
and therefore
\[
  \mathcal W_2(\mu P_t^\alpha,\nu P_t^\alpha)
  \le \left(C_{\rho}^{\mathrm{unif}}\right)^{1/2}
       e^{-\frac12 c_\alpha t}
       \left(\cW_{\rho_{\mathcal V_\alpha}}(\mu,\nu)\right)^{1/2}.
\]
This proves \eqref{eq:W2-acceleration} with $c_\alpha^{(2)}=\frac12 c_\alpha$.

Finally, since $\alpha_{\mathrm{W2}}\le \alpha_{\mathrm{global}}$, Corollary~\ref{cor:global-acceleration} gives
$c_\alpha\ge c_0+\kappa_{\mathrm{global}}\alpha$, hence
\[
  c_\alpha^{(2)}=\frac12 c_\alpha
  \ge \frac12\left(c_0+\kappa_{\mathrm{global}}\alpha\right)
  = c_0^{(2)}+\kappa^{(2)}\alpha,
\]
for all $\alpha\in(0,\alpha_{\mathrm{W2}}]$.
This completes the proof.
\end{proof}

\section{Proofs for the Results in Section~\ref{sec:case-study}}
\label{app:proof-case-study}

\subsection{Proof of Proposition~\ref{prop:MW-assumptions}}\label{app:MW-assumptions}
\begin{proof}
    We verify the properties sequentially based on the separable structure $U(q) = \sum_{i=1}^d v(q_i)$.

    \medskip\noindent
    \emph{(a) Regularity.}
    First, let us verify Assumption~\ref{assump:potential}(i)-(ii).
    The one-dimensional potential $v(s)$ has a continuous derivative $v'(s)$ satisfying $|v'(s) - v'(t)| \le |s-t|$ for all $s,t \in \R$ (since $v'$ is piecewise linear with slopes $\pm 1$ or $0$).
    For the $d$-dimensional potential, we sum the squares of the components:
    \[
        |\nabla U(q) - \nabla U(q')|^2 = \sum_{i=1}^d |v'(q_i) - v'(q'_i)|^2 \le \sum_{i=1}^d |q_i - q'_i|^2 = |q-q'|^2.
    \]
    Thus, $\nabla U$ is globally Lipschitz with constant $L=1$, independent of $d$. 

    \medskip\noindent
    \emph{(b) Dissipativity.}
    We verify Assumption~\ref{assump:potential}(iii).
    Set
    \[
      \bar\lambda:=\frac{1}{4+\gamma^2}\in\left(0,\frac14\right],
    \]
    and for $s\in\mathbb R$ define
    \[
      \Delta_{\bar\lambda}(s)
      :=\bar\lambda\left(v(s)+\frac{\gamma^2}{4}s^2\right)-\frac12\,s\,v'(s).
    \]
    We claim that $\sup_{s\in\mathbb R}\Delta_{\bar\lambda}(s)\le A_1(\gamma)$ with
    \[
      A_1(\gamma):=\frac{\gamma^4+6\gamma^2+16}{4(\gamma^4+10\gamma^2+24)}.
    \]
    
    \smallskip
    \noindent\textbf{Case 1: $|s|\le \frac12$.}
    Here $v(s)=\frac14-\frac12 s^2$ and $v'(s)=-s$, so that
    \[
      \Delta_{\bar\lambda}(s)
      =\frac{\bar\lambda}{4}
       +\left(\frac12+\bar\lambda\frac{\gamma^2-2}{4}\right)s^2.
    \]
    Since the coefficient of $s^2$ is positive, $\Delta_{\bar\lambda}$ is maximized at $|s|=\frac12$:
    \[
      \sup_{|s|\le 1/2}\Delta_{\bar\lambda}(s)
      =\Delta_{\bar\lambda}\!\left(\frac12\right)
      =\frac{3\gamma^2+10}{16(\gamma^2+4)}.
    \]
    
    \smallskip
    \noindent\textbf{Case 2: $|s|>\frac12$.}
    Here $v(s)=\frac12(|s|-1)^2=\frac12(s^2-2|s|+1)$ and $s v'(s)=s^2-|s|$. Hence
    \[
      \Delta_{\bar\lambda}(s)
      =-a|s|^2+b|s|+\frac{\bar\lambda}{2},
    \]
    where
    \[
      a:=\frac12-\bar\lambda\left(\frac12+\frac{\gamma^2}{4}\right)
        =\frac{\gamma^2+6}{4(\gamma^2+4)},
      \qquad
      b:=\frac12-\bar\lambda
        =\frac{\gamma^2+2}{2(\gamma^2+4)}.
    \]
    The concave quadratic $-ax^2+bx$ attains its maximum at $x_*=\frac{b}{2a}$, and therefore
    \[
      \sup_{|s|>1/2}\Delta_{\bar\lambda}(s)
      \le \frac{b^2}{4a}+\frac{\bar\lambda}{2}
      =\frac{\gamma^4+6\gamma^2+16}{4(\gamma^4+10\gamma^2+24)}
      =A_1(\gamma).
    \]
    Moreover,
    \[
      A_1(\gamma)-\Delta_{\bar\lambda}\!\left(\frac12\right)
      =\frac{(\gamma^2-2)^2}{16(\gamma^2+4)(\gamma^2+6)}\ge 0,
    \]
    so that $A_1(\gamma)$ also dominates the maximum in Case~1. Hence
    $\Delta_{\bar\lambda}(s)\le A_1(\gamma)$ for all $s\in\mathbb R$, i.e.
    \[
      \frac12\,s\,v'(s)
      \ge
      \bar\lambda\left(v(s)+\frac{\gamma^2}{4}s^2\right)-A_1(\gamma),
      \qquad s\in\mathbb R.
    \]
    Applying this inequality with $s=q_i$ for each coordinate $i=1,\dots,d$ and summing over $i$,
    we obtain
    \[
      \frac12 q\cdot\nabla U(q)
      = \sum_{i=1}^d \frac12 q_i v'(q_i)
      \ge \bar\lambda\left(\sum_{i=1}^d v(q_i)+\frac{\gamma^2}{4}\sum_{i=1}^d q_i^2\right) - d\,A_1(\gamma),
    \]
    which yields \eqref{eq:U-drift} with $\lambda=\bar\lambda$ and $A=dA_1(\gamma)$.
    
    \smallskip
    \noindent\textbf{Extension to all smaller $\lambda$.}
    Let $\lambda\in(0,\bar\lambda]$ be arbitrary. Since
    $v(s)+\frac{\gamma^2}{4}s^2\ge 0$ for all $s$, we have pointwise
    $\lambda\left(v(s)+\frac{\gamma^2}{4}s^2\right)
      \le
      \bar\lambda\left(v(s)+\frac{\gamma^2}{4}s^2\right)$,
    and therefore
    \[
      \Delta_{\lambda}(s)
      :=\lambda\left(v(s)+\frac{\gamma^2}{4}s^2\right)-\frac12\,s\,v'(s)
      \le
      \Delta_{\bar\lambda}(s)
      \le A_1(\gamma),
      \qquad s\in\mathbb R.
    \]
    Equivalently,
    \[
      \frac12\,s\,v'(s)
      \ge
      \lambda\left(v(s)+\frac{\gamma^2}{4}s^2\right)-A_1(\gamma),
      \qquad s\in\mathbb R,
    \]
    and summing over coordinates gives \eqref{eq:U-drift} for every
    $\lambda\in(0,1/(4+\gamma^2)]$ with the \emph{same} constant $A=dA_1(\gamma)$.

    \medskip\noindent
    \emph{(c) Asymptotic linear growth of the gradient.}
    Finally, we verify Assumption~\ref{assump:asymptotic-linear-drift}. Let us take $Q_\infty := I_d$.
    Since $|v'(s)-s|\le 1$ for all $s$, we have $|\nabla U(q)-q|\le \sqrt d$.
    For $|q|\ge \sqrt d$, define $\varrho(r):=\sqrt d/r\le 1$. Then
    \[
    |\nabla U(q)-q|\le \sqrt d = \varrho(|q|)\,|q|,\qquad |q|\ge \sqrt d,
    \]
    and clearly $\varrho(r)\to 0$ as $r\to\infty$ (for fixed $d$). This completes the proof.
\end{proof}


\subsection{Proof of Proposition~\ref{prop:MW-explicit-M}}\label{app:MW-explicit-M}
\begin{proof}
We prove (i)--(iii).

\medskip\noindent
\emph{Step 0: matrix form and preliminary constants.}
Write
\[
  a:=\frac{2+\gamma^2}{4\gamma},
  \qquad
  b:=\frac{1}{2\gamma},
  \qquad
  \mathcal M(q,p)=a|q|^2+b|p|^2.
\]
Then $\mathcal M(z)=\frac12 z^\top \mathsf K z$ with $z=(q,p)$ and 
\[
  \mathsf K=\begin{pmatrix}2a\,I_d&0\\0&2b\,I_d\end{pmatrix}
   =\begin{pmatrix}\frac{2+\gamma^2}{2\gamma}I_d&0\\0&\frac{1}{\gamma}I_d\end{pmatrix},
  \qquad
  \|\mathsf K\|_{\mathrm{op}}=\frac{2+\gamma^2}{2\gamma}.
\]
For the multi-well potential we have $U\ge 0$. Using \eqref{eq:V0-general-quadratic} and expanding,
\[
  \mathcal V_0(q,p)
  =U(q)+\frac{\gamma^2}{4}(1-\lambda)|q|^2+\frac12|p|^2+\frac\gamma2\langle q,p\rangle.
\]
Discarding $U(q)$ and writing the remaining quadratic form as
\[
  \frac{\gamma^2}{4}(1-\lambda)|q|^2+\frac12|p|^2+\frac\gamma2\langle q,p\rangle
  =\frac14\,(q,p)\cdot A\,(q,p),
  \qquad
  A:=\begin{pmatrix}\gamma^2(1-\lambda)I_d&\gamma I_d\\ \gamma I_d&2I_d\end{pmatrix},
\]
we obtain \eqref{eq:V0-lower-MW} with $c_1^{\mathrm{MW}}=\frac14\lambda_{\min}(A_1)$ where
$A_1=\left(\begin{matrix}\gamma^2(1-\lambda)&\gamma\\ \gamma&2\end{matrix}\right)$.
The eigenvalues of $A_1$ are explicit; hence \eqref{eq:c1MW-def} follows.

Finally, since $|\mathcal M(q,p)|\le a|q|^2+b|p|^2\le \frac{\|\mathsf K\|_{\rm op}}{2}(|q|^2+|p|^2)$,
combining with \eqref{eq:V0-lower-MW} gives \eqref{eq:M-growth-MW-tilde} with
$\widetilde C_{\mathcal M}^{\mathrm{MW}}=\frac{\|\mathsf K\|_{\rm op}}{2c_1^{\mathrm{MW}}}
=\frac{2+\gamma^2}{4\gamma c_1^{\mathrm{MW}}}$, proving (ii).

\medskip\noindent
\emph{Step 1: proof of (i).}
For this example, we may write $\nabla U(q)=q+r(q)$ with $|r(q)|\le \sqrt d$ for all $q$
(see Proposition~\ref{prop:MW-assumptions} with $Q_\infty=I_d$).
A direct computation gives
\begin{equation}\label{eq:generator-identity}
  \mathcal A_0\mathcal M(q,p)+\mathcal A'\mathcal V_0(q,p)
  = -B|q|^2-|p|^2+\mathcal R_{\mathrm{total}}(q,p),
\end{equation}
where $B=1+\frac{\gamma^2}{2}(1-\lambda)$ and
\[
  \mathcal R_{\mathrm{total}}(q,p)
  = -|r|^2- C_q\langle q,r\rangle - C_p\langle p,r\rangle,
  \qquad
  C_q:=2+\frac{\gamma^2}{2}(1-\lambda),\quad
  C_p:=\frac1\gamma+\frac\gamma2.
\]
We bound the cross terms using $|ab|\le \frac\varepsilon2 a^2+\frac{1}{2\varepsilon}b^2$.
For the $p$-term choose $a=\sqrt{C_{p}}|p|$, $b=\sqrt{C_{p}}|r|$ and $\varepsilon=1/C_p$:
\[
  C_p|\langle p,r\rangle|
  \le \frac12|p|^2+\frac{C_p^2}{2}|r|^2.
\]
For the $q$-term choose $a=\sqrt{C_{q}}|q|$, $b=\sqrt{C_{q}}|r|$ and $\varepsilon=B/C_q$:
\[
  C_q|\langle q,r\rangle|
  \le \frac{B}{2}|q|^2+\frac{C_q^2}{2B}|r|^2.
\]
Dropping the term $-|r|^2\le 0$ and using $|r|^2\le d$ in \eqref{eq:generator-identity} yields
\[
  \mathcal A_0\mathcal M+\mathcal A'\mathcal V_0
  \le -\left(\frac{B}{2}|q|^2+\frac12|p|^2\right)
     +\left(\frac{C_q^2}{2B}+\frac{C_p^2}{2}\right)d.
\]
Define
\[
  Q(q,p):=\frac{B}{2}|q|^2+\frac12|p|^2,
  \qquad
  C(\gamma,\lambda):=\frac{C_q^2}{2B}+\frac{C_p^2}{2}.
\]
Thus
\begin{equation}\label{eq:MW-Q-bound-proof-new}
  \mathcal A_0\mathcal M+\mathcal A'\mathcal V_0
  \le -Q(q,p)+C(\gamma,\lambda)\,d.
\end{equation}

Next we relate $Q$ to $\mathcal V_0$. Using $v(s)\le \frac12 s^2+\frac14$ and separability,
\[
  U(q)=\sum_{i=1}^d v(q_i)\le \frac12|q|^2+\frac d4.
\]
Hence, using again the explicit expansion of $\mathcal V_0$ (see \eqref{eq:V0-general-quadratic})
\[
  \mathcal V_0(q,p)
  \le \widetilde{\mathcal V}_0(q,p)+\frac d4,
  \qquad
  \widetilde{\mathcal V}_0(q,p)
  :=\frac{B}{2}|q|^2+\frac12|p|^2+\frac\gamma2\langle q,p\rangle.
\]
As quadratic forms, $Q$ and $\widetilde{\mathcal V}_0$ decompose into identical $2\times2$ blocks. Thus, it suffices to consider
\[
  Q_{\mathrm{mat}}:=\begin{pmatrix}B/2&0\\0&1/2\end{pmatrix},
  \qquad
  P_{\mathrm{mat}}:=\begin{pmatrix}B/2&\gamma/4\\ \gamma/4&1/2\end{pmatrix}.
\]
Since $\det(P_{\mathrm{mat}})=\frac{4B-\gamma^2}{16}=\frac{4+\gamma^2(1-2\lambda)}{16}>0$
(for $\lambda\le 1/4$), we have $P_{\mathrm{mat}}\succ 0$.
Therefore, the best constant $c$ in $Q\ge c\,\widetilde{\mathcal V}_0$ is the smallest generalized eigenvalue,
namely
\[
  c_{\mathrm{imp}}
  =\inf_{z\neq 0}\frac{z^\top Q_{\mathrm{mat}}z}{z^\top P_{\mathrm{mat}}z}
  =\frac{2\sqrt{B}}{2\sqrt{B}+\gamma}.
\]
Combining this with \eqref{eq:MW-Q-bound-proof-new} yields
\[
  \mathcal A_0\mathcal M+\mathcal A'\mathcal V_0
  \le -c_{\mathrm{imp}}\widetilde{\mathcal V}_0(q,p)+C(\gamma,\lambda)\,d.
\]
Finally, since $\mathcal V_0\le \widetilde{\mathcal V}_0+\frac d4$, we have
$-\widetilde{\mathcal V}_0\le -\mathcal V_0+\frac d4$, and therefore
\[
  \mathcal A_0\mathcal M+\mathcal A'\mathcal V_0
  \le -c_{\mathrm{imp}}\mathcal V_0(q,p)
     +\left(C(\gamma,\lambda)+\frac{c_{\mathrm{imp}}}{4}\right)d.
\]
This proves \eqref{eq:MW-first-order-ineq} with
\[
  C_{\mathrm{imp}}^{(d)}
  :=\left(C(\gamma,\lambda)+\frac{c_{\mathrm{imp}}}{4}\right)d.
\]

\medskip\noindent
\emph{Step 2: proof of (iii).}
By the definition in \eqref{eq:err-def}, 
\[
  \mathrm{Err}^{(d)}(q,p)=|\mathcal A'\mathcal M(q,p)|+|\Delta_q\mathcal M(q,p)|.
\]
Since $\mathcal M(q,p)=a|q|^2+b|p|^2$,
\[
  \Delta_q\mathcal M(q,p)=2ad=\frac{2+\gamma^2}{2\gamma}\,d.
\]
Moreover, $\mathcal A'=-\nabla U(q)\cdot\nabla_q$ and $\nabla_q\mathcal M(q,p)=2a\,q$.
Using $|\nabla U(q)|\le L|q|+|\nabla U(0)|=|q|$ and Young's inequality,
\[
  |\mathcal A'\mathcal M(q,p)|
  =|\nabla U(q)\cdot \nabla_q\mathcal M(q,p)|
  \le |q|\cdot 2a|q|
  =2a|q|^2
  \le \frac{2a}{c_1^{\mathrm{MW}}}\,\mathcal V_0(q,p).
\]
Thus \eqref{eq:ErrMW-bound} holds with
\[
  C_2^{\mathrm{MW}}:=\frac{2a}{c_1^{\mathrm{MW}}}
  =2\widetilde C_{\mathcal M}^{\mathrm{MW}}
  \quad\text{and}\quad
  C_2^{(d),\mathrm{MW}}:=\frac{2+\gamma^2}{2\gamma}\,d.
\]
Finally, with $\underline c_{\mathrm{imp}}=c_{\mathrm{imp}}$ from (i) and
$\widetilde C_{\mathcal M}=\widetilde C_{\mathcal M}^{\mathrm{MW}}$ from (ii),
the updated Proposition~\ref{prop:Valpha-drift} gives the drift-rate expansion
\eqref{eq:lambda-alpha-MW} with $\delta_{\mathrm{MW}}$ and $C_{\lambda,\mathrm{MW}}$ as in
\eqref{eq:deltaMW-ClambdaMW}. This completes the proof.
\end{proof}


\subsection{Proof of Lemma~\ref{lem:MW-feasibility}}\label{app:MW-feasibility}
\begin{proof}
\noindent\emph{Step 1: dissipativity for multi-well.}
For $|s|\le 1/2$, we have $v'(s)=-s$ and $s^2\le 1/4$. Hence
\begin{equation}\label{regime:1}
  s\,v'(s)=-s^2 \ge \lambda s^2 - \frac{1+\lambda}{4}.
\end{equation}

For $|s|\ge 1/2$, we have $v'(s)=s-\mathrm{sign}(s)$. Hence, $s\,v'(s)=s^2-|s|$.
Let $x:=|s|\ge 1/2$. Then for any $\lambda\in(0,1)$,
\[
  x^2-x
  =(1-\lambda)x^2-x+\lambda x^2
  \ge -\frac{1}{4(1-\lambda)}+\lambda x^2,
\]
because $\inf_{x\ge 0}\{(1-\lambda)x^2-x\}=-\frac{1}{4(1-\lambda)}$.
Therefore,
\begin{equation}\label{regime:2}
  s\,v'(s)\ge \lambda s^2-\frac{1}{4(1-\lambda)},\qquad |s|\ge \frac12.
\end{equation}
Combining the two regimes \eqref{regime:1} and \eqref{regime:2} yields the one-dimensional dissipativity bound
\[
  s\,v'(s)\ge \lambda s^2-D_0(\lambda),
  \qquad
  D_0(\lambda):=\max\left\{\frac{1+\lambda}{4},\ \frac{1}{4(1-\lambda)}\right\}.
\]
In particular this holds for every $\lambda\in(0,1/4]$ (with a finite additive constant
$D_0(\lambda)$). For the $d$-dimensional separable potential $U(q)=\sum_{i=1}^d v(q_i)$,
summing over coordinates gives
$\langle q,\nabla U(q)\rangle \ge \lambda|q|^2-d\,D_0(\lambda)$.

\smallskip
\noindent\emph{Step 2: feasibility of $\delta_{\mathrm{MW}}>\gamma\lambda$ for small $\lambda$.}
Recall $\delta_{\mathrm{MW}}=c_{\mathrm{imp}}-\lambda\widetilde C_{\mathcal M}^{\mathrm{MW}}$.
Define
\[
  F(\lambda):=\delta_{\mathrm{MW}}-\gamma\lambda
  =c_{\mathrm{imp}}(\lambda)-\left(\gamma+\widetilde C_{\mathcal M}^{\mathrm{MW}}(\lambda)\right)\lambda,
\]
where $c_{\mathrm{imp}}(\lambda)$ and $\widetilde C_{\mathcal M}^{\mathrm{MW}}(\lambda)$ are
given explicitly in Proposition~\ref{prop:MW-explicit-M}.
Both are continuous in $\lambda\in[0,1/4]$ and finite at $\lambda=0$. Moreover,
\[
  F(0)=c_{\mathrm{imp}}(0)
  =\frac{2\sqrt{1+\gamma^2/2}}{2\sqrt{1+\gamma^2/2}+\gamma}>0.
\]
Hence, by continuity, there exists $\lambda_\star(\gamma)\in(0,1/4]$ such that
$F(\lambda)>0$ for all $\lambda\in(0,\lambda_\star(\gamma)]$, i.e.,
$\delta_{\mathrm{MW}}>\gamma\lambda$.
This completes the proof.
\end{proof}



\subsection{Proof of Theorem~\ref{thm:HFHR-MW-high-dim}}\label{app:HFHR-MW-high-dim}
\begin{proof}
We apply Corollary~\ref{cor:global-acceleration} in dimension $1$ to the multi-well model.
The condition $\delta_{\mathrm{MW}}>\gamma\lambda$ is ensured by the choice of $\lambda$
in Lemma~\ref{lem:MW-feasibility}. The $d$-dimensional statement then follows by tensorization
of the cost $\rho_{\alpha,d}=\sum_{i=1}^d \rho_{\alpha,1}$ and the product structure of $U$.

\emph{Step 1: One-dimensional accelerated contraction with explicit constants.}
Consider first $d=1$. By Proposition~\ref{prop:MW-assumptions}, the one-dimensional potential $v$ satisfies
Assumption~\ref{assump:potential}. Moreover, Proposition~\ref{prop:MW-explicit-M} provides
an explicit quadratic corrector $\mathcal M$ such that the first-order improvement condition
holds with $\underline c_{\mathrm{imp}}=c_{\mathrm{imp}}$, and it also provides explicit
choices of $\widetilde C_{\mathcal M}^{\mathrm{MW}}$ and $C_2^{\mathrm{MW}}$ controlling the perturbation terms.
Consequently, Proposition~\ref{prop:Valpha-drift} applies in dimension $1$ and yields the improved drift
\[
  \lambda_\alpha \;\ge\; \lambda+\delta_{\mathrm{MW}}\alpha-C_{\lambda,\mathrm{MW}}\alpha^2,
  \qquad
  \delta_{\mathrm{MW}}:=c_{\mathrm{imp}}-\lambda\,\widetilde C_{\mathcal M}^{\mathrm{MW}},
  \quad
  C_{\lambda,\mathrm{MW}}:=C_2^{\mathrm{MW}}+\widetilde C_{\mathcal M}^{\mathrm{MW}}\,c_{\mathrm{imp}}.
\]
In particular, if $\delta_{\mathrm{MW}}>0$, then $\lambda_\alpha>\lambda$ for all sufficiently small $\alpha$.
Applying Corollary~\ref{cor:global-acceleration} in dimension $1$, we obtain constants
$\alpha_{\mathrm{MW}}>0$ and $\kappa_{\mathrm{MW}}>0$ (depending only on the one-dimensional model and on $\gamma$) such that,
for all $\alpha\in(0,\alpha_{\mathrm{MW}}]$ and all probability measures $\mu,\nu$ on $\R^{2}$,
\begin{equation}\label{eq:1D-contract-MW}
  \cW_{\rho_{\alpha,1}}\!\left(\mu P_t^{\alpha,(1)},\nu P_t^{\alpha,(1)}\right)
  \le e^{-(c_0+\kappa_{\mathrm{MW}}\alpha)t}\,\cW_{\rho_{\alpha,1}}(\mu,\nu),
  \qquad t\ge0.
\end{equation}
Here $c_0$ is the \emph{one-dimensional} contraction rate at $\alpha=0$.

\medskip\noindent
\emph{Step 2: Tensorization.}
Because $U(q)=\sum_{i=1}^d v(q_i)$ is separable and the driving Brownian motion is coordinate-wise independent,
the $d$-dimensional HFHR dynamics decouples into $d$ independent copies of the one-dimensional HFHR dynamics, and hence
$P_t^{\alpha,(d)}=\bigotimes_{i=1}^d P_t^{\alpha,(1)}$.

Fix any coupling $\pi$ of $\mu$ and $\nu$, and let $(Z,Z')\sim\pi$ with
$Z=(Z_1,\dots,Z_d)$ and $Z'=(Z_1',\dots,Z_d')$, where $Z_i,Z_i'\in\R^2$.
Run, conditionally on $(Z,Z')$, independent (nearly optimal) one-dimensional couplings on each coordinate,
and denote the resulting coupled pair at time $t$ by $(Z_t,Z_t')$.
By additivity of the cost $\rho_{\alpha,d}(z,z')=\sum_{i=1}^d\rho_{\alpha,1}(z_i,z_i')$ and independence,
\[
  \mathbb E\!\left[\rho_{\alpha,d}(Z_t,Z_t')\,|\,Z,Z'\right]
  =\sum_{i=1}^d \mathbb E\!\left[\rho_{\alpha,1}(Z_{t,i},Z_{t,i}')\,|\,Z_i,Z_i'\right].
\]
Applying the one-dimensional contraction \eqref{eq:1D-contract-MW} to each coordinate gives
\[
  \mathbb E\!\left[\rho_{\alpha,d}(Z_t,Z_t')\,|\,Z,Z'\right]
  \le e^{-(c_0+\kappa_{\mathrm{MW}}\alpha)t}\sum_{i=1}^d \rho_{\alpha,1}(Z_i,Z_i')
  = e^{-(c_0+\kappa_{\mathrm{MW}}\alpha)t}\,\rho_{\alpha,d}(Z,Z').
\]
Taking expectation and then infimum over all couplings $\pi$ yields
\[
  \cW_{\rho_{\alpha,d}}\!\left(\mu P_t^{\alpha,(d)},\nu P_t^{\alpha,(d)}\right)
  \le e^{-(c_0+\kappa_{\mathrm{MW}}\alpha)t}\,\cW_{\rho_{\alpha,d}}(\mu,\nu),
  \qquad t\ge0,
\]
for all $\alpha\in(0,\alpha_{\mathrm{MW}}]$.

\medskip\noindent
\emph{Step 3: Dimension independence.}
The constants $\alpha_{\mathrm{MW}}$ and $\kappa_{\mathrm{MW}}$ come entirely from the one-dimensional estimate
\eqref{eq:1D-contract-MW} and therefore do not depend on $d$.
Specifically, we take $\alpha_{\mathrm{MW}} := \alpha_{\mathrm{global}}$ as defined in Corollary~\ref{cor:global-acceleration} for the case $d=1$ (with $L=1$), which is the minimum of the branching and metric acceleration thresholds derived in Section~\ref{sec:acceleration}.
The bound on $c_\alpha^{(d)}$ follows immediately.
\end{proof}


\subsection{Proof of Proposition~\ref{prop:linear-lp-assumptions}}\label{app:linear-lp-assumptions}
\begin{proof}
(a) Since $\varepsilon>0$, each map $q_j\mapsto (q_j^2+\varepsilon^2)^{p/2}$ is smooth on $\mathbb R$; hence $g\in C^\infty$ and therefore $U\in C^\infty$.
Thus Assumption~\ref{assump:potential}(i) holds.
A direct computation gives
\[
\nabla^2 U(q)=\frac{1}{\sigma^2}X^\top X+\nabla^2 g(q),
\qquad
\nabla^2 g(q)=\iota\,\mathrm{diag}\left(\psi''(q_j)\right)_{j=1}^d,
\]
where $\psi(t):=(t^2+\varepsilon^2)^{p/2}$ and
\[
\psi''(t)=p\left(t^2+\varepsilon^2\right)^{\frac p2-2}\left(\varepsilon^2+(p-1)t^2\right)\ge 0.
\]
For the upper bound on $\psi''(t)$, one can use $\varepsilon^2+(p-1)t^2\le \varepsilon^2+t^2$ to obtain
\[
\psi''(t)\le p\left(t^2+\varepsilon^2\right)^{\frac p2-2}(t^2+\varepsilon^2)
= p(t^2+\varepsilon^2)^{\frac p2-1}
\le p(\varepsilon^2)^{\frac p2-1}=p\varepsilon^{p-2}.
\]
Thus $\|\nabla^2 g(q)\|_{\mathrm{op}}\le \ \iota p\varepsilon^{p-2}$, and
\[
\|\nabla^2 U(q)\|_{\mathrm{op}}
\le \frac{\|X^\top X\|_{\mathrm{op}}}{\sigma^2}+\iota p\varepsilon^{p-2}
= \frac{M}{\sigma^2}+\iota p\varepsilon^{p-2},
\]
which yields the claimed global Lipschitz constant for $\nabla U$.
Thus Assumption~\ref{assump:potential}(ii) holds.

(b) We can compute that
\[
\nabla U(q)=\frac{1}{\sigma^2}X^\top X q-\frac{1}{\sigma^2}X^\top y+\nabla g(q),
\qquad
\text{with}\quad\nabla g(q)=\iota p\left(q_j(q_j^2+\varepsilon^2)^{\frac p2-1}\right)_{j=1}^d.
\]
Note that $\langle \nabla g(q),q\rangle=\iota p\sum_{j=1}^d q_j^2\left(q_j^2+\varepsilon^2\right)^{\frac p2-1}\ge 0$.
Using $X^\top X\succeq mI_d$ and Cauchy--Schwarz inequality,
\[
\langle \nabla U(q),q\rangle
\ge \frac{1}{\sigma^2}\left\langle X^\top X q,q\right\rangle-\frac{1}{\sigma^2}\left\langle X^\top y,q\right\rangle
\ge \frac{m}{\sigma^2}|q|^2-\frac{|X^\top y|}{\sigma^2}|q|.
\]
Completing the square gives, for all $q$,
\[
\frac{m}{\sigma^2}|q|^2-\frac{|X^\top y|}{\sigma^2}|q|
\ge \frac{m}{2\sigma^2}|q|^2-\frac{|X^\top y|^2}{2m\sigma^2},
\]
which proves dissipativity.
Thus, Assumption~\ref{assump:potential}(iii) holds.

(c) Set $Q_\infty=\sigma^{-2}X^\top X$. Then
\[
\nabla U(q)-Q_\infty q
= -\frac{1}{\sigma^2}X^\top y+\iota p\,v(q),
\qquad
v_j(q):=q_j(q_j^2+\varepsilon^2)^{\frac p2-1}.
\]
For each coordinate, one has the elementary bound (e.g. split $|q_j|\ge \varepsilon$ and $|q_j|<\varepsilon$):
\[
|v_j(q)| \le |q_j|^{p-1}+\varepsilon^{p-1}.
\]
Hence,
\[
|v(q)|\le \left(\sum\nolimits_{j=1}^d |q_j|^{2p-2}\right)^{1/2}+\sqrt d\,\varepsilon^{p-1}.
\]
Recall the notation for the standard vector $r$-norm: $\|q\|_r := (\sum_{j=1}^d |q_j|^r)^{1/r}$ for $r>0$. Using the norm relation $\|q\|_{r}\le d^{\frac{1}{r}-\frac12}\|q\|_2$ for $0<r<2$ (here we apply it with $r=2p-2$, noting that $1<p<2$ implies $0<2p-2<2$), we have
\[
\left(\sum\nolimits_{j=1}^d |q_j|^{2p-2}\right)^{1/2}
= \|q\|_{2p-2}^{\,p-1}
\le \left(d^{\frac{1}{2p-2}-\frac12}\|q\|_2\right)^{p-1}
= d^{\frac{2-p}{2}}|q|^{p-1}.
\]
Therefore, for all $q$,
\[
|\nabla U(q)-Q_\infty q|
\le \frac{|X^\top y|}{\sigma^2}
+\iota p\left(d^{\frac{2-p}{2}}|q|^{p-1}+\sqrt d\,\varepsilon^{p-1}\right)
= c_0^{\mathrm{LR}}+c_1^{\mathrm{LR}}|q|^{p-1},
\]
with $c_0^{\mathrm{LR}},c_1^{\mathrm{LR}}$ as stated in \eqref{eq:c0lrc1lr-def}.
Dividing by $|q|$ (for $|q|\ge 1$) yields
\[
|\nabla U(q)-Q_\infty q|
\le \left(\frac{c_0^{\mathrm{LR}}}{|q|}+c_1^{\mathrm{LR}}|q|^{p-2}\right)|q|
=\varrho(|q|)|q|.
\]
Since $p-2<0$, both terms $c_0^{\mathrm{LR}}/r$ and $c_1^{\mathrm{LR}} r^{p-2}$ are decreasing in $r$, and $\varrho(r)\to 0$ as $r\to\infty$. This verifies Assumption~\ref{assump:asymptotic-linear-drift}. The proof is complete.
\end{proof}

\subsection{Proof of Proposition~\ref{prop:linear-lp-first-order}}\label{app:linear-lp-first-order}
\begin{proof}
    The results follow directly by applying Lemma~\ref{lem:first-order-improvement} to the Bayesian linear regression model defined in \eqref{eq:linear}, utilizing the explicit properties and bounds derived in Proposition~\ref{prop:linear-lp-assumptions}. Specifically, the spectral bounds, tail moduli, and corrector construction are obtained by substituting the specific forms of $U$ and $\nabla U$ into the general framework.
\end{proof}


\subsection{Proof of Lemma~\ref{lem:LR-feasibility-explicit-full}}\label{app:LR-feasibility-explicit-full}
\begin{proof}
\noindent\emph{Step 1: dissipativity for Bayesian linear regression.}
By Proposition~\ref{prop:linear-lp-assumptions}(b), for all $q\in\mathbb R^d$:
\begin{equation}\label{to:weaken}
\langle \nabla U(q),q\rangle
  \ge \frac{m}{2\sigma^2}|q|^2-\frac{|X^\top y|^2}{2m\sigma^2}.
\end{equation}
  
Fix any $\lambda\in(0,\bar\lambda]$, where $\bar\lambda\le m/(2\sigma^2)$. Weakening the quadratic coefficient in \eqref{to:weaken} gives
\[
  \langle \nabla U(q),q\rangle
  \ge \lambda|q|^2-\frac{|X^\top y|^2}{2m\sigma^2},
\]
which is exactly Assumption~\ref{assump:potential}(iii) (up to an additive constant), proving (i).

\medskip\noindent
\emph{Step 2: an explicit uniform lower bound on $\underline c_{\mathrm{imp}}(\lambda)$.}
Fix $\lambda\in(0,\bar\lambda]$. In Lemma~\ref{lem:first-order-improvement}, the first-order improvement constant
$\underline c_{\mathrm{imp}}(\lambda)$ can be chosen as in \eqref{eq:cimp-lemma}, with
\[
  a_{\min}(\lambda)=\lambda_{\min}(Q_\infty)+\frac{\gamma^2}{2}(1-\lambda),
  \qquad
  a_{\max}(\lambda)=\lambda_{\max}(Q_\infty)+\frac{\gamma^2}{2}(1-\lambda).
\]
Since $\lambda\le\bar\lambda$, we have the deterministic bounds
\[
  a_{\min}(\lambda)\ge m_\infty+\frac{\gamma^2}{2}(1-\bar\lambda)=a_{\min}^{-},
  \qquad
  a_{\max}(\lambda)\le M_\infty+\frac{\gamma^2}{2}=a_{\max}^{+}.
\]
Moreover, by Proposition~\ref{prop:linear-lp-first-order}(ii) we have for all $R'\ge 1$,
\[
  \delta_U(R')
    \le
    \frac{|X^\top y|}{\sigma^2\,R'}
    + \iota\,d^{1-\frac p2}\,(R')^{p-2}
    + \frac{\iota\,d\,\varepsilon^p+\frac{1}{2\sigma^2}|y|^2}{(R')^2}.
\]
In Lemma~\ref{lem:first-order-improvement} the cutoff radius satisfies
$R_0(\lambda)\ge \max\{1,C_{\mathrm{linear}}\}=R$, and since $\delta_U(\cdot)$ is nonincreasing in its argument,
\[
  \delta_U(R_0(\lambda))\le \delta_U(R)\le \delta_U^{+},
\]
where $\delta_U^{+}$ is given by \eqref{eq:LR-deltaU-plus-lemma}. Plugging the three bounds above into
\eqref{eq:cimp-lemma} yields
\[
  \underline c_{\mathrm{imp}}(\lambda)\ge \underline c_{\mathrm{imp}}^{-},
\]
with $\underline c_{\mathrm{imp}}^{-}$ defined in \eqref{eq:LR-cimp-lower-explicit-lemma}.

\medskip\noindent
\emph{Step 3: an explicit uniform upper bound on $\widetilde C_{\mathcal M}^{\mathrm{LR}}(\lambda)$.}
By Lemma~\ref{lem:first-order-improvement}, the corrector can be chosen as a quadratic function:
$\mathcal M(z)=\frac12 z^\top \mathsf K(\lambda) z$, where
\[
  \mathsf K(\lambda)=\int_0^\infty e^{tB^\top}\,C_{B_1}(\lambda)\,e^{tB}\,dt.
\]
Taking operator norms yields
\[
  \|\mathsf K(\lambda)\|_{\mathrm{op}}
  \le
  \left(\int_0^\infty \|e^{tB}\|_{\mathrm{op}}^{\,2}\,dt\right)\,
  \|C_{B_1}(\lambda)\|_{\mathrm{op}}.
\]
Since $C_{B_1}(\lambda)$ is an explicit symmetric matrix depending on $\lambda$ only through the coefficient
$Q_\infty+\frac{\gamma^2}{2}(1-\lambda)I_d$ (see Lemma~\ref{lem:first-order-improvement}),
a crude but explicit bound gives
\[
  \|C_{B_1}(\lambda)\|_{\mathrm{op}}
  \le
  2\left(1+\gamma+\|Q_\infty\|_{\mathrm{op}}+\frac{\gamma^2}{2}\right)
  = C_{B_1}^{+},
\]
where $C_{B_1}^{+}$ is defined in \eqref{eq:LR-CB1-plus-lemma}.

Next we bound $\int_0^\infty \|e^{tB}\|_{\mathrm{op}}^{2}\,dt$. Diagonalize
$Q_\infty=S^\top \mathrm{diag}(\nu_1,\dots,\nu_d)S$ with $S$ orthogonal and $\nu_i\in[m_\infty,M_\infty]$.
Then $B$ is orthogonally similar to a block diagonal matrix with $2\times2$ blocks
$B_{\nu}:=\left(\begin{smallmatrix}0&1\\ -\nu&-\gamma\end{smallmatrix}\right)$.
A direct computation of $e^{tB_\nu}$ (equivalently, the fundamental matrix of $x''+\gamma x'+\nu x=0$)
implies the uniform bound
\[
  \|e^{tB}\|_{\mathrm{op}}
  =\max_{1\le i\le d}\|e^{tB_{\nu_i}}\|_{\mathrm{op}}
  \le C_B e^{-\eta t},
\]
with $\eta$ and $C_B$ as in \eqref{eq:LR-eta-def-lemma}. Consequently,
\[
  \int_0^\infty \|e^{tB}\|_{\mathrm{op}}^{2}\,dt
  \le
  \int_0^\infty C_B^2 e^{-2\eta t}\,dt
  =\frac{C_B^2}{2\eta}.
\]
Substituting the bound 
for the integral and the uniform bound $\|C_{B_1}(\lambda)\|_{\mathrm{op}} \le C_{B_1}^{+}$ into the inequality for $\|\mathsf K(\lambda)\|_{\mathrm{op}}$ yields the uniform estimate
\[
  \|\mathsf K(\lambda)\|_{\mathrm{op}}
  \le \frac{C_B^2}{2\eta}\,C_{B_1}^{+}.
\]
Moreover, since $\lambda\le\bar\lambda$, the quadratic lower bound constant of $\mathcal V_0$ satisfies
$c_1(\gamma,\lambda)\ge c_1(\gamma,\bar\lambda)=c_1$ (because $1-\lambda\ge 1-\bar\lambda$ increases the
$2\times2$ block defining the bound). Therefore,
\[
  \widetilde C_{\mathcal M}^{\mathrm{LR}}(\lambda)
  =\frac{\|\mathsf K(\lambda)\|_{\mathrm{op}}}{2c_1(\gamma,\lambda)}
  \le
  \frac{1}{2c_1}\cdot \frac{C_B^2}{2\eta}\,C_{B_1}^{+}
  =\widetilde C_{\mathcal M}^{+},
\]
where $\widetilde C_{\mathcal M}^{+}$ is defined in \eqref{eq:LR-CB1-plus-lemma}.

\medskip\noindent
\emph{Step 4: conclude $\delta_{\mathrm{LR}}>\gamma\lambda$ on an explicit interval.}
For any $\lambda\in(0,\bar\lambda]$,
\[
  \delta_{\mathrm{LR}}-\gamma\lambda
  =\underline c_{\mathrm{imp}}(\lambda)-\left(\gamma+\widetilde C_{\mathcal M}^{\mathrm{LR}}(\lambda)\right)\lambda
  \ge \underline c_{\mathrm{imp}}^{-}-\left(\gamma+\widetilde C_{\mathcal M}^{+}\right)\lambda.
\]
Thus, if
$\lambda\le \frac{\underline c_{\mathrm{imp}}^{-}}{\gamma+\widetilde C_{\mathcal M}^{+}}$,
then $\delta_{\mathrm{LR}}>\gamma\lambda$. Combining with $\lambda\le\bar\lambda$ gives the explicit
choice \eqref{eq:LR-lambda-star-explicit-lemma}, proving (ii). The proof is complete.
\end{proof}

\subsection{Proof of Theorem~\ref{thm:HFHR-LR}}
\label{proof:thm:HFHR-LR}

\begin{proof}
The result follows directly from Corollary~\ref{cor:global-acceleration}.
Proposition~\ref{prop:linear-lp-assumptions} establishes that Assumptions~\ref{assump:potential} and \ref{assump:asymptotic-linear-drift} hold.
Proposition~\ref{prop:linear-lp-first-order} provides the explicit construction of the quadratic corrector $\mathcal M$ and establishes the first-order drift improvement with explicit constants $\delta_{\mathrm{LR}}$ and $C_{\lambda,\mathrm{LR}}$.
Lemma~\ref{lem:LR-feasibility-explicit-full} guarantees that by choosing $\lambda \le \lambda_\star(\gamma)$, the acceleration condition $\delta_{\mathrm{LR}} > \gamma\lambda$ is satisfied.
Therefore, all conditions of Corollary~\ref{cor:global-acceleration} are met, implying the existence of the acceleration constants $\alpha_{\mathrm{LR}}$ and $\kappa_{\mathrm{LR}}$. The proof is complete.
\end{proof}


\subsection{Proof of Proposition~\ref{prop:classification-assumptions}}\label{app:classification-assumptions}
\begin{proof}
(a) Since $h,\varphi\in C^2$ and $q\mapsto \langle q,x_i\rangle$ is linear, each summand
$q\mapsto \varphi\!\left(y_i-h(\langle q,x_i\rangle)\right)$ is $C^2$, and hence so is $U$.
Moreover $\varphi\ge 0$ and $\frac{\iota}{2}|q|^2\ge 0$ imply $U\ge 0$. Thus, Assumption~\ref{assump:potential}(i) holds.

By differentiating $U$, we can compute that
\[
\nabla U(q)=-\frac1n\sum_{i=1}^n \varphi'\!\left(y_i-h(\langle q,x_i\rangle)\right)\,h'(\langle q,x_i\rangle)\,x_i+\iota q.
\]
Let $s_i(q):=\langle q,x_i\rangle$ and $t_i(q):=y_i-h(s_i(q))$.
Differentiating $\nabla U$ yields
\[
\nabla^2 U(q)=\iota I_d+\frac1n\sum_{i=1}^n a_i(q)\,x_i x_i^\top,
\]
with
\[
a_i(q)=\varphi''(t_i(q))\,\left(h'(s_i(q))\right)^2-\varphi'(t_i(q))\,h''(s_i(q)).
\]
Using the uniform bounds on $|\varphi'|$, $|\varphi''|$, $|h'|$ and $|h''|$, we get for all $q$,
\[
|a_i(q)|
\le \Phi_2 H_1^2+\Phi_1 H_2.
\]
Therefore,
\[
\|\nabla^2 U(q)\|_{\mathrm{op}}
\le \iota + \frac1n\sum_{i=1}^n |a_i(q)|\,\|x_i x_i^\top\|_{\mathrm{op}}
\le \iota + (\Phi_2 H_1^2+\Phi_1 H_2)\,B_x^2,
\]
since $\|x_i x_i^\top\|_{\mathrm{op}}=|x_i|^2\le B_x^2$. Thus, Assumption~\ref{assump:potential}(ii) holds.

(b) From the gradient expression and Cauchy--Schwarz inequality,
\[
|\nabla U(q)-\iota q|
\le \frac1n\sum_{i=1}^n |\varphi'(t_i(q))|\,|h'(s_i(q))|\,|x_i|
\le \Phi_1 H_1 B_x=:C_0.
\]
Hence
\[
\langle \nabla U(q),q\rangle
= \iota|q|^2 + \langle \nabla U(q)-\iota q,\,q\rangle
\ge \iota|q|^2 - C_0|q|
\ge \frac{\iota}{2}|q|^2 - \frac{C_0^2}{2\iota},
\]
where we completed the square to get
the last inequality. This gives the desired dissipativity inequality.
Thus, Assumption~\ref{assump:potential}(iii) holds.

(c) With $Q_\infty=\iota I_d$ and the bound $|\nabla U(q)-\iota q|\le C_0$,
for $|q|\ge 1$ we have
\[
|\nabla U(q)-Q_\infty q|
\le C_0
= \frac{C_0}{|q|}\,|q|
= \varrho(|q|)\,|q|.
\]
Since $\varrho(r)=C_0/r$ is decreasing and vanishes at infinity, Assumption~\ref{assump:asymptotic-linear-drift} holds. The proof is complete.
\end{proof}

\subsection{Proof of Proposition~\ref{prop:classification-first-order}}\label{app:classification-first-order}

\begin{proof}  
    The results follow directly by applying Lemma~\ref{lem:first-order-improvement} to the Bayesian binary classification model defined in \eqref{eq:classification}. The spectral bounds, tail moduli, and corrector construction are obtained by substituting the specific potential properties derived in Proposition~\ref{prop:classification-assumptions} into the general framework.
\end{proof}


\subsection{Proof of Lemma~\ref{lem:BC-feasibility-explicit-full}}\label{app:BC-feasibility-explicit-full}

\begin{proof}
\noindent\emph{Step 1: dissipativity with an arbitrary $\lambda\le \iota/2$.}
By Proposition~\ref{prop:classification-assumptions}(b) (where the ridge coefficient is denoted by $\iota$), for all $q\in\R^d$,
$\langle \nabla U(q),q\rangle
  \ge \frac{\iota}{2}|q|^2-\frac{C_0^2}{2\iota}$.

Fix any $\lambda\in(0,\bar\lambda]$; since $\bar\lambda\le\iota/2$, we obtain
\[
  \langle \nabla U(q),q\rangle
  \ge \lambda|q|^2-\frac{C_0^2}{2\iota},
\]
which is Assumption~\ref{assump:potential}(iii) (up to an additive constant).

\medskip\noindent
\emph{Step 2: a uniform lower bound on $\underline c_{\mathrm{imp}}(\lambda)$.}
For $\lambda\in(0,\bar\lambda]$, in Lemma~\ref{lem:first-order-improvement} we have
$a_{\min}=a_{\max}=a(\lambda)=\iota+\frac{\gamma^2}{2}(1-\lambda)\ge a^{-}$.
Moreover, the cutoff radius satisfies $R_0(\lambda)\ge R$, and by \eqref{eq:BC-deltaU} and monotonicity,
$\delta_U(R_0(\lambda))\le \delta_U(R)\le \delta_U^{+}$.
Plugging these bounds into \eqref{eq:cimp-lemma} yields
\[
  \underline c_{\mathrm{imp}}(\lambda)\ge \underline c_{\mathrm{imp}}^{-},
\]
with $\underline c_{\mathrm{imp}}^{-}$ defined in \eqref{eq:BC-cimp-minus} (using $a^-$ which depends on $\iota$).

\medskip\noindent
\emph{Step 3: a uniform upper bound on $\widetilde C_{\mathcal M}^{\mathrm{BC}}(\lambda)$.}
Lemma~\ref{lem:first-order-improvement} provides a quadratic corrector
$\mathcal M(z)=\frac12 z^\top \mathsf K(\lambda)z$ with
$\mathsf K(\lambda)=\int_0^\infty e^{tB^\top}C_{B_1}(\lambda)e^{tB}\,dt$.
Taking operator norms,
\begin{equation}\label{follows:1}
  \|\mathsf K(\lambda)\|_{\mathrm{op}}
  \le
  \left(\int_0^\infty \|e^{tB}\|_{\mathrm{op}}^{\,2}\,dt\right)\,
  \|C_{B_1}(\lambda)\|_{\mathrm{op}}.
\end{equation}
Since $Q_\infty=\iota I_d$, one has the crude bound
$\|C_{B_1}(\lambda)\|_{\mathrm{op}}\le C_{B_1}^{+}$ with $C_{B_1}^{+}$ as defined in the lemma statement (using $\iota$). Moreover, the block ODE
representation of $e^{tB}$ (equivalently the damped oscillator $x''+\gamma x'+\iota x=0$) implies
$\|e^{tB}\|_{\mathrm{op}}\le C_B e^{-\eta t}$ with $\eta,C_B$ defined using $\iota$. Hence
\begin{equation}\label{follows:2}
  \int_0^\infty \|e^{tB}\|_{\mathrm{op}}^{\,2}\,dt\le \frac{C_B^2}{2\eta}.
\end{equation}
Therefore, it follows from \eqref{follows:1} and \eqref{follows:2} that
\begin{equation}\label{follows:3}
  \|\mathsf K(\lambda)\|_{\mathrm{op}}
  \le \frac{C_B^2}{2\eta}\,C_{B_1}^{+}.
\end{equation}
Finally, for $\lambda\le\bar\lambda$, the baseline quadratic lower bound constant satisfies
$c_1(\gamma,\lambda)\ge c_1(\gamma,\bar\lambda)=c_1$. Hence, it follows from \eqref{follows:3} that
\[
  \widetilde C_{\mathcal M}^{\mathrm{BC}}(\lambda)
  =\frac{\|\mathsf K(\lambda)\|_{\mathrm{op}}}{2c_1(\gamma,\lambda)}
  \le
  \frac{1}{2c_1}\cdot \frac{C_B^2}{2\eta}\,C_{B_1}^{+}
  =\widetilde C_{\mathcal M}^{+}.
\]

\medskip\noindent
\emph{Step 4: conclude $\delta_{\mathrm{BC}}>\gamma\lambda$ on an explicit interval.}
For $\lambda\in(0,\bar\lambda]$,
\[
  \delta_{\mathrm{BC}}-\gamma\lambda
  =\underline c_{\mathrm{imp}}(\lambda)-\left(\gamma+\widetilde C_{\mathcal M}^{\mathrm{BC}}(\lambda)\right)\lambda
  \ge \underline c_{\mathrm{imp}}^{-}-\left(\gamma+\widetilde C_{\mathcal M}^{+}\right)\lambda.
\]
Hence $\delta_{\mathrm{BC}}>\gamma\lambda$ whenever
$\lambda\le \underline c_{\mathrm{imp}}^{-}/(\gamma+\widetilde C_{\mathcal M}^{+})$.
Combining with $\lambda\le\bar\lambda$ yields exactly \eqref{eq:BC-lambda-star}.
This completes the proof.
\end{proof}

\subsection{Proof of Theorem~\ref{thm:HFHR-BC}}\label{proof:thm:HFHR-BC}

\begin{proof}
The result follows directly from Corollary~\ref{cor:global-acceleration}.
Proposition~\ref{prop:classification-assumptions} establishes that Assumptions~\ref{assump:potential} and \ref{assump:asymptotic-linear-drift} hold.
Proposition~\ref{prop:classification-first-order} provides the explicit construction of the quadratic corrector $\mathcal M$ and establishes the first-order drift improvement with explicit constants $\delta_{\mathrm{BC}}$ and $C_{\lambda,\mathrm{BC}}$.
Lemma~\ref{lem:BC-feasibility-explicit-full} guarantees that by choosing $\lambda \le \lambda_\star(\gamma)$, the acceleration condition $\delta_{\mathrm{BC}} > \gamma\lambda$ is satisfied.
Therefore, all conditions of Corollary~\ref{cor:global-acceleration} are met, implying the existence of the acceleration constants $\alpha_{\mathrm{BC}}$ and $\kappa_{\mathrm{BC}}$. The proof is complete.
\end{proof}

\end{document}